\documentclass[12pt,a4paper]{amsart}
\makeatletter
\renewcommand\normalsize{%
    \@setfontsize\normalsize{11.7}{14pt plus .3pt minus .3pt}%
    \abovedisplayskip 10\p@ \@plus4\p@ \@minus4\p@
    \abovedisplayshortskip 6\p@ \@plus2\p@
    \belowdisplayshortskip 6\p@ \@plus2\p@
    \belowdisplayskip \abovedisplayskip}
\renewcommand\small{%
    \@setfontsize\small{9.5}{12\p@ plus .2\p@ minus .2\p@}%
    \abovedisplayskip 8.5\p@ \@plus4\p@ \@minus1\p@
    \belowdisplayskip \abovedisplayskip
    \abovedisplayshortskip \abovedisplayskip
    \belowdisplayshortskip \abovedisplayskip}
\renewcommand\footnotesize{%
    \@setfontsize\footnotesize{8.5}{9.25\p@ plus .1pt minus .1pt}
    \abovedisplayskip 6\p@ \@plus4\p@ \@minus1\p@
    \belowdisplayskip \abovedisplayskip
    \abovedisplayshortskip \abovedisplayskip
    \belowdisplayshortskip \abovedisplayskip}
\ifdefined\pdfpagewidth
\else
\fi
\calclayout
\makeatother

\usepackage{amsfonts}
\usepackage{amssymb}
\usepackage{graphicx}
\usepackage{amsmath}
\usepackage{latexsym}
\usepackage{amscd}
\usepackage{xypic}
\usepackage{mathrsfs}
\usepackage{enumitem} 
\usepackage{braket}
\usepackage{comment}
\usepackage{bm}
\usepackage{color}
\usepackage{stmaryrd}
\usepackage{marginnote}

\usepackage{hyperref}

\allowdisplaybreaks 

\numberwithin{equation}{section}
\newtheorem{theorem}{Theorem}[section]
\newtheorem{lemma}[theorem]{Lemma}
\newtheorem{corollary}[theorem]{Corollary}

\newtheorem{proposition}[theorem]{Proposition}

\theoremstyle{remark}
\newtheorem{remark}[theorem]{Remark}
\newtheorem{claim}[theorem]{Claim}

\newtheorem*{theorem*}{Theorem}

\theoremstyle{definition}
\newtheorem{definition}[theorem]{Definition}

\theoremstyle{definition}

\newcommand{\II}{\mathbb{I}}

\newcommand{\MM}{\mathbb{M}}
\newcommand{\NN}{\mathbb{N}}

\newcommand{\RR}{\mathbb{R}}

\newcommand{\ZZ}{\mathbb{Z}}

\newcommand{\cC}{\mathcal C}
\renewcommand{\cD}{\mathcal D}

\newcommand{\cF}{\mathcal F}
\newcommand{\cG}{\mathcal G}
\renewcommand{\cH}{\mathcal H}

\newcommand{\cK}{\mathcal K}

\newcommand{\cP}{\mathcal P}

\renewcommand{\cR}{\mathcal R}
\newcommand{\cS}{\mathcal S}
\newcommand{\cT}{\mathcal T}

\newcommand{\bx}{\mathbf{x}}
\newcommand{\by}{\mathbf{y}}

\newcommand{\bOh}{\mathbf{0}}

\newcommand{\eps}{\varepsilon}

\DeclareMathOperator{\sing}{sing}
\DeclareMathOperator{\reg}{reg}

\DeclareMathOperator{\spt}{spt}
\DeclareMathOperator{\spine}{spine}

\DeclareMathOperator{\dist}{dist}
\DeclareMathOperator{\Graph}{graph}
\DeclareMathOperator{\Dil}{Dil}

\newcommand{\mres}{\lfloor}

\title[Generic regularity for min.\ hyp.\ in dimensions $9$ and $10$]{Generic regularity for minimizing hypersurfaces in dimensions 9 and 10}

\author{Otis Chodosh} 
\address{OC: Department of Mathematics, Bldg.\ 380, Stanford University, Stanford, CA 94305, USA}
\email{ochodosh@stanford.edu}
\author{Christos Mantoulidis} 
\address{CM: Department of Mathematics, Rice University, Houston, TX 77005, USA}
\email{christos.mantoulidis@rice.edu}
\author{Felix Schulze}
\address{FS: Department of Mathematics, Zeeman Building, University of Warwick, Gibbet Hill Road, Coventry CV4 7AL, UK}
\email{felix.schulze@warwick.ac.uk} 

\begin{document}

\begin{abstract}
We prove that singularities of area minimizing hypersurfaces can be perturbed away in ambient dimensions $9$ and $10$.
\end{abstract}

\maketitle

\section{Introduction}

Let $\Gamma$ be a smooth, closed (compact and boundaryless), oriented, $(n-1)$-dimensional submanifold of $\RR^{n+1}$. Among all smooth, compact, oriented hypersurfaces $M \subset \RR^{n+1}$ with $\partial M = \Gamma$, does there exist one with \textit{least area}? 

Foundational results in geometric measure theory (see \cite{Federer:GMT, MassariMiranda,Giusti,Maggi:finite-per}) can be used to produce an integral $n$-current $T$ with least mass among all those with boundary equal to the multiplicity-one current represented by $\Gamma$. When $n+1 \leq 7$, it is known that $T$ is supported on a smooth area-minimizing hypersurface that solves the original differential geometric problem (see \cite{Fleming:plateau, DeGiorgi:bernstein, Almgren:regularity, Simons:minvar, Hardt-Simon:boundary-regularity}). When $n+1 \geq 8$, smooth minimizers can fail to exist  (see \cite{BDG:Simons}) but it is nevertheless known that away from a closed set $\sing T \subset \RR^{n+1}\setminus\Gamma$ of Hausdorff dimension $\leq n-7$, the support of $T$ will be a smooth hypersurface with boundary $\Gamma$ (see \cite{Federer:dimension-reduction, Hardt-Simon:boundary-regularity}).

A fundamental result of Hardt--Simon \cite{Hardt-Simon:isolated-singularities} shows that the singularities (necessarily isolated points) of $7$-dimensional minimizing currents in $\RR^8$ can be \emph{eliminated} by a slight perturbation of the boundary, $\Gamma$, thus yielding solutions to the original differential geometric problem. It has been a longstanding conjecture that similar results hold in higher dimensions (see \cite[Problem 108]{Yau:problems}, \cite[Problem 5.16]{Brothers:open-problems}, \cite[Conjecture 16]{Gromov:101}), where the singular set can have positive Hausdorff dimension and can be a priori quite complicated.\footnote{See \cite{Simon:stable.sing} for complicated singular sets of stable (rather than minimizing) hypersurfaces.} In this paper we use a perturbative approach motivated from our work on generic regularity in mean curvature flow\footnote{In our joint work with Choi \cite{CCMS:low-ent-gen}(which was a follow-up to \cite{CCMS:generic1}) we discovered a strong analogy between mean curvature flow of generic initial data and the generic regularity and uniqueness for area-minimzers. See also the earlier foundational work of Colding--Minicozzi on generic mean curvature flow \cite{ColdingMinicozzi:generic}.} to obtain a generic regularity result for minimizers in ambient dimensions $9$ and $10$. 

Our main theorem is:

\begin{theorem}[Generic regularity for the Plateau problem]\label{theo:main.plateau}
Let $n+1 \in \{8,9,10\}$. Consider a smooth, closed, oriented, $(n-1)$-dimensional submanifold $\Gamma \subset \RR^{n+1}$. There exist $C^\infty$-small  perturbations $\Gamma'$ of $\Gamma$ (as graphs in the normal bundle of $\Gamma$) with the property that there exists a least-area smooth, compact, oriented hypersurface $M' \subset \RR^{n+1}$ with $\partial M' = \Gamma'$.
\end{theorem}

See Theorem \ref{theo:main.plateau.baire} for our result in the language of integral $n$-currents that implies the geometric theorem above. The case $n+1=8$ in the special setting of minimizing boundaries is due to Hardt--Simon \cite[Section 5]{Hardt-Simon:isolated-singularities}.\footnote{The setting of boundaries applies, e.g., when $\Gamma$ is connected; see \cite[Corollary 11.2]{Hardt-Simon:boundary-regularity}.} These results can be extended in a straightforward manner (which we omit) to cover area minimization in smooth, closed, oriented $(n+1)$-dimensional Riemannian manifolds (in place of $\RR^{n+1}$) provided $\Gamma$ is homologically trivial. 

The work of Hardt--Simon on $7$-dimensional minimizing currents in $\RR^8$ was then used by Smale \cite{Smale:generic} to show that given an $8$-dimensional smooth closed oriented Riemannian manifold $(N,g)$ and a nonzero homology class $[\alpha] \in H_7(N^{8},\ZZ)$, the area-minimizer in $[\alpha]$ can be taken to be smooth as long as we perturb the metric appropriately.\footnote{See \cite{MazzeoSmale} for some related work.} We extend this result to ambient $9$ and $10$ dimensional manifolds, too:

\begin{theorem}[Generic regularity for homology minimizers]\label{theo:main.homology}
Let $n+1 \in \{8,9,10\}$. Consider a closed, oriented, $(n+1)$-dimensional Riemannian manifold $(N, g)$. There exist $C^\infty$-small perturbations $g'$ of $g$ with the property that for every nonzero $[\alpha] \in H_n(N, \ZZ)$ there exist disjoint, smooth, closed, oriented hypersurfaces $M_1, \ldots, M_Q \subset N$ and $k_1, \ldots, k_Q \in \ZZ$ so that $\sum_{i=1}^Q k_i \llbracket M_i\rrbracket$ is of least area, with respect to $g'$, in $[\alpha]$.
\end{theorem}

See Theorem \ref{theo:main.homology.baire} for the result in the language of integral currents that implies the geometric theorem above.

\begin{remark}
For 2 dimensional integral currents or flat chains mod 2 that minimize area in their homology class (in any codimension), generic regularity has been established by White \cite{White:generic2dim}. In recent work, Li--Wang \cite{LiWang:generic8dim} have proven generic regularity for locally stable (as opposed to minimizing) minimal hypersurfaces in $8$-dimensional manifolds (cf.\ \cite{Wang:deformation,CLS,Edelen:degen}). 
\end{remark}

\begin{remark}
Theorem \ref{theo:main.homology} allows one to extend Schoen--Yau's stable minimal hypersurface obstruction to positive scalar curvature to $9$ and $10$ ambient dimensions (positivity of scalar curvature is an open condition). For example, by combining Theorem \ref{theo:main.homology} with the Schoen--Yau inductive descent method \cite{SY:descent}, we may conclude that the connected sum\footnote{More generally, an SYS manifold of dimension $\leq 10$ does not admit positive scalar curvauture. See \cite[\S 2.7]{gromov2019lectures} for the definition and further discussion.} $T^{n+1}\#M^{n+1}$ with $M$ closed does not admit positive scalar curvature for $n+1\leq 10$ (a well-known reduction of Lohkamp shows that this implies the positive mass theorem in these dimensions). See also the work of Schoen--Yau and Lohkamp \cite{SY:sing,Lohkamp:PSC}. 
\end{remark}

Here are the main points of our strategy to prove Theorem \ref{theo:main.plateau}:
\begin{enumerate}
	\item[(i)] \textbf{Generic regularity via foliations}. Let $T$ be any minimizer with prescribed boundary $\Gamma$. We construct a 1-parameter family $\{ T_t \}$ of pairwise disjoint minimizers which form a local foliation on either side of $T_0 = T$.\footnote{For technical reasons (non-uniqueness of minimizers) our foliations may have gaps.} We will show that there exist $t$ arbitrarily close to $0$ for which $T_t$ is smooth.
		
	\item[(ii)] \textbf{Top density as improving quantity}. One point in common with \cite{CCMS:generic1} is that our generic regularity theorem relies on an \textit{iterative improvement} under perturbations. Specifically, we show the existence of a dimensional constant $\eta > 0$ (in $8$, $9$, or $10$ ambient dimensions), such that for any $T_{t_0}$ there exists a nearby $T_t$ (with $t$ arbitrarily close to $t_0$) satisfying\footnote{Recall that $\sup \emptyset = -\infty$.}
\begin{equation} \label{eq:intro.iteration}
	\sup_{\bx \in \sing T_t} \Theta_{T_t}(\bx) \leq \sup_{\bx \in \sing T_{t_0}} \Theta_{T_{t_0}}(\bx) - \eta,
\end{equation}
where $\Theta_T$ denotes the $n$-dimensional area density of $T$; see Section \ref{sec:preliminaries}. On the other hand, we know from Allard's regularity theorem \cite{Allard:first-variation} that there exists another dimensional constant $\Theta_n^* > 1$ with the property that
\begin{equation} \label{eq:intro.allard}
	\Theta_T(\bx) \geq \Theta_n^*
\end{equation}
at any $\bx \in \sing T$. In particular, \eqref{eq:intro.allard} will gurantee that finitely many iterations of \eqref{eq:intro.iteration} will suffice to arrive at a smooth minimizer. 

The proof of \eqref{eq:intro.iteration} is technically very different here than in our previous work \cite{CCMS:generic1} and relies on two new ingredients, (iii) and (iv) below.

	\item[(iii)] \textbf{Nearly top-density points align along spines}. In Proposition \ref{prop:dens-drop-sing} we prove the infinitesimal analog of \eqref{eq:intro.iteration} with a strict inequality and $\eta=0$ whenever $T_{t_0}$ is a nonflat minimizing cone\footnote{These are the blowup models of singular minimizers.} We then use this to prove in Lemma \ref{lemm:cone-split-pre} that in any nearly conical scale of $T_{t_0}$, all nearly top-density singular points of $T_{t_0}$ as well as all those of nearby $T_t$ must be near the approximate spine of $T_{t_0}$. Since we know from  \cite{Simons:minvar} that singular codimension-one minimizing cones have at most $(n-7)$-dimensional spines, it follows that the set of nearly top-density singular points across all $T_t$'s is, coarsely, at most $(n-7)$-dimensional. 
	
	\item[(iv)] \textbf{Regular part separation}. We then exploit Simon's \cite{Simon:asymptotic-decay} partial Harnack theory for positive Jacobi fields on minimizing cones, with recent modifications by Wang \cite{Wang:smoothing}, to deduce quantitative separation estimates in Lemma \ref{lemm:sep-est} for the $T_t$ along (suitable subsets of) their regular part. This translates into a superlinear H\"older continuity estimate for the ``time $t$'' parameter corresponding to top-density singular points $\bx \in \sing T_t$ (see Claim \ref{clai:packing.superlinear}, and substitute $\lambda_0 > 2 + \varepsilon$).
	
	Combined, (iii) and (iv) imply a packing estimate on the set 
\[ \{ t : T_t \text{ contains a nearly top-density singular point} \} \subset \RR, \]
\textit{roughly} proving that it can be at most $\tfrac{n-7}{2+\varepsilon}$-dimensional with $\varepsilon \in (0, 1)$ a dimensional constant that comes from the analysis of positive Jacobi fields on nonflat minimizing cones. In our setting, we have a 1-parameter family of perturbations $T_t$ to work with, so we can guarantee \eqref{eq:intro.iteration} for $t$ arbitrarily close to $t_0$ as long as our ambient dimension $n+1$ is such that
\[ \frac{n-7}{2+\varepsilon} < 1 \iff n < 9+\varepsilon, \]
i.e., $n+1 = 8, 9, 10$.
\end{enumerate}

For the homology problem adaptation required to prove Theorem \ref{theo:main.homology}, we need to perturb the ambient metric (rather than a prescribed boundary) in a way that guarantees we get a 1-parameter family $\{ T_t \}$ of pairwise disjoint minimizers which form a local foliation on either side of $T_0 = T$, our initial minimizer. Unlike in the sample strategy described above, for the Plateau problem, here we could not produce 1-parameter families of a \emph{continuous} parameter $t$, but rather arbitrarily fine discrete families, where the ``fineness'' must depend on the central leaf $T_0$. This means two things for the proof. First, we need to recalibrate our family $\{ T_t \}$ at each iteration in Step (ii), rather than keep using the same family $\{ T_t \}$ throughout as one could in $\RR^{n+1}$. Second, in order to avoid presenting two parallel proofs, we also adopted this ``discrete'' approach for the original Plateau problem in $\RR^{n+1}$, too.

\begin{remark}[Foliations]
	The use of foliations to derive generic regularity-type theorems for variational problems goes back at least thirty years. Foliations were used by Hardt--Simon \cite{Hardt-Simon:isolated-singularities} to show the generic regularity of minimizers in $\RR^8$. The groundbreaking work of Evans--Spruck on level set flows (\cite{Evans-Spruck:level-set-flow-i, Evans-Spruck:level-set-flow-ii, Evans-Spruck:level-set-flow-iii, Evans-Spruck:level-set-flow-iv}) that generic mean curvature flows are ``nonfattening'' also used foliations; see also \cite{Ilmanen:elliptic-regularization}. It would be interesting to better understand bigger picture implications of the ability to use monotone foliations across geometric PDE problems.  
\end{remark}

\begin{remark}[Obstacle problem]
 It is worth highlighting an interesting connection to the recent generic regularity results for free boundaries in the obstacle problem \cite{FROS:generic} (cf.\ \cite{Monneau,CSV:obst,FS:obst-fine}) and shortly after the Signorini problem \cite{FernandezRealRosOton, FernandezRealTorresLatorre}. That work, too, relies on a subtle derivation of superlinear H\"older-continuity estimate on a ``time $t$'' function for a foliation to prove the smallness of a spacetime singular set across all time parameters $t$. While the obstacle problem is technically very different from the area minimization problem and thus our techniques are altogether quite different, we find this connection very interesting and worth exploring. 
\end{remark}

\subsection{Organization} In Section \ref{sec:preliminaries} we recall some preliminaries and necessary notation from geometric measure theory regarding minimizing integral currents. In Section \ref{sec:minimizing.foliations} we prove our main technical tools regarding the behavior of monotone families (foliations possibly with gaps) consisting of minimizers. We then prove Theorem \ref{theo:main.plateau} in Section \ref{sec:main.plateau} and Theorem \ref{theo:main.homology} in Section \ref{sec:main.homology}. Finally, in Appendices \ref{app:hardt.simon.boundary} and \ref{app:positive.jacobi.fields.cones} we collect some background technical facts regarding the Hardt--Simon boundary regularity of minimizing integral currents and results from Simon's and Wang's analysis of positive Jacobi fields on nonflat minimizing cones.

\subsection{Acknowledgments} We are grateful to Kyeongsu Choi, Nick Edelen, Eugenia Malinnikova, Luca Spolaor, and Brian White for helpful conversations related to this paper, and to Zhihan Wang for pointing out an issue with the statement of Theorem \ref{theo:high.density.packing} in an earlier version of the paper. We are also grateful to Alessio Figalli, Xavier Ros-Oton, Joachim Serra, and Leon Simon for their interest and kind encouragement. O.C. was supported by NSF grants (DMS-2016403 and DMS-2304432), a Terman Fellowship, and a Sloan Fellowship. C.M. was supported by an NSF grant (DMS-2147521).

\section{Preliminaries and notation} \label{sec:preliminaries}

\subsection{Integral currents, regular, singular parts, $\II_n$, $\II_n^1$}

Let $(\breve N, g)$ be an $(n+1)$-dimensional Riemannian manifold without boundary. 

We adopt the standard notation from geometric measure theory that $\II_k(\breve N)$ is the space of integral $k$-currents in $\breve N$, with $k \in \{ 0, 1, \ldots, n, n+1 \}$. This is a weak class of objects pioneered by Federer and Fleming (\cite{FedererFleming}) that is well-suited to homological area-minimization and the Plateau problem in high dimensions. We refer the reader to the standard references \cite{Federer:GMT,Simon:GMT}. We will be interested in elements of $\II_n(\breve N)$ that are \textit{homologically mass-minimizing} in the sense of \cite[5.1.6]{Federer:GMT}, and we'll just say \textit{minimizing} for brevity. 

Recall that for every $T \in \II_k(\breve N)$ with $k \in \{ 1, \ldots, n, n+1 \}$ we have $\partial T \in \II_{k-1}(\breve N)$ and there is a well-defined partition of the support of $T$ away from that of $\partial T$, i.e.,
\[ \spt T \setminus \spt \partial T = \reg T \cup \sing T, \]
where
\begin{align*}
	\reg T = \{ \bx \in \spt T \setminus \spt \partial T : \; 
		& \spt T \cap B_r(\bx) \text{ is a smooth $k$-dimensional submanifold} \\
		& \text{ for sufficiently small } r > 0 \}
\end{align*}
denotes the set of ``interior regular points'' and
\[ \sing T = \spt T \setminus (\spt \partial T \cup \reg T) \]
denotes the set of all remaining ``interior singular points.'' When $T \in \II_n(\breve N)$ is minimizing, it is a celebrated result of Simons \cite{Simons:minvar} and Federer \cite{Federer:dimension-reduction} that $\sing T$ has Hausdorff dimension $\leq n-7$.

The regularity at the boundary is more delicate and, for this and other reasons, it will be convenient in our proof to introduce the following notation:

\begin{definition}
	Let $\II_k^1(N) \subset \II_k(N)$ denote the space of \emph{multiplicity-one} integral $k$-currents; that is, those with density function equal to $1$ $\cH^k$-a.e. on their support.
\end{definition}

It is a fundamental result of Hardt--Simon \cite{Hardt-Simon:boundary-regularity} that when $T \in \II_n^1(\breve N)$ is minimizing and $\partial T = \llbracket \Gamma \rrbracket \in \II_{n-1}^1(\breve N)$ for a smooth, closed, oriented, $(n-1)$-dimensional submanifold $\Gamma$, then for every $\bx \in \spt \partial T$ and sufficiently small $r > 0$, $\spt T \cap B_r(\bx)$ is an $n$-dimensional submanifold with boundary equal to $\Gamma \cap B_r(\bx)$. When $T \in \II_n(\breve N) \setminus \II_n^1(\breve N)$ the situation is still addressed by \cite{Hardt-Simon:boundary-regularity} but becomes more complicated; consider, e.g., the sum of two concentric disks in $\RR^2$ with the same orientation but different radii. We recall the relevant results we'll use for $(\breve N, g) \cong \RR^{n+1}$ in Appendix \ref{app:hardt.simon.boundary}. Finally, see \cite{White:boundary-regularity-multiplicity} for the even more general case where $\partial T \in \II_{n-1}(\breve N) \setminus \II_{n-1}^1(\breve N)$ (but still smooth) is allowed.

\subsection{Interior regularity scale}  \label{subsec:regularity.scale}

In order to carry out sharper analysis on minimizers, we will sometimes need to work on parts of the regular set with some control on the regularity scale. Let us introduce the relevant terminology.

\begin{definition}
	Let $T \in \II_n(\breve N)$. For $\bx \in \spt T$ we define the \emph{interior regularity scale} $r_T(\bx)$ to be\footnote{For notational convenience, we suppress the dependence on $\breve N$.} $r_T(\bx) = 0$ for every $\bx \in \sing T \cup \spt \partial T$, and otherwise $r_T(\bx) \in (0, \infty]$ is the supremum of $r\in(0,\operatorname{inj}_{\breve N, g}(\bx))$ (where $\operatorname{inj}_{\breve N, g}(\bx)$ is the injectivity radius at $\bx \in (\breve N, g)$) so that $\partial T \mres B_r(\bx) = 0$, and $T \mres B_r(\bx)$ is an integral current supported on a smooth hypersurface having second fundamental form $|A|\leq r^{-1}$.
\end{definition}

It is easy to check that $\bx\mapsto r_T(\bx)$ is continuous (see Lemma \ref{lemm:reg-scale-cts} below) with respect to the standard extended topology on the target space $[0, \infty]$. Therefore:

\begin{definition}
For $T \in \II_n(\breve N)$ minimizing and for $\delta>0$, we denote
\[
\cR_{\geq \delta}(T) := \{\bx \in \spt T : r_T(\bx) \geq \delta\},
\]
Note that $\cR_{\geq \delta}(T)  \subset \reg T$ is relatively closed.
\end{definition}

More generally we have the following continuity property for the interior regularity scale:

\begin{lemma}\label{lemm:reg-scale-cts}
Let $T, T_1, T_2, \ldots \in \II_n(\breve N)$ be minimizing, $T_j \rightharpoonup T$, and $\spt \partial T_j \to \spt \partial T$ in the local Hausdorff sense. If $\bx \in \spt T$, $\bx_j \in \spt T_j$, and $\bx_j \to \bx$, then $r_{T_j}(\bx_j) \to r_T(\bx)$. 
\end{lemma}
\begin{proof}
First we prove $r_T(\bx)  \leq \liminf_{j\to\infty} r_{T_j}(\bx_j)$. Without loss of generality, assume that $r_T(\bx) > 0$; otherwise there is nothing to prove. Fix $r \in (0, r_T(\bx))$.  We have that $\partial T \mres B_r(\bx) = 0$ and $T\mres B_r(\bx)$ is a union of smooth minimal hypersurfaces satisfying $|A|\leq r^{-1}$. Note that $\partial T_j \mres B_r(\bx) = 0$ for $j$ sufficiently large. By locally decomposing $T_j \mres B_r(\bx)$ into boundaries (\cite[Corollary 27.8]{Simon:GMT}) and invoking De Giorgi's regularity for minimizing boundaries (\cite{DeGiorgi:minimizing-regularity}) on compact subsets of $B_r(\bx)$, we find that $T_j \mres B_r(\bx)$ consists of a union of smooth minimal hypersurfaces converging smoothly (with multiplicity) to $T$. This shows that $r \leq \liminf_{j\to\infty} r_{T_j}(\bx_j)$. Since $r < r_T(\bx)$ was arbitrary, this yields the desired result.

Now we prove the opposite inequality, $r_T(\bx) \geq \limsup_{j\to\infty} r_{T_j}(\bx_j)$. Without loss of generality, assume that $\limsup_{j\to\infty}r_{T_j}(\bx_j)>0$; otherwise there is nothing to prove. Fix $r \in (0, \limsup_{j\to\infty}r_{T_j}(\bx_j))$ and pass to a subsequence so that $r_{T_j}(\bx_j) > r$ for all $j$. By passing this information to the limit we get $r_T(\bx) \geq r$. Since $r < \limsup_{j\to\infty}r_{T_j}(\bx_j)$ was arbitrary, this completes the proof. 
\end{proof}

\subsection{Soft separation of minimizers}

\begin{lemma} \label{lemm:reg.connected}
	Suppose that $T \in \II_n(\breve N)$ is minimizing with $\partial T = 0$, and that $\spt T$ is connected. Then, $\reg T$ is also connected.
\end{lemma}
\begin{proof}
	This follows from \cite[Theorem A (ii)]{Ilmanen:smp}, though the argument essentially goes back to \cite{Simon:smp}.
\end{proof}

Our presentation will be further simplified with the following definition:

\begin{definition}
	Suppose that $T, T' \in \II_n(\breve N)$. We say that $T$, $T'$  \textit{cross smoothly} at $p \in \reg T \cap \reg T'$ if, for all sufficiently small $r > 0$, there are points of $\reg T'$ on both sides of $\reg T$ within $B_r(p)$ and vice versa.
\end{definition}

The power of this definition comes from the strong maximum principle for minimizing currents and their unique continuation. Indeed we have the following lemma:

\begin{lemma} \label{lemm:disjoint.or.cross}
	Let $T, T' \in \II^1_n(\breve N) \setminus \{ 0 \}$ be minimizing with $\partial T = \partial T' = 0$ and with $\spt T$, $\spt T'$ connected. Then, exactly one of the following holds:
	\begin{enumerate}
		\item[(i)] $T = \pm T'$.
		\item[(ii)] $\spt T$ and $\spt T'$ are disjoint.
		\item[(iii)] $T$ and $T'$ cross smoothly at some point.
	\end{enumerate}
\end{lemma}
\begin{proof}
	Clearly the three alternatives are mutually exclusive for nontrivial currents, so it suffices to show that at least one of them always holds. 
	
	Suppose (ii) above fails. Then $\reg T \cap \reg T' \neq \emptyset$ by  \cite[Theorem 1]{Simon:smp}. If (iii) failed too, there would exist some $r > 0$ and $p \in \reg T\cap \reg T'$ so that $(\sing T \cup \sing T')\cap B_r(p) = \emptyset$, $\reg T \cap B_r(p)$ and $\reg T' \cap B_r(p)$ are connected and thus $T$, $T'$ are boundaries in $B_r(p)$ and, finally, if $B_r(p) \setminus \reg T = E_+ \cup E_-$, then $\reg T' \cap B_r(p) \subset \bar E_+$, up to perhaps swapping $E_\pm$. Then, $T = \pm T'$ by the strong maximum principle and unique continuation, so (i) holds.
\end{proof}

In many cases of interest, we can guarantee conclusion (ii) above by keeping track of, e.g., boundaries in a slightly bigger domain.\footnote{Some regularity assumptions of Lemma \ref{lemm:disjoint} can be relaxed, but this is not necessary for us.}

\begin{lemma} \label{lemm:disjoint}
	Fix a smooth, compact, oriented subdomain $\Omega \subset \breve N$ and smooth, disjoint (possibly empty) hypersurfaces $\Gamma, \Gamma' \subset \partial \Omega$ that homologically bound a compact $\Sigma \subset \partial \Omega$ in the sense that $\partial \llbracket \Sigma \rrbracket = \llbracket \Gamma \rrbracket - \llbracket \Gamma' \rrbracket$. Let $T, T' \in \II^1_n(\breve N) \setminus \{ 0 \}$ be minimizing with $\partial T = \llbracket \Gamma \rrbracket$, $\partial T' = \llbracket \Gamma' \rrbracket$, with $\spt T\setminus \spt \partial T $ and $\spt T' \setminus \spt \partial T' $ connected, $\spt T \setminus \partial T \subset \Omega \setminus \partial \Omega$, $\spt T' \setminus \partial T' \subset \Omega \setminus \partial \Omega$, and with $T$ and $T' + \llbracket \Sigma \rrbracket$ homologous in $\Omega$. Then, either
	\begin{enumerate}
		\item[(i)] $T = T'$ (in which case $\Gamma = \Gamma' = \emptyset$), or
		\item[(ii)] $\spt T$ and $\spt T'$ are disjoint.
	\end{enumerate}
\end{lemma}
\begin{proof}
	Assume that (ii) fails; otherwise, there is nothing to show. 
		
	Let's suppose that $T$, $T'$ do not cross smoothly at any point of $\spt T \cap \spt T'$. It follows from Lemma \ref{lemm:disjoint.or.cross} that $T = \pm T'$. If $T = T'$, we are done. Otherwise, $T = -T'$ and $\partial T = \partial T' =0$, so $T$,  $T'$ are homologous in $\Omega$. It then follows that $2T$ is null-homologous in $\Omega$. Using the long exact sequence in relative cohomology, this contradicts that $H_n(\Omega, \ZZ) \cong H^1(\Omega,\partial\Omega,  \ZZ)$ (Lefschetz duality) $\cong \operatorname{Hom}(H_1(\Omega,\partial\Omega, \ZZ), \ZZ)$ is torsion-free.
	
	Thus, we are left to study the situation in which
	\begin{equation} \label{eq:disjoint.x.smooth}
		T, T' \text{ cross smoothly at some } \bx \in \Omega \setminus \partial \Omega.
	\end{equation}
	We will derive a contradiction by combining the cut-and-paste decomposition of \cite{White:boundary-regularity-multiplicity} with a modification taken after the proof of \cite[Theorem 2.2]{Morgan:regularity-modulo-n}. Specifically, we will follow the construction of \cite{White:boundary-regularity-multiplicity} with $p=2$ in the inductive step, and with $T + T' + \llbracket \Sigma \rrbracket$ playing the role of the mod-2 boundary.
	
	By assumption, we have
	\begin{equation} \label{eq:disjoint.support.lateral}
		\spt T \cap (\Sigma \setminus \partial T) = \spt T' \cap (\Sigma \setminus \partial T') = \emptyset.
	\end{equation}
	We also have
	\begin{equation} \label{eq:disjoint.same.bdries}
		\partial T = \partial (T' + \llbracket \Sigma \rrbracket) = \llbracket \Gamma \rrbracket.
	\end{equation}
	Since $T$ and $T' + \llbracket \Sigma \rrbracket$ are homologous in $\Omega$, 
	\begin{equation} \label{eq:disjoint.delta}
		T' + \llbracket \Sigma \rrbracket - T = \partial \Delta
	\end{equation}
	for some $\Delta \in \II_{n+1}(\breve N)$ with $\spt \Delta \subset \Omega$. Let $U$  denote the Caccioppoli set corresponding to the odd multiplicity portion of $\Delta$. It follows from parity considerations that
	\begin{equation} \label{eq:disjoint.e}
		\Delta - \llbracket U \rrbracket = 2 E
	\end{equation}
	where $E \in \II_{n+1}(\breve N)$ has $\spt E \subset \Omega$. We have
	\begin{equation} \label{eq:disjoint.support.lateral.mult.1}
		\spt E \cap  (\Sigma \setminus \partial \Sigma) = \emptyset
	\end{equation}
	from the constancy theorem, \eqref{eq:disjoint.support.lateral} and the fact that, by construction (see \eqref{eq:disjoint.delta}), $\Delta$ has multiplicity one near $\Sigma$. Let us define:
	\begin{equation} \label{eq:disjoint.tildet}
		\tilde T := T + \partial E,
	\end{equation}
	\begin{equation} \label{eq:disjoint.that}
		\hat T := \tilde T + \partial \llbracket U \rrbracket.
	\end{equation}
	By construction, $\tilde T$, $\hat T$ are homologous to $T$ in $\Omega$. By \eqref{eq:disjoint.tildet}, \eqref{eq:disjoint.that}, and \eqref{eq:disjoint.same.bdries}, 
	\begin{equation} \label{eq:disjoint.support.decomposition}
		\tilde T + \hat T = T + T' + \llbracket \Sigma \rrbracket,
	\end{equation}
	\begin{equation} \label{eq:disjoint.support.decomposition.bdry}
		\partial \tilde T = \partial \hat T = \llbracket \Gamma \rrbracket.
	\end{equation}
  In the ``$\nu = 2$'' inductive step of \cite[Decomposition Theorem]{White:boundary-regularity-multiplicity} White proves:
	\begin{equation} \label{eq:disjoint.support.decomposition.mass}
		\Vert T + T' + \llbracket \Sigma \rrbracket \Vert = \Vert \tilde T \Vert + \Vert \hat T \Vert
	\end{equation}
	We briefly recall how. One proves this at the level of densities at points $\by$ where the underlying rectifiable sets have a planar approximate tangent plane with an integer multiplicity (this holds for $\cH^n$-a.e. $\by$). Since $\hat T$, $\tilde T$ differ by the boundary of a Caccioppoli set by \eqref{eq:disjoint.that}, it follows at all such $\by$ that the integer multiplicity of $\hat T$ equals $0$ or $1$ plus that of $\tilde T$. In particular, their integer multiplicities cannot have opposite signs. Then \eqref{eq:disjoint.support.decomposition} implies 
	\[ \Theta_{\tilde T}(\by) + \Theta_{\hat T}(\by) = \Theta_{T+T'+\llbracket \Sigma \rrbracket}(\by), \]
	at all such $\by$, which in turn implies \eqref{eq:disjoint.support.decomposition.mass}.
	
	It also follows from \eqref{eq:disjoint.support.lateral} and \eqref{eq:disjoint.support.lateral.mult.1} that $\Vert \llbracket \Sigma \rrbracket \Vert \leq \Vert \hat T \Vert$. Set
	\begin{equation} \label{eq:disjoint.tildetprime}
		\tilde T' := \hat T - \llbracket \Sigma \rrbracket.
	\end{equation}
	Then, \eqref{eq:disjoint.support.lateral}, \eqref{eq:disjoint.support.decomposition.mass}, \eqref{eq:disjoint.tildetprime}, and the unit multiplicity along $\Sigma$ yield
	\begin{equation} \label{eq:disjoint.support.decomposition.mass.no.sigma}
		\Vert T + T' \Vert = \Vert \tilde T \Vert + \Vert \tilde T' \Vert,
	\end{equation}
	while \eqref{eq:disjoint.support.decomposition.bdry} and \eqref{eq:disjoint.tildetprime} give
	\begin{equation} \label{eq:disjoint.support.decomposition.bdry.no.sigma}
		\partial \tilde T = \llbracket \Gamma \rrbracket \text{ and } \partial \tilde T' = \llbracket \Gamma' \rrbracket.
	\end{equation}
	It follows from the minimizing nature of $T$ and $T'$ and from \eqref{eq:disjoint.support.decomposition.mass.no.sigma} that
	\[ \MM(T) + \MM(T') \leq \MM(\tilde T) + \MM(\tilde T') = \MM(T+T') \leq \MM(T) + \MM(T'), \]
	so, by \eqref{eq:disjoint.support.decomposition.bdry.no.sigma} and the equality in homology of $T$, $\tilde T$ and of $T'$, $\tilde T'$, it follows that
	\begin{equation} \label{eq:disjoint.tildet.prime.minimizing}
		\tilde T, \tilde T' \text{ are minimizing.}
	\end{equation}
	
	We now derive a contradiction locally near $\bx$, the crossing point from \eqref{eq:disjoint.x.smooth}. 	Choose $r > 0$ small enough that $B_r(\bx) \subseteq \Omega \setminus \partial \Omega$ and
	\begin{equation} \label{eq:disjoint.support.decomposition.T.bdryA}
		T \mres B_r(\bx) = \partial \llbracket A \rrbracket \mres B_r(\bx),
	\end{equation}
	\begin{equation} \label{eq:disjoint.support.decomposition.Tprime.bdryAprime}
		T' \mres B_r(\bx) = \partial \llbracket A' \rrbracket \mres B_r(\bx),
	\end{equation}
	for two open sets $A$, $A' \subseteq B_r(\bx)$. It follows from \eqref{eq:disjoint.delta} and the constancy theorem that there is $k\in \ZZ$ such that
\begin{equation*}
 \Delta \mres B_r(\bx) = k \llbracket B_r(\bx) \rrbracket + \llbracket A'\setminus A \rrbracket - \llbracket A\setminus A' \rrbracket\, ,
\end{equation*}
	which in case $k$ is even implies that 
	$$U\cap B_r(\bx) = A'\setminus A \cup A\setminus A'\, .$$
	From \eqref{eq:disjoint.e} we see that
	$$ 2E\mres B_r(\bx) = k \llbracket B_r(\bx) \rrbracket - 2 \llbracket A\setminus A' \rrbracket\, .$$
	Thus
	$$ \tilde T\mres B_r(\bx) = T\mres B_r(\bx)+\partial E\mres B_r(\bx) = \partial \llbracket A'\cap A \rrbracket$$
	and
	$$ \tilde T'\mres B_r(\bx) = \partial \llbracket A'\cup A \rrbracket\, .$$
	In case $k$ is odd a similar reasoning yields that 
$$ \tilde T\mres B_r(\bx) = \partial \llbracket A'\cup A \rrbracket$$
	and
	$$ \tilde T'\mres B_r(\bx) = \partial \llbracket A'\cap A \rrbracket\, .$$
	By \eqref{eq:disjoint.tildet.prime.minimizing} both $\tilde T, \tilde T'$ are minimising, but $\bx \in \spt \tilde T \cap \spt \tilde T'$. Thus by Lemma  \ref{lemm:disjoint.or.cross} we have $\tilde T\mres B_r(\bx) = \tilde T'\mres B_r(\bx)$ which contradicts that $T$ and $T'$ are smoothly crossing at $\bx$.  This completes the proof.

\end{proof}

\subsection{Minimizing cones}

To engage in the refined study of the local structure of a minimizing $T \in \II_n(\breve N)$ at $\bx \in \sing T$ one often supposes without loss of generality that $T \in \II_n^1(\breve N)$ and that in fact $T$ is a minimizing boundary. This can be arranged by the local decomposition of $T$ into minimizing boundaries (\cite[Corollary 27.8]{Simon:GMT}). Then one takes a sequence of dilations $(\Dil_{\bx, \lambda_i})_\# (\breve N, g)$ of $(\breve N, g)$ centered\footnote{To be precise, by $(\Dil_{\bx, \lambda_i})_\# (\breve N, g)$ we mean the pointed Riemannian manifold $(\breve N,\lambda_i^{-2} g,\bx)$, which will converge to flat Euclidean space $(\RR^{n+1},g_{\RR^{n+1}},\bOh)$ in the pointed Cheeger--Gromov sense. We'll also use the same notation $\Dil_{\bx,\lambda_i}$ for the linear map $\RR^{n+1}\to\RR^{n+1}$, $\by \mapsto \lambda_i^{-1}(\by-\bx)$. The distinction will be clear from context.} at $\bx$ with $\lambda_i \to 0$, and after passing to a subsequence (not labeled) obtains (\cite[Theorem 35.1]{Simon:GMT})
\begin{equation} \label{eq:tangent.cone}
	\cC = \lim_{i \to \infty} (\Dil_{\bx, \lambda_i})_\# T \in \II^1_n(\RR^{n+1})
\end{equation}
which is also a minimizing boundary and invariant under dilations centered at $\bOh \in \RR^{n+1}$, i.e., it is a \emph{cone}. This $\cC$ is called a \emph{tangent cone} of $T$ at $\bx$. We often identify $\cC$ with its support in $\RR^{n+1}$ and simply refer to $\cC \subseteq \RR^{n+1}$ as a tangent cone of $T$ at $\bx$.

The main tool in the proof of \eqref{eq:tangent.cone} is the compactness theorem for minimizing boundaries and the approximate monotonicity formula for the density function
\begin{equation} \label{eq:density.scale.r}
	\Theta_T(\bx, r) = \frac{\MM(T \mres B_r(\bx))}{\omega_n r^n}, \;  0 < r < r_0,
\end{equation}
where $\MM$ denotes the mass of a current, $\omega_n$ is the volume of a unit $n$-ball in $\RR^n$, and $r_0 = \min\{ \operatorname{inj}_{(\breve N, g)}(\bx), \dist_g(\bx, \spt \partial T) \}$. The approximate monotonicity formula implies that
\begin{equation} \label{eq:monotonicity.lambda}
	r \mapsto e^{\Lambda r} \Theta_T(\bx, r) \text{ is non-decreasing},
\end{equation}
where $\Lambda = \Lambda(\breve N, g)$ is a constant converging to $0$ as $(\breve N, g) \to \RR^{n+1}$ (or a fixed open ball in it) in $C^2$. Note that \eqref{eq:monotonicity.lambda} guarantees that
\begin{equation} \label{eq:density.scale.0}
	\Theta_T(\bx) = \lim_{r \to 0} \Theta_T(\bx, r)
\end{equation}
exists. Allard's regularity theorem guarantees that $\Theta_T(\bx) > 1$ (since $\bx \in \sing T$) and, in fact, that $\Theta_T(\bx) \geq 1 + \alpha$ for some $\alpha = \alpha(n) > 0$. For $n\geq 7$ denote
\begin{equation} \label{eq:theta.star}
	\Theta^*_n := \text{minimal density $\Theta_\cC(\bOh)$ of a nonflat minimizing cone $\cC^n \subset \RR^{n+1}$};
\end{equation}
it is not hard to show the minimum is attained. Note that this is only well-defined for $n \geq 7$, since by \cite{Simons:minvar} there are no nonflat minimizing cones when $n \leq 6$.

There is also a rigidity statement associated with \eqref{eq:monotonicity.lambda} when we're in $\RR^{n+1}$ (where $\Lambda = 0$). It implies that if $T \in \II_n^1(\RR^{n+1})$ is a minimizing boundary and $\Theta_T(\bOh, R_2) = \Theta_T(\bOh, R_1)$ for some $R_2 > R_1 > 0$, then $T$ is a minimizing cone $\cC \subseteq \RR^{n+1}$. This follows from the rigidity case of \eqref{eq:monotonicity.lambda}, the connectedness of $\reg T$ (see Lemma \ref{lemm:reg.connected}), and unique continuation.

A key insight in the regularity theory, particularly in the implementation of Federer's dimension reduction scheme to study the regularity of minimizing cones $\cC^n \subseteq \RR^{n+1}$, is to consider
\begin{equation}
\spine \cC : = \{\bx \in \cC : \Theta_\cC(\bx) = \Theta_\cC(\bOh)\}.
\end{equation}
It follows from a standard argument (``cone-splitting'') that $\spine \cC\subset \RR^{n+1}$ is a linear subspace and that $\cC = \cC + \bx$ for $\bx \in \spine \cC$ so $\cC$ splits with respect to the orthogonal decomposition of $\RR^{n+1}$ as $ (\cC \cap (\spine \cC)^\perp) \oplus  \spine \cC$. It also follows from \cite{Simons:minvar} that $\dim \spine \cC \leq n-7$ whenever $\cC$ is a \textit{nonflat} minimizing cone.

\section{Minimizing foliations} \label{sec:minimizing.foliations}

\subsection{Density-drop and cone-splitting}

The key innovation used by Hardt--Simon to prove Theorem \ref{theo:main.plateau} in $\RR^8$ is the following result. Note that the assumption $\sing \cC = \{ \bOh \}$ below is always true for nonflat minimizing cones in $\RR^8$, but not in $\RR^{n+1}$ for $n+1 \geq 9$.

\begin{theorem}[{\cite[Theorem 2.1]{Hardt-Simon:isolated-singularities}}]\label{thm:HS-fol}
Suppose that $\cC^n \subset \RR^{n+1}$ is a minimizing cone with $\sing \cC = \{\bOh\}$. Then, there exist two minimizing $T_\pm \in \II^1_n(\RR^{n+1})$ with $\partial T_\pm = 0$ and with $\spt T_\pm$ being smooth, connected, and so that $\spt T_+$ and $\spt T_-$ lie in distinct components of $\RR^{n+1}\setminus \spt \cC$. These $T_\pm$ have the additional property that for all minimizing $T \in \II^1_n(\RR^{n+1})$ with $\partial T = 0$ and with connected $\spt T$, it holds that
\[ T \text{ does not cross } \cC \text{ smoothly} \]
if and only if 
\[ T = \pm \cC \text{ or } (\Dil_{\bOh,\lambda})_\#T_\pm  \text{ for some $\lambda \neq 0$.} \]
\end{theorem}

Our key observation in \cite[Appendix D]{CCMS:low-ent-gen} was that a weaker version of Hardt--Simon's Theorem \ref{thm:HS-fol} suffices to prove Theorems \ref{theo:main.plateau} and \ref{theo:main.homology} in $8$ ambient dimensions. Specifically, one just needs to prove the following density-drop result on either side of a minimizing cone with an isolated singularity:
 
\begin{proposition}[cf. {\cite[Proposition D.2]{CCMS:low-ent-gen}}]\label{prop:dens-drop-smooth}
Suppose that $\cC^n \subset \RR^{n+1}$ is a minimizing cone with $\sing \cC = \{\bOh\}$. If $T \in \II^1_n(\RR^{n+1})$ is minimizing with $\partial T = 0$, $\spt T$ is connected, and $T$ does not cross $\cC$ smoothly, then
\[ \Theta_T(\bx) < \Theta_{\cC}(\bOh) \text{ for every } \bx \in \spt T \setminus \{ \bOh \}. \]
\end{proposition}

Theorem \ref{thm:HS-fol} immediately implies Proposition \ref{prop:dens-drop-smooth} using Allard's regularity theorem, but Proposition \ref{prop:dens-drop-smooth} extends to cover minimizers on one side of arbitrary nonflat minimizing cones (i.e., arbitrarily singular ones) provided it is recast as a density-drop / cone-splitting dichotomy:

\begin{proposition} \label{prop:dens-drop-sing}
Suppose that $\cC^n \subset \RR^{n+1}$ is a nonflat minimizing cone. If $T \in \II^1_n(\RR^{n+1})$ is minimizing with $\partial T = 0$, $\spt T$ is connected, and $T$ does not cross $\cC$ smoothly, then
\[ \Theta_T(\bx) \leq \Theta_{\cC}(\bOh) \text{ for every } \bx \in \spt T. \]
Equality holds if and only if $T= \pm \cC$ and $\bx \in \spine \cC$. 
\end{proposition}

This is a key tool for our paper.

\begin{remark}
The existence part of Theorem \ref{thm:HS-fol} was recently proven without the assumption $\sing \cC = \{\bOh\}$ by Wang \cite{Wang:smoothing}. Additional uniqueness results have been proven by Simon \cite{Simon:liouville} and Edelen--Sz\'{e}kelyhidi \cite{EdelenSzekelyhidi:liouville}. If the uniqueness in the Hardt--Simon result (Theorem \ref{thm:HS-fol}) was known to hold for all minimizing cones, then combined with the strong maximum principle \cite{Simon:smp}, Proposition \ref{prop:dens-drop-sing} above would immediately follow, but the proofs of Theorems \ref{theo:main.plateau} and \ref{theo:main.homology} would not simplify significantly.\footnote{Lohkamp has claimed to extend the uniqueness statement in Theorem \ref{thm:HS-fol} to cover all minimizing cones \cite{Lohkamp:HS} using his theory of ``hyperbolic unfoldings.'' Even assuming such a uniqueness statement, it remained unclear prior to our paper how it could be used to prove Theorems \ref{theo:main.plateau}, \ref{theo:main.homology} in any dimension $n+1 \geq 9$; cf.\ the remark in \cite[p.\ 3]{Lohkamp:HS}.} 
\end{remark}

\begin{proof}[Proof of Proposition \ref{prop:dens-drop-sing}]
First note that $T$ decomposes a priori by \cite[Theorem 27.6]{Simon:GMT} into a sum of a priori several nested minimizing boundaries in $\RR^{n+1}$. It follows from Lemma \ref{lemm:disjoint.or.cross} and the unit multiplicity of $T$ that the nested boundaries have disjoint supports. But $\spt T$ is assumed to be connected, so $T$ is itself a minimizing boundary. 

For any $\lambda_j\to \infty$, the blowdown $(\Dil_{\bOh,\lambda_j})_\#T$ will subsequentially converge to a tangent cone at infinity $\cC'$.\footnote{Tangent cones at infinity exist when the density ratios $\Theta_T(\bx,r)$ are uniformly bounded as $r\to\infty$; here, $T$ is a boundary, which allows for comparison with parts of spheres (cf.\ \cite[Theorem 37.2]{Simon:GMT}).} Note that $\cC'$ is a minimizing boundary (since $T$ was) and does not cross $\cC$ smoothly (or else $T$ would, too). Since $\bOh \in \spt \cC\cap\spt \cC'$, we have $\cC' = \cC$ by Lemma \ref{lemm:disjoint.or.cross} up to swapping orientations. For $\bx \in \spt T$, we have by the monotonicity formula that 
\[
\Theta_T(\bx) \leq \lim_{r\to\infty} \Theta_T(\bx,r) = \Theta_{\cC'}(\bOh) =\Theta_{\cC}(\bOh), 
\]
as claimed.

Now, assume that $\Theta_T(\bx) = \Theta_{\cC}(\bOh)$ for some $\bx \in \spt T$. By the same computation, we see that $r\mapsto \Theta_T(\bx,r)$ is constant, so $T$ is a cone centered at $\bx$ by the rigidity case of the monotonicity formula. The blow-down of a cone centered at $\bx$ is the same cone centered at $\bOh$, so $(\Dil_{\bx,1})_\# T = \cC$. 

The proof will be completed once we show that $\bx \in \spine \cC$. We begin by assuming that $\spt T \cap \spt \cC \not = \emptyset$. By Lemma \ref{lemm:disjoint.or.cross} we find that $T = \pm \cC$. We then have
\[
\Theta_\cC(\bOh) = \Theta_{(\Dil_{\bx,1})_\# T}(\bOh) = \Theta_T(\bx) = \Theta_\cC(\bx) 
\]
which implies that $\bx \in \spine \cC$. 

Now suppose that $\spt T \cap \spt \cC = \emptyset$. We will show that this leads to a contradiction. Write $\cC = \pm \partial \llbracket E_\cC \rrbracket$ where $E_{\cC}$ is a dilation-invariant and connected open set (cf.\ \cite[Theorem 1]{BG}). Up to replacing $E_\cC$ with the interior of its complement, we see that 
\[  E_{\cC} + \bx \subset  E_{\cC}. \]
Indeed, $E_\cC$ are connected with connected boundaries that do not intersect. Since $E_\cC$ is dilation-invariant, we have
\[
E_{\cC} + \lambda \bx = (\Dil_{\bOh,\lambda^{-1}})_\#( E_{\cC} + \bx) \subset  E_{\cC}
\]
for any $\lambda > 0$. Sending $\lambda\to 0$ we see that 
\[ \bx \cdot \nu_\cC \geq 0 \text{ along } \reg \cC \]
for the unit normal $\nu_\cC$ to $\reg \cC$ pointing into $E_\cC$. Note that $\bx \cdot \nu_\cC$ is a Jacobi field along $\reg \cC$ since the minimal surface equation is dilation invariant. If $\bx\cdot \nu_\cC \equiv 0$, then $\cC$ splits in the $\bx$ direction, so $\bx \in \spine \cC$. Thus, it suffices to show that $\bx\cdot \nu_\cC \not \equiv 0$ leads to a contradiction. In this case, it must hold that $\bx\cdot \nu_\cC > 0$ by the strict maximum principle and connectedness of $\reg \cC$. Now, we can apply \cite{Simon:asymptotic-decay}  (see also \cite[Theorem 1.4]{Wang:smoothing}) to see that $\bx \cdot \nu_\cC$ has to blow up towards $\sing C$. However, $|\bx \cdot \nu_\cC| \leq r$ on $\cC \cap B_r$. This proves that $\sing \cC$ is empty so $\cC$ is flat. This is a contradiction. 
\end{proof}

Proposition \ref{prop:dens-drop-sing} has a crucial high-codimension-clustering consequence for high-density points of \textit{all} disjoint minimizers near our central one. 

\begin{lemma} \label{lemm:cone-split-pre}
Let $\gamma > 0$. There exists $\eta_0=\eta_0(n,\gamma) \in (0, \tfrac13)$ with the following property.

Take $\eta \in (0, \eta_0)$ and $T \in \II^1_n(B_{\eta^{-1}}(\bOh))$ to be minimizing with $\partial T = 0$ in $B_{\eta^{-1}}(\bOh)$, $\spt T$ connected, and with $\bOh \in \spt T$ satisfying 
\[ \Theta_{T}(\bOh) \geq \max\{ \Theta_n^*, \Theta_{T}(\bOh,1) \} - 3\eta. \]
Then, there exists a fixed subspace $\Pi \subset \RR^{n+1}$, with $\dim \Pi \leq n-7$, such that whenever $T' \in \II^1_n(B_{\eta^{-1}}(\bOh))$ is also minimizing with $\partial T' = 0$ in $B_{\eta^{-1}}(\bOh)$, has connected $\spt T'$, does not cross $T$ smoothly in $B_{\eta^{-1}}(\bOh)$, and $\bx' \in \spt T' \cap \bar B_{1}(\bOh)$ satisfies 
\[ \Theta_{T'}(\bx') \geq \max\{ \Theta_n^*, \Theta_{T}(\bOh,1) \} -3\eta, \]
it must be that $\bx' \in U_{\gamma}(\Pi)$ where $U_\gamma(\Pi)$ is the distance $< \gamma$ neighborhood of $\Pi$.
\end{lemma}
\begin{proof} 
Assume, for the sake of contradiction, that for all sufficiently large $j$ there exist $T_j$ satisfying the conditions above with $\eta = j^{-1}$, so in particular
\begin{equation} \label{eq:cone-splitting-pre-center}
	\bOh \in \spt T_j \text{ and } \Theta_{T_j}(\bOh) \geq \max\{ \Theta_n^*, \Theta_{T_j}(\bOh,1) \} - j^{-1}
\end{equation}
and so that for every subspace $\Pi \subset \RR^{n+1}$ with $\dim \Pi \leq n-7$ there is a $T'_{\Pi, j}$ as above with
\begin{equation} \label{eq:cone-splitting-pre-leaf}
	\bx_{\Pi,j}' \in \spt T_{\Pi, j}' \cap \bar B_1(\bOh) \setminus U_{\gamma}(\Pi), \; \Theta_{T'_{\Pi, j}}(\bx'_{\Pi, j}) \geq \max\{ \Theta_n^*, \Theta_{T_j}(0,1)\} - j^{-1}.
\end{equation}
Arguing as in the start of the proof of Proposition \ref{prop:dens-drop-sing}, $T_j$, $T_j'$ are minimizing  boundaries in $B_j(\bOh)$. Passing to a subsequence (not relabeled), $T_j \rightharpoonup T$ is a minimizing boundary in $\RR^{n+1}$ whose support is necessarily connected (see \cite[Theorem 1]{BG}). Moreover, by \eqref{eq:cone-splitting-pre-center} and the upper-semicontinuity of density,
\begin{equation} \label{eq:cone-splitting-post-center}
	\Theta_{T}(\bOh) \geq \max \{ \Theta_n^*, \Theta_{T}(\bOh,1) \}.
\end{equation}
By the monotonicity formula and its rigidity case, together with unique continuation, $T$ is a \textit{nonflat} minimizing cone, say $T = \cC$. 

We now fix $\Pi := \spine \cC$, which we know is $\leq (n-7)$-dimensional since $\cC$ is a nonflat minimizing cone. Passing to a further subsequence (not relabeled) we have $T'_{\Pi, j} \rightharpoonup T'$ for a minimizing boundary in $\RR^{n+1}$. Note that, by the monotonicity formula, $\spt T'$ is connected since each $\spt T'_{\Pi,j}$ was connected. Also, $T'$ cannot cross $T$ smoothly since $T_j$ and $T_j'$ are assumed not to cross smoothly for any $j$. Finally, note that
\[ \bx_{\Pi,j}' \to \bx' \in \spt T' \cap \bar B_1(\bOh) \setminus U_\gamma(\Pi), \]
with
\begin{equation} \label{eq:cone-splitting-post-leaf}
	\Theta_{T'}(\bx') \geq \max\{ \Theta_n^*, \Theta_T(\bOh, 1) \} = \Theta_\cC(\bOh)
\end{equation}
by \eqref{eq:cone-splitting-pre-leaf}, the upper-semicontinuity of densities, the definition of $\Theta_n^*$, and $T = \cC$ being nonflat. By Proposition \ref{prop:dens-drop-sing}, $T' = \pm \cC$ and $\bx' \in \spine \cC \cap \bar B_1(\bOh) \setminus U_\gamma(\Pi)$, contradicting $\spine \cC = \Pi$.
\end{proof}

\begin{remark} \label{rema:cone-split-pre}
	Using straightforward modifications to the proof of Lemma \ref{lemm:cone-split-pre}, we may assume $B_{\eta^{-1}}(\bOh)$ is endowed with a background metric $g$ that is $\eta$-close in $C^2$ to being Euclidean. Indeed, in our sequence of counterexamples we can also keep track of metrics $g_j$ on $B_j(\bOh)$ that are $j^{-1}$-close to Euclidean in $C^2$. In the limit, the $g_j$ converge in $C^2$ to a Euclidean metric on $\RR^{n+1}$. The approximate monotonicity formula has a correction factor that vanishes as $j \to \infty$ and guarantees the upper-semicontinuity of densities in the limit as $j \to \infty$. Thus, \eqref{eq:cone-splitting-post-center}, \eqref{eq:cone-splitting-post-leaf} still hold in the $\RR^{n+1}$ blow-up limit, leading to the same contradiction.
\end{remark}

\subsection{Local-to-global separation estimates}

The partial Harnack theory for positive Jacobi fields developed in  \cite{Simon:asymptotic-decay}, with suitable modifications from \cite[Appendix A]{Wang:smoothing}, will let us prove local-to-global implications of the clustering of high-density singular points of non-intersecting minimizers. We point the reader to Appendix \ref{app:positive.jacobi.fields.cones} for the necessary partial Harnack theory background material.

\begin{definition} \label{defi:jacobi.field}
If $\cC^n \subset \RR^{n+1}$ is a minimizing cone, we say that $u$ is a Jacobi field on $\reg \cC$ if $u : \reg \cC \to \RR$ is a classical solution of the Jacobi field equation
\[ \Delta_\cC u + |A_\cC|^2 u = 0 \text{ on } \reg \cC, \]
where $\Delta_\cC$ is the induced Laplace-Beltrami operator and $A_\cC$ is the second fundamental form of $\reg \cC$ in $\RR^{n+1}$. 
\end{definition}

Note, in particular, that we place no growth assumptions on $u$ near $\sing \cC$. Positive Jacobi fields will play a particularly important role in our paper, because they arise in situations where we have two disjoint minimizers  converging to the minimizing cone $\cC$. In that case, the size of the positive Jacobi field informs the vertical separation of our minimizers from each other. 

\begin{lemma}\label{lemm:growth-estimate}
There exists $\rho_0 = \rho_0(n) > 0$ such that for any nonflat minimizing cone $\cC^n$ in $\RR^{n+1}$ and any positive Jacobi field $u$ on $\reg \cC$, 
\[
\inf_{\cR_{\geq r\rho}(\cC) \cap \partial B_r(\bOh)} u \leq r^{-\kappa_n^*} \sup_{\cR_{\geq \rho}(\cC) \cap \partial B_1(\bOh)} u 
\]
for all $r\geq 1$ and $\rho \in (0, \rho_0)$. Here, for $n \geq 7$, $\kappa_n^*$ is explicitly given by
\begin{equation}\label{def:gamma_n}
\kappa_n^* : = \frac{n-2}{2} - \sqrt{\frac{(n-2)^2}{4} - (n-1)}.
\end{equation}
\end{lemma}

\begin{remark} \label{rema:kappa.n}
It is crucial for our argument that $\kappa_n^* > 1$ for all $n$. Most importantly,
\begin{align*}
	\kappa_8^* & \approx 1.58, \\
	\kappa_9^* & \approx 1.44, \\
	\kappa_{10}^* & \approx 1.35.
\end{align*}
\end{remark}

\begin{proof}[Proof of Lemma \ref{lemm:growth-estimate}]
We take $\rho_0$ to be as in Lemma \ref{lem:Wang-1} and will also use the notation $\cR_{\geq \delta}(\Sigma)$ from \eqref{eq:r.link}. Invoke Lemma \ref{lem:Wang-2} with a smooth compact domain $\Omega$ satisfying $\cR_{\geq \rho_0}(\Sigma) \subset \Omega \subset \cR_{\geq \rho}(\Sigma)$ to deduce
\begin{align*}
	\left( \inf_{\omega \in \cR_{\geq \rho}(\Sigma)} u(r \omega) \right) \int_{\Omega}\varphi_\Omega \, d\mu_\Sigma(\omega) 
		& \leq V(r) \\
		& \leq r^{-\kappa_n^*} V(1) \\
		& \leq r^{-\kappa_n^*} \left( \sup_{\omega \in \cR_{\geq \rho}(\Sigma)} u(\omega) \right) \int_\Omega \varphi_\Omega \, d\mu_\Sigma.  
\end{align*}
The result follows by observing that $r(\cR_{\geq \rho}(\cC) \cap \partial B_1(\bOh)) = \cR_{\geq r\rho}(\cC) \cap \partial B_r(\bOh)$.
\end{proof}

We will show that the growth estimate above together with the Harnack inequality implies a growth estimate for the infimum of a positive Jacobi field $u$ on $\reg C$.

\begin{proposition}\label{prop:Harnack-cone}
There exist $\rho_1=\rho_1(n) \in (0, \rho_0(n))$ and $H = H(n) > 0$ such that if $\cC^n$ is a nonflat minimizing cone in $\RR^{n+1}$ and $u$ is a positive Jacobi field on $\reg \cC$, then, for every $r > 0$,
\[ \cR_{\geq r \rho_1}(\cC)\cap \partial B_r(\bOh) \neq \emptyset \]
and
\[
 \sup_{\cR_{\geq r \rho_1/2}(\cC) \cap \partial B_r(\bOh)} u \leq H \inf_{\cR_{\geq r \rho_1/2}(\cC) \cap \partial B_r(\bOh)} u.
\]
\end{proposition}
\begin{proof}
By scaling invariance it suffices to set $r = 1$. 

We first show there exists $\rho_1 \in (0, \rho_0)$ so that $\cR_{\geq \rho_1}(\cC)\cap \partial B_1(\bOh) \neq \emptyset$ for all minimizing cones $\cC^n\subset \RR^{n+1}$. Otherwise, for all $j > \rho_0^{-1}$ there would exist $\cC_j$ with
\begin{equation} \label{eq:Harnack-cone.scale.contradiction}
	\cR_{\geq j^{-1}}(\cC_j) \cap \partial B_1(\bOh) = \emptyset.
\end{equation}
After passing to a subsequence (not relabeled), we have $\cC_j\rightharpoonup \cC$ for a minimizing cone $\cC$. This contradicts \eqref{eq:Harnack-cone.scale.contradiction} together with Lemma \ref{lemm:reg-scale-cts} and $\cup_{\rho > 0} \cR_{\geq \rho}(\cC) = \reg \cC \neq \emptyset$.

We now claim that this choice of $\rho_1$ satisfies the assertion for $H$ sufficiently large. If not, there are cones $\cC_j$ and positive Jacobi fields $u_j$ on $\reg \cC_j$ such that (after rescaling)
\begin{equation} \label{eq:Harnack-cone.sup.contradiction}
	\sup_{\cR_{\geq \rho_1/2}(\cC_j)\cap\partial B_1(\bOh)} u_j = 1,
\end{equation} 
\begin{equation} \label{eq:Harnack-cone.inf.contradiction}
	\inf_{\cR_{\geq \rho_1/2}(\cC_j)\cap\partial B_1(\bOh)} u_j \to 0.
\end{equation}
By Allard's regularity theorem, the standard Harnack inequality, and standard Schauder estimates, we know that after passing to a subsequence (not relabeled), $\cC_j \rightharpoonup \cC$ and $u_j \to u$ locally smoothly in $\reg \cC$, where $\cC$ is a minimizing cone and $u$ is a nonnegative Jacobi field on $\reg \cC$, $u=1$ at one point on $\reg \cC$ (by \eqref{eq:Harnack-cone.sup.contradiction}),  and $u = 0$ at another point of $\reg \cC$ (by \eqref{eq:Harnack-cone.inf.contradiction}). This contradicts the strict maximum principle since $\reg \cC$ is connected (by Lemma \ref{lemm:reg.connected}). 
\end{proof}

Combining Lemma \ref{lemm:growth-estimate} and (twice) Proposition \ref{prop:Harnack-cone} yields the following.

\begin{corollary}\label{coro:decay-est-Jacobi}
There exist $\rho_1=\rho_1(n) \in (0, \rho_0(n))$ and $H = H(n) > 0$ such that for any nonflat minimizing cone  $\cC^n$ in $\RR^{n+1}$ and any positive Jacobi field $u$ on $\reg \cC$, 
\[
\sup_{\cR_{\geq r\rho_1/2}(\cC) \cap \partial B_r(\bOh)} u \leq H^2 r^{-\kappa_n^*} \inf_{\cR_{\geq \rho_1/2}(\cC) \cap \partial B_1(\bOh)} u 
\]
for $\kappa_n^*$ as in \eqref{def:gamma_n} and all $r \geq 1$. 
\end{corollary}

The decay estimate for positive Jacobi fields on cones can be ``delinearized'' to yield the following distance decay estimate for non-intersecting minimizers, when one of them is close to a cone. 

Fix $\rho_2 = \rho_2(n) := \tfrac12 \rho_1(n)$ ($< \rho_1(n) < \rho_0(n)$).

\begin{lemma}\label{lemm:sep-est}
Let $\lambda \in (0, \kappa_n^*+1)$. There exist $\eta_1=\eta_1(n,\lambda) \in (0, \tfrac13)$, $A_0=A_0(n,\lambda)>1$ with the following property: 

If $\eta \in (0,\eta_1)$, $T, T' \in \II^1_n(B_{\eta^{-1}}(\bOh))$ are minimizing with $\partial T = \partial T' = 0$ in $B_{\eta^{-1}}(\bOh)$, $\spt T$ and $\spt T'$ are connected, $T$ and $T'$ do not cross smoothly in $B_{\eta^{-1}}(\bOh)$, and 
\[
\min\{\Theta_T(\bOh),\Theta_{T'}(\bx')\} \geq \max\{ \Theta_n^*, \Theta_T(\bOh,1) \} - 3\eta,
\]
with $\bx' \in \bar B_1(\bOh)$, then
\[ \cR_{\geq \rho_2}(T) \cap \partial B_1(\bOh) \neq \emptyset, \; \cR_{\geq A_0 \rho_2}(T) \cap \partial B_{A_0}(\bOh) \neq \emptyset, \]
and
\[
	d(\cR_{\geq A_0\rho_2}(T) \cap \partial B_{A_0}(\bOh), \spt T') / A_0 \leq A_0^{-\lambda} \cdot d(\cR_{\geq \rho_2}(T) \cap \partial B_1(\bOh), \spt T').
\]
Above, $d$ denotes the infimum of pairwise distances. 
\end{lemma} 
\begin{proof}
Choose $A_0=A_0(n,\lambda)$ sufficiently large so that 
\begin{equation} \label{eq:coro-est-A}
	H^2 A_0^{-\kappa_n^*} \leq \tfrac12 A_0^{1-\lambda},
\end{equation}
with $H$, $\kappa_n^*$ as in Corollary \ref{coro:decay-est-Jacobi}. 

We argue by contradiction. Assume that $T_j$ and $T'_j$ are as above with $\eta = j^{-1}$, and
\begin{equation}\label{eq:density-assumption}
 \min\{\Theta_{T_j}(\bOh),\Theta_{T'_j}(\bx_j')\} \geq \max\{ \Theta_n^*, \Theta_{T_j}(\bOh,1) \} - j^{-1}
\end{equation}
with $\bx_j' \in B_1(\bOh)$, but either
\begin{itemize}
 \item[(i)] $\cR_{\geq \rho_2}(T_j) \cap \partial B_1(\bOh) = \emptyset$ or $\cR_{\geq A_0 \rho_2}(T_j) \cap \partial B_{A_0}(\bOh) = \emptyset$; or
 \item[(ii)] the desired distance estimate fails, i.e.
 \begin{equation} \label{eq:sep-est.fail}
 	A_0^{-\lambda} \cdot d(\cR_{\geq \rho_2}(T_j) \cap \partial B_1(\bOh), \spt T_j') < d(\cR_{\geq A_0\rho_2}(T_j) \cap \partial B_{A_0}(\bOh), \spt T_j') / A_0.
\end{equation}
\end{itemize}

Arguing as in the start of the proof of Proposition \ref{prop:dens-drop-sing}, we know $T_j$, $T_j'$ are minimizing  boundaries in $B_j(\bOh)$. Then assumption \eqref{eq:density-assumption} and Proposition \ref{prop:dens-drop-sing} taken together imply that after passing to a subsequence (not relabeled), $T_j$ and $T_j'$ converge to the same minimizing cone $\cC$ with multiplicity one (after perhaps swapping orientations). By Lemma \ref{lemm:reg-scale-cts}, Proposition \ref{prop:Harnack-cone}, and $\rho_2 = \tfrac12 \rho_1 < \rho_1$, we find that (i) cannot hold for infinitely many $j$. 

So let's assume only (ii) holds. Because $T_j$, $T_j'$ converge locally smoothly to $\cC$ away from $\sing \cC$, and do not cross smoothly, we can take the difference of the corresponding graphical functions over $\cC$ and use the standard Harnack inequality and the connectedness of $\reg \cC$ (see Lemma \ref{lemm:reg.connected}) to define a positive Jacobi field $u$ on $\reg \cC$. Then (ii) rescales to give, in the limit as $j \to \infty$, 
\begin{equation} \label{eq:sep-est.bad.jf}
	A_0^{1-\lambda} \inf_{\cR_{\geq \rho_2}(\cC) \cap \partial B_1(\bOh)} u \leq \sup_{\cR_{\geq A_0\rho_2}(\cC)\cap\partial B_{A_0}(\bOh)} u
\end{equation}
This contradicts Corollary \ref{coro:decay-est-Jacobi} (applied with $r = A_0$) in view of \eqref{eq:coro-est-A}.
\end{proof}

\begin{remark} \label{rema:sep-est}
	Using straightforward modifications to the proof of Lemma \ref{lemm:sep-est}, we may assume $B_{\eta^{-1}}(\bOh)$ is endowed with a background metric $g$ that is $\eta$-close in $C^2$ to being Euclidean. Indeed, in our sequence of counterexamples we can also keep track of metrics $g_j$ on $B_j(\bOh)$ that are $j^{-1}$-close to Euclidean in $C^2$. In the limit, the $g_j$ converge in $C^2$ to a Euclidean metric on $\RR^{n+1}$. As before, the approximate monotonicity formula guarantees that we still have upper-semicontinuity of densities in the limit as $j \to \infty$, so \eqref{eq:density-assumption} passes to the limit. Likewise, locally in the regular part we have control of the coefficients of the Jacobi field equation under the $C^2$ convergence of $g_j$ to a Euclidean metric, so we can still obtain \eqref{eq:sep-est.bad.jf} in the $\RR^{n+1}$ blow-up limit, leading to the same contradiction.
\end{remark}

\subsection{High-density packing estimate}

We arrive at our main technical result. Below, it will be convenient to denote for integral currents $T$
\begin{equation} \label{eq:calD}
	\cD(T) := \sup \{ \Theta_{T}(\bx) : \bx \in \sing T \},
\end{equation}
with the convention that $\sup \emptyset = -\infty$.

\begin{theorem}[Discrete high-density packing estimate] \label{theo:high.density.packing}
	Let $n+1 \in \{ 8, 9, 10 \}$. There exists $\eta = \eta(n) \in (0, 1)$ with the following property:
	
	Suppose $(\breve N, g)$ is an $(n+1)$-dimensional Riemannian manifold without boundary. Let $\Omega$ be a compact subdomain of $\breve N$ with smooth boundary $\partial \Omega \neq \emptyset$, and let $\cT \subset \II^1_n(\breve N)$ be a set of minimizers that are homologous in $(\Omega, \partial \Omega)$ and where every $T \in \cT$ has the following properties:
	\begin{enumerate}
		\item[(a)] $\partial T = \llbracket \Gamma(T) \rrbracket$ for a nonempty, smooth, two-sided $\Gamma(T) \subset \partial \Omega$, 
		\item[(b)] $\spt T \subset \Omega$ is connected, and $\spt T \cap \partial \Omega = \Gamma(T)$ smoothly and transversely.
	\end{enumerate}
	Let $m \in \NN$, $m \geq 2$, $\theta \in (0, 1)$, $T \in \cT$. Take $D>0$ smaller than the normal injectivity radius of $\Gamma(T)\subset \partial\Omega$. There is $\sigma = \sigma(m,\theta,D,T)>0$ and for any $\nu>0$ there is $K_0 = K_0(m, \theta, D, T, \nu)>0$ so that the following holds. Consider any list
	\[ T =: T_0, T_1, \ldots, T_K \in \cT \]
	with $K \geq K_0$, satisfying
	\begin{enumerate}
		\item[(c)] $\Gamma(T_k) = \Graph_{\Gamma(T)} h_k$ for a smooth $h_k : \Gamma(T) \to \RR$ satisfying $\Vert h_k \Vert_{C^{m,\theta}(\Gamma)} \leq D$, where we're identifying the normal line bundle of $\Gamma(T)$ with $\Gamma \times \RR$ and giving $\Gamma(T_k)$ the graph orientation,
		\item[(d)] $\max_{k = 1, \ldots, K} \dist_{\cF,\Omega}(T, T_k) \leq \sigma$,\footnote{Here, $\dist_{\cF, \Omega}$ is the flat distance in $\Omega$, which is denoted $\cF_\Omega$ in \cite[4.1.24]{Federer:GMT}.}
		\item[(e)] $\min_{k = 1, \ldots, K} \inf_{\Gamma(T)} (h_{k} - h_{k-1}) \geq \nu K^{-1}$,
	\end{enumerate}
	we must have 
	\[ \min_{k = 1, \ldots, K} \cD(T_k) \leq \cD(T) - \eta. \]
\end{theorem}
\begin{proof}
Given Remark \ref{rema:kappa.n} and $n+1 \leq 10$, we can fix\footnote{This is the source of the dimensional restriction in our theorems.}
\begin{equation} \label{eq:packing.lambda0}
	\lambda_0 \in (2, \kappa_n^*+1) \text{ so that } 7-n+\lambda_0 > 0.
\end{equation}
With $\lambda_0$ fixed, let $\eta_1 = \eta_1(n)$, $A_0 = A_0(n)$ be as in Lemma \ref{lemm:sep-est}. Then let $C_0 = C_0(n)$ be such that the following property holds for any ball $B_1(\bOh)$ endowed with a metric that is $\eta_1$-close in $C^2$ to being Euclidean: if $\Pi$ is any subspace of $\RR^{n+1}$ with $\dim \Pi \leq n-7$, and $\{ B_\gamma(\by_j) \}_{j=1}^Q$ is a radius $\gamma \in (0, \tfrac12)$ Vitali cover of the $\gamma$-neighborhood $U_\gamma(\Pi) \cap B_1(\bOh)$, i.e., a cover with pairwise disjoint $B_{\gamma/5}(\by_j)$'s, then:
\begin{equation} \label{eq:packing.Vitali}
	Q \leq C_0 \gamma^{7-n}.
\end{equation}
This follows from elementary packing estimates. Now set
\begin{equation} \label{eq:packing.gamma}
	\gamma := A_0^{-q}
\end{equation}
where $q = q(n)$ is large enough that
\begin{equation} \label{eq:packing.gamma.constraint}
	2 C_0 \gamma^{7-n+\lambda_0} < 1.
\end{equation}
With these choices, let $\eta_0$ be as in Lemma \ref{lemm:cone-split-pre} and let $\eta \in (0, \min \{ \eta_0, \eta_1 \})$ be arbitrary.

Subject to these choices, fix $T\in \cT$. Note that by Allard's regularity theorem there is nothing to prove if $\sing T = \emptyset$, so let us assume $\sing T \neq \emptyset$. We define
\[ \delta_0 := \tfrac12 \inf \{ d(\bx, \partial \Omega) : \bx \in \sing T \}. \]
The Hardt--Simon boundary regularity theorem (Theorem \ref{theo:hardt-simon-boundary-regularity}) implies $\delta_0 > 0$. We now claim that it's possible to choose $\sigma>0$ so that if $T'\in \II_n(\breve N)$ is a minimizer homologous to $T$ in $(\Omega,\partial\Omega)$ with $\dist_{\cF,\Omega}(T,T') \leq \sigma$ and $\partial T'=\llbracket \Gamma(T')\rrbracket$ where $\Gamma(T')\subset \partial\Omega$ is the normal graph of $h \in C^{m,\theta}(\Gamma(T))$ with $\Vert h \Vert_{C^{m,\theta}(\Gamma)} \leq D$ then
\begin{equation} \label{eq:packing.density.limit.constraints}
	T' \in \II_n^1(\breve N)
\end{equation}
and  $\spt T'\subset \Omega$ is connected and intersects $\partial\Omega$ transversely,
\begin{equation} \label{eq:packing.density.sing.interior}
	\inf \{ d(\bx, \partial \Omega) : \bx \in \sing T' \} \geq \delta_0,
\end{equation}
\begin{equation} \label{eq:packing.density.uniform.pinching}
	\cD(T') < \cD(T) + \eta.
\end{equation}
The existence of $\sigma>0$ follows immediately. As $\sigma\to 0$ any such $T'$ converges weakly to $T$ and the boundary curves converge in $C^{m,\theta'}$. Thus, we can apply the Hardt--Simon boundary regularity theorem (Theorem \ref{theo:hardt-simon-boundary-regularity}), the Allard boundary regularity theorem, and upper semicontinuity of density under flat convergence to conclude. 

We now fix $\nu>0$ and show it's possible to find $K_0=K_0(m,\theta,D,T,\nu)$ with the given properties. If this were to fail, then there would exist a sequence $K_\ell \in \NN$ with $K_\ell \to \infty$ and for each $\ell=1,2,\dots$ and a family 
\[ T =: T_{\ell,0}, T_{\ell,1}, \ldots, T_{\ell,K_\ell} \in \cT \]
of  $K_\ell+1$ elements, so that $\Gamma(T_{\ell,k}) = \Graph_{\Gamma(T)} h_{\ell,k}$ with 
\begin{equation} \label{eq:packing.Tl.close}
	\max_{k = 1, \ldots, K_\ell} \dist_{\cF,\Omega}(T, T_{\ell,k}) \leq \sigma,
\end{equation}
\begin{equation} \label{eq:packing.hl.bounded}
	\max_{k = 1, \ldots, K_\ell} \Vert h_{\ell,k} \Vert_{C^{m,\theta}(\Gamma)} \leq D,
\end{equation}
\begin{equation} \label{eq:packing.hl.far}
	\min_{k = 1, \ldots, K_\ell} \inf_{\Gamma(T)} (h_{\ell,k} - h_{\ell,k-1}) \geq \nu K_\ell^{-1},
\end{equation}
\begin{equation} \label{eq:packing.density.pinched}
	\min_{k = 1, \ldots, K_\ell} \cD(T_{\ell,k}) > \cD(T) - \eta.
\end{equation}
Note that we automatically have:
\begin{equation} \label{eq:packing.disjoint.support}
	T_{\ell,0}, T_{\ell,1}, \ldots, T_{\ell,K_\ell} \text{ have pairwise disjoint supports}
\end{equation}
by virtue of Lemma \ref{lemm:disjoint} applied with $T := T_{\ell,k}$, $T' := T_{\ell,k'}$, $\Gamma := \Gamma(T_{\ell,k})$, $\Gamma' := \Gamma(T_{\ell,k'})$, for $0 \leq k < k' \leq K_\ell$, and with $\Sigma$ being the lateral region bounded by $\Gamma(T_{\ell,k}), \Gamma(T_{\ell,k'})$. 

\begin{claim} \label{clai:packing.density.uniform.O1.pinching}
	There exists $r_0 \in (0, \tfrac12 \delta_0)$ such that
	\begin{equation} \label{eq:packing.density.uniform.O1.pinching}
		\Theta_{T_{\ell,k}}(\bx, r) < \cD(T) + 2 \eta
	\end{equation}
	for all $r \in (0, r_0)$, $\ell \in \NN$, $k = 1, \ldots, K_\ell$, and $\bx \in \Omega$ with $d(\bx, \partial \Omega) \geq \tfrac12 \delta_0$.
\end{claim}
\begin{proof}[Proof of claim]
	Suppose \eqref{eq:packing.density.uniform.O1.pinching} failed with $r_i \to 0$. Then we would have
	\[ \Theta_{T_{\ell_i,k_i}}(\bx_i, r_i) \geq \cD(T) + 2 \eta. \]
	for $\bx_i \in \Omega$ with $d(\bx_i, \partial \Omega) \geq \tfrac12 \delta_0$, and $\ell_i \in \NN$, $k_i \in \{ 1, \ldots, K_{\ell_i} \}$. 
	
	By choice of $\sigma>0$ above, we can pass to a subsequence (not relabeled) along which $\bx_{\ell_i} \to \bx$ with $d(\bx, \partial \Omega) \geq \tfrac12 \delta_0$ and $T_{\ell_i,k_i} \rightharpoonup T'$ with $T'$ satisfying \eqref{eq:packing.density.limit.constraints}, \eqref{eq:packing.density.sing.interior}, \eqref{eq:packing.density.uniform.pinching}. It follows from the weak convergence and the approximate monotonicity formula on manifolds that
	\[ \Theta_{T'}(\bx, r) \geq \cD(T) + \tfrac32 \eta \] 
	for all small $r > 0$. But $\Theta_{T'}(\bx, r) \to \Theta_{T'}(\bx)$ as $r \to 0$, contradicting \eqref{eq:packing.density.uniform.pinching}.
\end{proof}

Define, for each $\ell \in \NN$,
\[ \cS_\ell := \cup_{k=0}^{K_\ell} \{ \bx \in \sing T_{\ell,k} : \Theta_{T_{\ell,k}}(\bx) \geq \cD(T) - \eta \}. \]

\begin{claim} \label{clai:packing.covering.cardinality}
	Fix $\ell \in \NN$. We can inductively define, for each $m = 0, 1, \ldots$, a covering
	\begin{equation} \label{eq:packing.covering}
		\cS_\ell \subset \cup_{j=1}^{Q_{\ell,m}} B_{\gamma^m}(\by_{\ell,m,j})
	\end{equation}
	such that $\by_{\ell,m,j} \in \cS_\ell$ and $Q_{\ell,m} \leq Q C_0^m \gamma^{m(7-n)}$, where $Q$ does not depend on $\ell$ or $m$.
\end{claim}
\begin{proof}[Proof of claim]
	Let $\operatorname{inj}_{(\breve N, g)}(\bx)$ denote the injectivity radius at $\bx \in (\breve N, g)$. Fix $m_0$ large enough that 
	\begin{equation} \label{eq:packing.covering.cardinality.m0.1}
		\gamma^{m_0} < \tfrac12 \min  \{r_0, \delta_0 \eta, \min_\Omega \operatorname{inj}_{(\breve N, g)}(\cdot) \},
	\end{equation}
	and that for every $\bx \in \Omega$, $m \geq m_0$,
	\begin{equation} \label{eq:packing.covering.cardinality.m0.2}
		(\Dil_{\bx, \gamma^m})_\# (B_{\gamma^m \eta^{-1} }(\bx), g) \text{ is $\eta$-close in $C^2$ to a Euclidean ball}.
	\end{equation}
	
	For $m = 0, 1, \ldots, m_0$, take the covering to be any Vitali covering of $\cS_\ell$ with balls of radius $\gamma^m$. The bound on $Q_{\ell,m}$ is a trivial packing estimate.
	
	Now for the inductive step, assuming we have \eqref{eq:packing.covering} for some $m\geq m_0+1$, it suffices to show for every ball $B_{\gamma^m}(\by_{\ell,m,j})$ as in \eqref{eq:packing.covering} that $\cS_\ell \cap B_{\gamma^m}(\by_{\ell,m,j})$ can be covered by $\leq C_0 \gamma^{7-n}$ balls of radius $\gamma^{m+1}$ and centered on points in $\cS_\ell$. 
	
	So fix any one such $B_{\gamma^m}(\by)$, where the indices have been suppressed. By definition,
	\[ \by \text{ is a center of a ball in the covering} \implies \by \in \sing T_{\ell,k} \]
	for some $k \in \{ 0, 1, \ldots, K_\ell \}$. Rescale $\breve N' := (\Dil_{\by,\gamma^m})_{\#} \breve N$ so that
	\[ (\Dil_{\by,\gamma^m})_{\#} B_{\gamma^m}(\by) = B_1(\bOh). \]
	It follows from \eqref{eq:packing.covering.cardinality.m0.1} that, in this rescaled setting, 
	\[ \tilde T := (\Dil_{\by,\gamma^m})_{\#} T_{\ell,k} \]
	has $\spt \partial \tilde T \cap B_{\eta^{-1}}(\bOh) = \emptyset$ as well as, by Claim \ref{clai:packing.density.uniform.O1.pinching} and \eqref{eq:packing.density.pinched},
	\begin{align*}
		\Theta_{\tilde T}(\bOh) 
			& = \Theta_{T_{\ell,k}}(\by) \\
			& \geq \cD(T) - \eta \\
			& \geq \Theta_{T_{\ell,k}}(\by, \gamma^m) - 3 \eta \\
			& = \Theta_{\tilde T}(\bOh, 1) - 3 \eta.
	\end{align*}
	 Let $\tilde T_c$ denote the restriction of $\tilde T \mres B_{\eta^{-1}}(\bOh)$ to the connected component of its support containing $\bOh$. Then:
	\[ \Theta_{\tilde T_c}(\bOh) = \Theta_{\tilde T}(\bOh) \geq \Theta_{\tilde T}(\bOh, 1) - 3 \eta \geq \Theta_{\tilde T_c}(\bOh, 1) - 3 \eta. \]
	 Now suppose $\by' \in \cS_\ell \cap B_{\gamma^m}(\by)$. Then 
	\[ \by' \in \cS_\ell \implies \by' \in \sing T_{\ell,k'} \]
	for some $k' \in \{ 0, 1, \ldots, K_\ell \}$. Setting
	\[ \tilde T' := (\Dil_{\by, \gamma^m})_{\#} T_{\ell,k'} \]
	we still have $\spt \partial \tilde T' \cap B_{\eta^{-1}}(\bOh) = \emptyset$ and, at $\bx' := \Dil_{\by,\gamma^m}(\by') \in B_1(\bOh)$:
	\begin{align*}
		\Theta_{\tilde T'}(\bx')
			& = \Theta_{T_{\ell,k'}}(\by') \\
			& \geq \cD(T) - \eta \\
			& \geq \Theta_{\tilde T}(\bOh, 1) - 3 \eta.
	\end{align*}
	 Again, let $\tilde T'_c$ denote the restriction of $\tilde T' \mres B_{\eta^{-1}}(\bOh)$ to the connected component of its support containing $\bx'$, so that
	\[ \Theta_{\tilde T'_c}(\bx') = \Theta_{\tilde T'}(\bx') \geq \Theta_{\tilde T}(\bOh, 1) - 3 \eta \geq \Theta_{\tilde T_c}(\bOh, 1) - 3 \eta. \]
	 Now \eqref{eq:packing.disjoint.support}, \eqref{eq:packing.Vitali}, \eqref{eq:packing.covering.cardinality.m0.1}, \eqref{eq:packing.covering.cardinality.m0.2}, and Lemma \ref{lemm:cone-split-pre} applied  to $\tilde T_c$, $\tilde T'_c$, guarantee that $\bx'$ is contained in the distance $< \gamma$ neighborhood of a fixed $\leq n-7$ dimensional subspace. Thus, after unrescaling, $\cS_\ell \cap B_{\gamma^m}(\by)$ can be covered by $\leq C_0 \gamma^{7-n}$ balls of radius $\gamma^{m+1}$ centered at points of $\cS_\ell$. 
\end{proof}

\begin{claim} \label{clai:packing.superlinear}
	For sufficiently large $m$ and $\ell$, if $B_{\gamma^{m}}(\by)$ is a ball in the $m$-th level covering of $\cS_\ell$ defined in Claim \ref{clai:packing.covering.cardinality}, $\by' \in \cS_\ell \cap B_{\gamma^m}(\by)$, $\by \in \sing T_{\ell,k}$, $\by' \in \sing T_{\ell,k'}$, then
	\begin{equation} \label{eq:packing.superlinear}
		|k-k'| < 2^m \gamma^{m\lambda_0} K_\ell.
	\end{equation}
\end{claim}
\begin{proof}[Proof of claim]
	Assume, for contradiction, that \eqref{eq:packing.superlinear} fails as $m_i \to \infty$, $\ell_i \to \infty$. Let
	\[ \by_i \in \cS_{\ell_i}, \; \by_i \in \sing T_{\ell_i,k_i}, \; \by_i' \in \cS_\ell \cap B_{\gamma^{m_i}}(\by_i), \; \by_i' \in \sing T_{\ell_i,k_i'} \]
	be points violating \eqref{eq:packing.superlinear}, i.e., with
	\begin{equation} \label{eq:packing.superlinear.false}
		|k_i-k_i'| \geq 2^{m_i} \gamma^{m_i \lambda_0} K_{\ell_i}.
	\end{equation}
	Note that $d_g(\by_i, \by_i') \to 0$. After passing to a subsequence (not relabeled) we have $T_{\ell_i,k_i}$, $T_{\ell_i,k_i'} \rightharpoonup T'$ with $T'$ satisfying \eqref{eq:packing.density.limit.constraints}, \eqref{eq:packing.density.sing.interior}, \eqref{eq:packing.density.uniform.pinching}. By Lemma \ref{lemm:reg-scale-cts} and Allard's boundary regularity theorem $\spt T_{\ell_i,k_i} \setminus \sing T_{\ell_i,k_i}$ and $\spt T_{\ell_i,k_i'} \setminus \sing T_{\ell_i,k_i}$ become graphical over $\spt T' \setminus \sing T'$ locally in compact sets as $i \to \infty$.\footnote{Technically, we need to bend our Fermi coordinates near $\partial \Omega$, but this doesn't affect the argument so we omit the standard details.} Let $h'_{\ell_i,k_i}$, $h'_{\ell_i,k_i'}$ be the corresponding height functions defined on subsets of $\spt T' \setminus \sing T'$ that exhaust this set as $i \to \infty$. By \eqref{eq:packing.disjoint.support}, we can arrange $h'_{\ell_i,k_i} < h'_{\ell_i,k_i'}$ on their common domain after perhaps swapping sheets. By \eqref{eq:packing.hl.far}, \eqref{eq:packing.density.limit.constraints}, \eqref{eq:packing.superlinear.false}, we have for a uniform $c > 0$ that
	\begin{equation} \label{eq:packing.superlinear.bdry}
		\min_{\Gamma(T')} (h'_{\ell_i,k_i'} - h'_{\ell_i,k_i}) \geq c \nu (k_i' - k_i) K_{\ell_i}^{-1} \geq c \nu 2^{m_i} \gamma^{m_i\lambda_0}.
	\end{equation}
	Let $m_0$ again be chosen to satisfy both \eqref{eq:packing.covering.cardinality.m0.1}, \eqref{eq:packing.covering.cardinality.m0.2}. By  \eqref{eq:packing.disjoint.support}, \eqref{eq:packing.covering.cardinality.m0.1}, \eqref{eq:packing.covering.cardinality.m0.2}, we can apply Lemma \ref{lemm:sep-est} iteratively until scale $B_{\gamma^{m_0}}(\by_i)$ (again, restricting to connected components in each scale as in the proof of Claim \ref{clai:packing.covering.cardinality}, and noting that the distance decrease by the appearance of a new connected component when going to a bigger scale gives a distance inequality in our favor), obtaining
	\begin{equation} \label{eq:packing.superlinear.intr}
		d(\cR_{\geq \rho_2 \gamma^{m_0}}(T_{\ell_i,k_i}) \cap \partial B_{\gamma^{m_0}}(\by_i), \spt T_{\ell_i,k_i'}) \leq 2 \gamma^{(m_i-m_0) \lambda_0}.
	\end{equation}
	It follows from $2^{m_i} \to \infty$, Lemma \ref{lemm:reg-scale-cts}, \eqref{eq:packing.Tl.close}, \eqref{eq:packing.superlinear.bdry}, \eqref{eq:packing.superlinear.intr}, the Harnack inequality and Schauder theory (the regular kind) up to $\Gamma(T')$ (where our parametrization is regular by \eqref{eq:packing.density.limit.constraints}) that a renormalization of $h'_{\ell_i,k_i'} - h'_{\ell_i,k_i}$ yields a nonnegative Jacobi field on $\spt T' \setminus \sing T'$ that is positive on part of $\Gamma(T')$, and thus on part of $\reg T'$, but also has an interior vanishing point in $\reg T'$. This violates the strict maximum principle since $\reg T'$ is connected (by Lemma \ref{lemm:reg.connected} applied away $\Gamma(T')$). Thus, \eqref{eq:packing.superlinear.false} is false, so \eqref{eq:packing.superlinear} holds whenever $m$ and $\ell$ are sufficiently large. This completes the proof of Claim \ref{clai:packing.superlinear}.
\end{proof}

To complete the proof of Theorem \ref{theo:high.density.packing}, pass to a subsequence $\{ \ell_m \}$ of $\{ \ell_i \}$ which diverges quickly enough as $m \to \infty$ so that
\begin{equation} \label{eq:packing.ell.m}
	\tfrac12 K_{\ell_m} + 1 > Q (C_0 \gamma^{7-n})^m,
\end{equation}
where $C_0$ is as in \eqref{eq:packing.Vitali} and $Q$ is as in Claim \ref{clai:packing.covering.cardinality}. 

For each ball $B_{\gamma^{m}}(\by_{\ell_m,m,j})$, $j \in \{ 1, \ldots, Q_{\ell_m,m} \}$, in the cover of Claim \ref{clai:packing.covering.cardinality}, let
\[ \cK_{m,j} := \{ k \in \{ 0, 1, \ldots, K_{\ell_m} \} : \sing T_{\ell_m,k} \cap \cS_{\ell_m} \cap B_{\gamma^{m}}(\by_{\ell_m,m,j}) \neq \emptyset \}. \]
It follows from \eqref{eq:packing.density.pinched} that
\begin{equation} \label{eq:packing.all.covered}
	\cup_{j=1}^{Q_{\ell_m,m}} \cK_{m,j} = \{ 0, 1, \ldots, K_{\ell_m} \}.
\end{equation}
It also follows from Claim \ref{clai:packing.superlinear} that, for all sufficiently large $m$, 
\begin{equation} \label{eq:packing.nearby.covered}
	\# \cK_{m,j} \leq 1 + 2 \lfloor 2^{m} \gamma^{m \lambda_0} K_{\ell_m} \rfloor \leq 1 + 2^{m+1} \gamma^{m \lambda_0} K_{\ell_m},
\end{equation}
where $\#$ above denotes the cardinality of the set. Together, \eqref{eq:packing.all.covered} and \eqref{eq:packing.nearby.covered} imply
\[ K_{\ell_m} + 1 \leq (1 + 2 ^{m+1} \gamma^{m \lambda_0} K_{\ell_m}) Q_{\ell_m,m}. \]
From Claim \ref{clai:packing.covering.cardinality} we deduce
\[ K_{\ell_m} + 1 \leq ((C_0 \gamma^{7-n})^m + 2 (2 C_0 \gamma^{7-n+\lambda_0})^m K_{\ell_m}) Q. \]
Recalling \eqref{eq:packing.gamma.constraint}, we can take $m$ large so that $2 (2 C_0 \gamma^{7-n+\lambda_0})^m Q < \frac 12$, and thus
\[ \tfrac12 K_{\ell_m} + 1 \leq (C_0 \gamma^{7-n})^m Q. \]
This violates \eqref{eq:packing.ell.m} and completes our proof of Theorem \ref{theo:high.density.packing}.
\end{proof}

\section{Theorem \ref{theo:main.plateau}: Generic smoothness of Plateau solutions} \label{sec:main.plateau}

We will prove the following result that implies Theorem \ref{theo:main.plateau}.

\begin{theorem} \label{theo:main.plateau.baire}
	Let $\Gamma^\circ$ be a smooth, closed, oriented, $(n-1)$-dimensional submanifold of $\RR^{n+1}$ with $n+1 \in \{ 8, 9, 10 \}$. Denote by $\cG^\circ$ a small enough neighborhood of the trivial section of the normal bundle of $\Gamma^\circ$ whose elements can be identified with submanifolds near $\Gamma^\circ$. Then, there exists an open and dense subset $\cG^\circ_* \subset \cG^\circ$ with the property that for every $\Gamma \in \cG^\circ_*$ there exists a minimizing $T \in \II^1_n(\RR^{n+1})$ with $\partial T = \llbracket \Gamma \rrbracket$ that in fact has the form $T = \llbracket M \rrbracket$ for a smooth, compact, oriented hypersurface $M$ with $\partial M = \Gamma$.
\end{theorem}

The theorem will follow from the combination of two technical lemmas, which we state below but whose proofs we briefly postpone. 

\begin{lemma} \label{lemm:main.plateau.uniqueness.shrinking}
	Fix $\Gamma$ as in Theorem \ref{theo:main.plateau.baire} and assume that $T \in \II_n(\RR^{n+1})$ is minimizing with $\partial T = \llbracket \Gamma \rrbracket$ and $\spt T$ is connected.  Then, for every neighborhood $\cG$ of $\Gamma$, there exists a $\Gamma' \in \cG$ and a uniquely\footnote{Recall that $T'$ is uniquely minimizing if $\tilde T' \in \II_n(\RR^{n+1})$ has $\partial \tilde T'=\llbracket \Gamma'\rrbracket$ then $\MM(T') \leq \MM(\tilde T')$ with equality if and only if $\tilde T'=T'$.} minimizing $T' \in \II_n(\RR^{n+1})$ with $\partial T' = \llbracket \Gamma' \rrbracket$ and which can be guaranteed to have $\spt T' \subsetneq \spt T$ connected. 
	
	Our construction is such that $T' \in \II_n^1(\RR^{n+1})$ if and only if $T \in \II_n^1(\RR^{n+1})$. Likewise, if $\spt T$ is a smooth hypersurface with compact boundary, then our construction is such that $\spt T'$ is smooth and strictly stable under Dirichlet boundary conditions.\footnote{We mean this in the sense of the second varation being positive definite for normal variations that vanish along the boundary.}
\end{lemma}

(See also the recent work of \cite{CaldiniMarcheseMerloSteinbruchel}.)

\begin{lemma} \label{lemm:main.plateau.main.lemma}
	Fix $\Gamma$ as in Theorem \ref{theo:main.plateau.baire} and assume that $T \in \II^1_n(\RR^{n+1})$ is uniquely minimizing with $\partial T = \llbracket \Gamma \rrbracket$ and $\spt T$ is connected. Then, for every neighborhood $\cG$ of $\Gamma$, there exists a $\Gamma' \in \cG$ and a minimizing $T' \in \II_n^1(\RR^{n+1})$ with $\partial T' = \llbracket \Gamma' \rrbracket$ so that, in fact, $\spt T'$ is a smooth, compact, connected, oriented hypersurface with boundary $\Gamma'$ and $\spt T' \cap \spt T = \emptyset$.
\end{lemma}

Let us see how these two lemmas imply Theorem \ref{theo:main.plateau.baire}.

\begin{proof}[Proof of Theorem \ref{theo:main.plateau.baire} assuming Lemmas \ref{lemm:main.plateau.uniqueness.shrinking}, \ref{lemm:main.plateau.main.lemma}]
	Suppose that $\Gamma^\circ$ has connected components $\Gamma^\circ_1, \ldots, \Gamma^\circ_q$. We'll retain this enumeration of the components going forward. It'll also be convenient to agree that we'll write $\Gamma = (\Gamma_1, \ldots, \Gamma_q)$ for any $\Gamma \in \cG^\circ$.
	
	For $\cP \subset \{ 1, \ldots, q \}$ denote by
	\[ \cG^\circ_{*,\cP} \subset \cG^\circ \]
	the set of all $\Gamma = (\Gamma_1, \ldots, \Gamma_q) \in \cG^\circ$ with the following defining property: \textbf{if} there exists $T \in \II_n(\RR^{n+1})$ minimizing with $\partial T = \sum_{i \in \cP} \llbracket \Gamma_i \rrbracket$ and satisfying
	\begin{enumerate}
		\item[(C$_0$)] $\spt T$ is connected, 
	\end{enumerate}
	\textbf{then} $T$ also satisfies:
	\begin{enumerate}
		\item[(C$_1$)] $T \in \II_n^1(\RR^{n+1})$, \textit{and}
		\item[(C$_2$)] $\spt T$ is smooth hypersurface with boundary, \textit{and}
		\item[(C$_3$)] $\spt T$ is strictly stable under Dirichlet boundary conditions, \textit{and}
		\item[(C$_4$)] $T$ is \textit{uniquely} minimizing.
	\end{enumerate}
	Note that:
	\begin{itemize}
		\item By Corollary \ref{coro:hardt-simon-decomposition}, if $T$ satisfies (C$_1$) and (C$_2$) then it is of the form we need for our theorem.
		\item If a single minimizing $T$ with $\partial T = \sum_{i \in \cP} \llbracket \Gamma_i \rrbracket$ satisfies (C$_0$)-(C$_4$), then $\Gamma \in \cG^\circ_{*,\cP}$ due to (C$_4$).
	\end{itemize}
		
	\begin{claim} \label{clai:main.plateau.baire.open}
		For all $\cP \subset \{ 1, \ldots, q \}$, $\cG^\circ_{*,\cP}$ is open in $\cG^\circ$.
	\end{claim}
	\begin{proof}[Proof of claim]
		We argue by contradiction. Suppose that $\Gamma \in \cG^\circ_{*,\cP}$ and that there exists a sequence (where the index is suppressed) $\Gamma' \to \Gamma$ in the smooth topology with $\Gamma' \in \cG^\circ \setminus \cG^\circ_{*,\cP}$. For each $\Gamma'$, there exists a minimizing $T' \in \II_n(\RR^{n+1})$ with $\partial T' = \sum_{i \in \cP} \llbracket \Gamma_i' \rrbracket$ and $T'$ satisfying (C$_0$) but failing (C$_1$) or (C$_2$) or (C$_3$) or (C$_4$) above. 
		
		By passing to subsequences we can assume that $T' \rightharpoonup T \in \II_n(\RR^{n+1})$ that is minimizing with $\partial T = \sum_{i \in \cP} \llbracket \Gamma_i \rrbracket$. Note that since each $T'$ satisfies (C$_0$), $T$ does too by 
		the upper semicontinuity of density, and since $\Gamma \in \cG^\circ_{*,\cP}$, $T$ then also satisfies (C$_1$)-(C$_4$).
		
		It also follows from the upper semicontinuity of density and $T$ satisfying (C$_1$) that, up to discarding finitely many terms in the sequence, the $T'$ all satisfy (C$_1$) as well. Since $T$ also satisfies (C$_2$), it then follows from Allard's interior and boundary regularity theorem together with the upper semicontinuity of densities that the convergence $\spt T' \to \spt T$ is smooth up to the boundary; in particular, up to discarding finitely many terms in the sequence, the $T'$ also all satisfy (C$_2$) as well. We now conclude by showing that neither (C$_3$) nor (C$_4$) can fail for infinitely many $T'$.
		
		First suppose that infinitely many $T'$ fail (C$_3$). Pass to such a subsequence. Then each $\spt T'$ is degenerate and carries nontrivial Jacobi fields that vanish on the boundary. Because of the smooth convergence $\spt T' \to \spt T$, these yield (after the standard normalization procedure) nontrivial Jacobi fields on $\spt T$ that vanish on the boundary. Thus, $\spt T$ is degenerate, a contradiction to $T$ satisfying (C$_3$).
		
		It must therefore be that infinitely many $T'$ fail (C$_4$). Pass to such a subsequence. Let $\tilde T'$ be any competing minimizer with respect to $g'$ that is distinct from $T'$. We have $\tilde T' \rightharpoonup T$ after passing to a further subsequence because $T$ satisfies (C$_4$). By the argument above (it only relies on the limit $T$ satisfying (C$_1$) and (C$_2$)) the convergence of the supports is again smooth and graphical with multiplicity one. Using the standard construction in the setting of a graphical convergence $\spt T'$, $\spt \tilde T' \to \spt T$, we produce (after the standard normalization procedure) a nontrivial Jacobi field on $\spt T$ that vanishes on the boundary. Thus, $\spt T$ is degenerate, again a contradiction to $T$ satisfying (C$_3$).

		Since all cases lead to a contradiction, $\cG^\circ_{*,\cP}$ is indeed open.
	\end{proof}
	
	\begin{claim} \label{clai:main.plateau.baire.dense}
		For all $\cP \subset \{ 1, \ldots, q \}$, $\cG^\circ_{*,\cP}$ is dense in $\cG^\circ$.
	\end{claim}
	\begin{proof}[Proof of claim]
		We prove this claim by induction on the number of elements of $\cP$. The base case ($\cP = \emptyset$) is trivial, so we proceed to the inductive step. 
		
		We fix $\emptyset \neq \cP \subset \{ 1, \ldots, q \}$ and assume that $\cG^\circ_{*,\cP'}$ is dense in $\cG^\circ$ for all $\cP' \subset \{ 1, \ldots, q \}$ with $\# \cP' < \# \cP$. Let $\Gamma \in \cG^\circ$. We show there exist elements of $\cG^\circ_{*,\cP}$ arbitrarily close to $\Gamma$. 	Note that we know, by our inductive hypothesis, that
		\begin{equation} \label{eq:main.plateau.baire.dense.induction}
			\bigcap_{\cP' \subsetneq \cP} \cG^\circ_{*,\cP'} \text{ is open and dense in } \cG^\circ
		\end{equation}
		since each $\cG^\circ_{*,\cP'}$ with $\cP' \subsetneq \cP$ is already open and dense. Perturb $\Gamma$ (without relabeling) so that
		\begin{equation} \label{eq:main.plateau.baire.dense.setup}
			\Gamma \in \bigcap_{\cP' \subsetneq \cP} \cG^\circ_{*,\cP'}
		\end{equation}
		Now if $\Gamma \in \cG^\circ_{*,\cP}$, then we're done perturbing. So assume $\Gamma \not \in \cG^\circ_{*,\cP}$ going forward.
		
		Let $T \in \II_n(\RR^{n+1})$ be a minimizer with $\partial T = \sum_{i \in \cP} \llbracket \Gamma_i \rrbracket$, satisfying (C$_0$), but failing (C$_1$) or (C$_2$) or (C$_3$) or (C$_4$). 
		
		\textit{Iterative improvement}. We repeat the following process finitely many times to ensure $T$ eventually satisfies (C$_1$). In particular, if $T$ already satisfies (C$_1$), proceed to ``\textit{Final perturbation}'' below. Let
		\[ T = \sum_{i=1}^m \llbracket M^{(i)} \rrbracket \]
		be the nested decomposition given by Corollary \ref{coro:hardt-simon-decomposition}. Note that $M^{(1)}$ is connected too since $\bar M^{(1)} = \spt T$ is connected and Lemma \ref{lemm:reg.connected} applies to $\llbracket M^{(1)} \rrbracket$ away from $\partial M^{(1)}$.
		
		Note that $m \geq 2$ since $T$ doesn't satisfy (C$_1$), i.e., $T \not \in \II_n^1(\RR^{n+1})$. We claim that in fact $m=2$. Write
		\[ T_1 = \llbracket M^{(1)} \rrbracket, \; T_2 = \llbracket M^{(2)} \rrbracket. \]
		Note that $T_1$, $T_2$, and $T_1 + T_2$ are individually minimizing since, by the nested nature of the decomposition, $\Vert T_1 \Vert, \; \Vert T_2 \Vert, \; \Vert T_1 + T_2 \Vert \leq \Vert T \Vert$ so \cite[Lemma 33.4]{Simon:GMT} applies (given the minimizing nature of $T$). Also, $T_1$ and $T_1 + T_2$ satisfy (C$_0$) since $T$ does and, again by the nestedness of the decomposition, $\spt T = \spt T_1 = \spt (T_1 + T_2)$. Let
		\[ \emptyset \neq \cP_1, \; \cP_2 \subset \cP \]
		denote the set of indices of $\Gamma_i \in \cP$ occurring as boundaries of $T_1$, $T_2$, respectively. Note that $\cP_1 \cup \cP_2$ index the boundary components of $T_1 + T_2$ (by the nestedness of the decomposition) and $\cP_1 \cap \cP_2 = \emptyset$ (because our boundaries are assigned multiplicity one). Now since $T_1 + T_2$ manifestly does not satisfy (C$_1$), it follows from \eqref{eq:main.plateau.baire.dense.setup} that
		\[ \cP_1 \cup \cP_2 = \cP \implies T = T_1 + T_2 \text{ and } m=2. \]
		From this we deduce that $\cP_1 \subsetneq \cP$. In view of \eqref{eq:main.plateau.baire.dense.setup} and $T_1$ satisfying (C$_0$), it follows from our setup that $T_1$ satisfies (C$_1$)-(C$_4$) too. Even though $\spt T_2$ might be disconnected, $T_2$ also satisfies (C$_1$)-(C$_4$) given the nested nature of the decomposition.
		
		Apply Lemma \ref{lemm:main.plateau.main.lemma} to perturb $(\cup_{i \in \cP_1} \Gamma_i, T_1) \mapsto (\cup_{i \in \cP_1} \Gamma_i', T_1')$ with
		\begin{equation} \label{eq:main.plateau.baire.dense.push}
			\spt T_1' \cap \spt T = \emptyset, 
		\end{equation}
		and still within the open set in \eqref{eq:main.plateau.baire.dense.induction}; in particular, $T_1'$ still satisfies (C$_0$)-(C$_4$). Define:
		\[ \Gamma' := (\Gamma \setminus \cup_{i \in \cP_1} \Gamma_i) \cup (\cup_{i \in \cP_1} \Gamma_i'), \]
		and enumerate its components again as $(\Gamma_1', \ldots, \Gamma_q')$ (note that $\Gamma_i'$ remains unambiguously defined if $i \in \cP_1$). 
		
		Let $\tilde T' \in \II_n(\RR^{n+1})$ be a minimizer with $\partial \tilde T' = \sum_{i \in \cP} \llbracket \Gamma'_i \rrbracket$ and which satisfies (C$_0$). If $\tilde T'$ satisfies (C$_1$), we proceed to ``\textit{Final perturbation}'' below with $\Gamma' \mapsto \Gamma$ and $\tilde T' \mapsto T$. 
		
		Otherwise, $\tilde T'$ does not satisfy (C$_1$), i.e., $\tilde T' \not \in \II_n^1(\RR^{n+1})$, and arguing as we did for $T$, we have the nested decomposition
		\[ \tilde T' = \tilde T'_1 + \tilde T'_2 \]
		with $\tilde T'_1, \tilde T'_2 \in \II_n^1(\RR^{n+1})$ and $\tilde T'_1$ satisfying (C$_0$)-(C$_4$) and $\tilde T'_2$ satisfying (C$_1$)-(C$_4$). From this we will obtain a contradiction as long as our perturbation $T_1 \mapsto T_1'$ is sufficiently small. We also write $\tilde \cP'_1, \tilde \cP'_2 \subset \cP$ to contain the indices of components of $\spt \partial \tilde T'_1, \spt \partial \tilde T'_2$. Note that by the upper semicontinuity of density together with the density considerations in Corollary \ref{coro:hardt-simon-decomposition} we have $\tilde \cP_2' \subset \cP_2$ for sufficiently small perturbations $T_1 \mapsto T_1'$. We cannot have $\tilde \cP_2' = \cP_2$ because then it'd follow that $\tilde \cP_1' = \cP_1$ and since $T_1'$ satisfies (C$_4$) we'd have $\tilde T'_1 = T'_1$ and since $T_2$ satisfies (C$_4$) too we'd also have $\tilde T_2' = T_2$, a contradiction to \eqref{eq:main.plateau.baire.dense.push}.
		
		Thus, $\tilde \cP_2' \subsetneq \cP_2$. We can therefore restart our iterative improvement process with $\tilde T'$ in place of $T$ and $\Gamma'$ in place of $\Gamma$. This terminates after finitely many steps because $\cP_2$ is finite.
				
		\textit{Final perturbation}. We may assume, now, that $T$ satisfies (C$_0$) as well as (C$_1$). We then apply the following steps,\footnote{This part of the perturbation argument is identical to the one for manifolds in the next section.} skipping steps (i) and (ii) and proceeding to (iii) with $(\Gamma'', T'') := (\Gamma, T)$ if $T$ satisfies (C$_2$):
		\begin{enumerate}
			\item[(i)] Apply Lemma \ref{lemm:main.plateau.uniqueness.shrinking} to perturb $(\cup_{i \in \cP} \Gamma_i, T) \mapsto (\cup_{i \in \cP} \Gamma_i', T')$ so that $T' \in \II_n^1(\RR^{n+1})$ satisfies (C$_0$) and (C$_1$) and  (C$_4$).
			\item[(ii)] Apply Lemma \ref{lemm:main.plateau.main.lemma} to perturb $(\cup_{i \in \cP} \Gamma_i', T') \mapsto (\cup_{i \in \cP} \Gamma_i'', T'')$ so that $T'' \in \II_n^1(\RR^{n+1})$ satisfies (C$_0$) and (C$_1$) and (C$_2$).
			\item[(iii)] Finally, re-apply Lemma \ref{lemm:main.plateau.uniqueness.shrinking} to perturb $(\cup_{i \in \cP} \Gamma_i'', T'') \mapsto (\cup_{i \in \cP} \Gamma_i''', T''')$ so that $T''' \in \II_n^1(\RR^{n+1})$ satisfies (C$_0$), (C$_1$), (C$_2$), (C$_3$), (C$_4$). 
		\end{enumerate}
		Thus, $T'''$ satisfies (C$_0$)-(C$_4$). As we've discussed this implies $\Gamma''' \in \cG^\circ_{*,\cP}$.
	\end{proof}
	
	It follows from the finiteness of $\cP$ and Claims \ref{clai:main.plateau.baire.open}, \ref{clai:main.plateau.baire.dense} that
	\[ \cG^\circ_* := \bigcap_{\cP \subset \{ 1, \ldots, q \}} \cG^\circ_{*,\cP} \]
	is also open and dense in $\cG^\circ$.
	
	It remains to prove that $\cG^\circ_*$ has the desired smooth minimization property. To that end, let $\Gamma \in \cG^\circ_*$ and suppose that $T \in \II_n(\RR^{n+1})$ is mininimizing with $\partial T = \llbracket \Gamma \rrbracket$. Let $C$ denote any component of $\spt T$. Note that (\cite[Lemma 33.4]{Simon:GMT}) $T_\cP := T \mres C \in \II_n(\RR^{n+1})$ is minimizing with $\partial T_\cP = \llbracket \Gamma_\cP \rrbracket$ for some nonempty $\cP \subset \{ 1, \ldots, q \}$ and $\Gamma_\cP = \cup_{i \in \cP} \Gamma_i$. By construction, $\spt T_\cP$ is connected. It follows from $\Gamma \in \cG^\circ_* \subset \cG^\circ_{*,\cP}$ and the definition of $\cG^\circ_{*,\cP}$ that $\spt T_\cP = C$ is smooth of the desired form. Since the component $C$ of $\spt T$ was arbitrary, it follows that $T$ has the desired form.
\end{proof}

We now return to give proofs for the main lemmas. Below, $U_s(\cdot)$ will always denote a distance $< s$ tubular neighborhood in $\RR^{n+1}$. 

\begin{proof}[Proof of Lemma \ref{lemm:main.plateau.uniqueness.shrinking}]
	This is a consequence of Lemma \ref{lemm:hardt-simon-uniqueness} with (adapting notation from the lemma) $s > 0$ small enough that $\cup_{i=1}^m \Gamma'^{(i)} \in \cG$ and so  that $M'^{(1)} = M^{(1)} \setminus U_s(\Gamma^{(1)})$ is still connected. In case $m=1$ and $M^{(1)}$ is smooth, the strict stability of $M'^{(1)}$ is an easy consequence of the domain monotonicity of Dirichlet eigenvalues for Schr\"odinger operators.
\end{proof}

\begin{proof}[Proof of Lemma \ref{lemm:main.plateau.main.lemma}]
	We will apply an iterative improvement argument with $\cG$ fixed (as given). 
	
	\textit{Iterative improvement}. It follows from the Hardt--Simon boundary regularity theorem (particularly the decomposition in Corollary \ref{coro:hardt-simon-decomposition}, writing $M = M^{(1)}$ below) that
	\[ \spt T = M \cup E \]
	for a smooth, connected, oriented $M$ with boundary $\partial M = \Gamma$, and with $E = \bar M \setminus M \subset \RR^{n+1}$ compact with Hausdorff dimension $\leq n-7$, and since $T \in \II_n^1(\RR^{n+1})$, we also have $T = \llbracket M \rrbracket$, i.e., $T$ is the multiplicity-one current defined by $M$. 
	
	From the compactness of $\Gamma$, there exists $\gamma > 0$ such that $\bar U_{\gamma}(\Gamma) \cap E = \emptyset$. From the argument in the proof of Lemma \ref{lemm:main.plateau.uniqueness.shrinking} we know that, if $s \in (0, \gamma)$ is sufficiently small and we set
	\begin{equation} \label{eq:main.plateau.main.lemma.U.Sigma}
		U := U_s(\Gamma), \;  \Sigma := \partial U_s(\Gamma),
	\end{equation}
	then
	\[ T' := T \mres (\RR^{n+1} \setminus \bar U) \in \II_n^1(\RR^{n+1}) \]
	is uniquely minimizing with $\partial T' = \llbracket \Gamma' \rrbracket$, where $\Gamma' := M \cap \Sigma \in \cG$. Now set $\Gamma'_0 := \Gamma'$ and take a unit speed foliation
	\[ (\Gamma'_t)_{t \in [0, \delta]} \subset \Sigma, \]
	where $\delta > 0$ is small enough that $\Gamma'_t \in \cG$ for $t \in [0, \delta]$. For $t \in [0, \delta]$, let $T'_t \in \II_n(\RR^{n+1})$ be minimizing with $\partial T'_t = \partial \llbracket \Gamma'_t \rrbracket$. In what follows, it is important that the unit speed foliation be taken with consistently pointing unit normal vectors to $M$, which we can do because $M$ is connected and oriented. We also take care to choose consistently oriented unit normals (near the boundary, which is only $C^\infty$-perturbed in each step) for all steps in the iterative improvement to ensure our families are consistently monotone.

	Since $T' \in \II_n^1(\RR^{n+1})$ and is uniquely minimizing, we must have $T'_0 = T'$, as well as
	\begin{equation} \label{eq:main.with.uniqueness.limit}
		T'_t \in \II_n^1(\RR^{n+1}) \text{ and } T'_t \rightharpoonup T'_0 \text{ as } t \to 0,
	\end{equation}
	the former holding true by upper semicontinuity of densities. As before, $\spt T'_t = M'_t \cup E'_t$, where $M'_t$ is a (not yet known to be connected) smooth, oriented, minimizing hypersurface with boundary $\Gamma'_t$, $E'_t = \bar M'_t \setminus M'_t \subset \RR^{n+1}$ is a compact set with Hausdorff dimension $\leq n-7$, and $T'_t = \llbracket M'_t \rrbracket$. By Allard's boundary regularity theorem, $M'_t \to M'_0$ smoothly near $\Gamma'_0$, so we may further restrict $\delta > 0$ so that $\spt T'_t \cap U_s(\Gamma) = \emptyset$ and all $M'_t$ meet $\Sigma$ transversely, and the intersection is a graph over $\Gamma'_0$ in $\Sigma$.
	
	We may apply Theorem \ref{theo:high.density.packing} with the infinite family $\cT = \{ T'_t \}_{t \in [0,\delta]}$; see Lemma \ref{lemm:plateau.minimizing.family} below to see why all hypotheses are met. It follows that there exists $t > 0$ such that
	\[ \cD(T'_t) \leq \cD(T') - \eta, \]
	where $\cD(\cdot)$ is as in \eqref{eq:calD} and $\eta = \eta(n) > 0$. By construction, $T'_t$ is disjoint from $T$.
	
	If $\cD(T'_t) < \Theta_n^*$, then we are done by Allard's regularity theorem. Otherwise, if $\cD(T'_t) \geq \Theta^*$, we apply Lemma \ref{lemm:main.plateau.uniqueness.shrinking} to obtain a nearby $T^*$ with $\spt T^* \subset \spt T'_t$ and restart the iterative argument with $T^*$ in place of $T$. We will only iterate finitely many times before eventually reaching $\cD(T^*) < \Theta_n^*$, since $\eta$ is fixed.
\end{proof}

\begin{lemma} \label{lemm:plateau.minimizing.family}
	Suppose that $(T'_t)_{t \in [0,\delta]}$ are as in the iterative step in the proof of Theorem \ref{theo:main.plateau.baire} above. Then for $\delta' > 0$ small and $K \in \NN$ large, the list
	\[ T'_0, T'_{\delta'/K}, T'_{2\delta'/K}, \ldots, T'_{\delta'} \]
	satisfies the hypotheses of Theorem \ref{theo:high.density.packing} with $\breve N = \RR^{n+1}$ and $\Omega = \bar B_R(\bOh) \setminus U$ for large $R$ and with $U$ as in \eqref{eq:main.plateau.main.lemma.U.Sigma}.
\end{lemma}
\begin{proof}
	Take $R$ large enough so that the convex hull of $U$ has closure in $B_{R-1}(\bOh)$. 
	
	Conditions (a), (c), (e) of Theorem \ref{theo:high.density.packing} are automatic from our construction. Condition (d) follows from \eqref{eq:main.with.uniqueness.limit}. 
	
	It remains to prove (b). The smooth transverse intersection near the boundary follows from the Hardt--Simon boundary regularity theorem (Theorem \ref{theo:hardt-simon-boundary-regularity}), Allard's boundary regularity, and \eqref{eq:main.with.uniqueness.limit}. For connectedness, we argue by contradiction. Suppose that there were $t_i \to 0$ and $T_i := T'_{t_i} \rightharpoonup T'$ with $\spt T'_i$ containing $\geq 2$ components. Without loss of generality, we restrict all our currents to $\Omega \setminus \partial \Omega$ where they have no boundary. We can then write $T_i = T_i^a+ T_i^b$, for $T_i^a, T_i^b \in \II_n(\Omega \setminus \partial \Omega) \setminus \{ 0 \}$ with $\partial T_i^a = \partial T_i^b = 0$ and $\spt T_i^a \cap \spt T_i^b = \emptyset$ in $\Omega \setminus \partial \Omega$. Note that the latter disjointness guarantees $\MM(T_i^a) + \MM(T_i^b) = \MM(T_i)$, so both $T_i^a$ and $T_i^b$ are minimizing (\cite[Lemma 33.4]{Simon:GMT}). Thus, they individually satisfy $T_i^a \rightharpoonup T^a \neq 0$, $T_i^b \rightharpoonup T^b \neq 0$, with $T^a$ and $T^b$ both minimizing in $\Omega \setminus \partial \Omega$ with $\partial T^a = \partial T^b = 0$. (We have $T_a$, $T_b \neq 0$ because of the graphical convergence of $\Gamma'_{t_i} \to \Gamma_0'$ and uniform lower density bounds at a fixed scale by the monotonicity formula.) Moreover,
	\[ T^a + T^b = T', \]
	and thus
	\[ \Theta_{T^a}(\cdot) + \Theta_{T^b}(\cdot) = \Theta_{T'}(\cdot). \]
	In particular, the multiplicity-one nature of $T'$ forces $T^a$ and $T^b$ to also have multiplicity-one, and $\reg T^a$ and $\reg T^b$ to be disjoint. Thus, $\spt T^a$ and $\spt T^b$ are also disjoint by Lemma \ref{lemm:disjoint.or.cross}. Thus, $\reg T^a$ and $\reg T^b$ are relatively open and relatively closed in  $\reg T$. But $\reg T$ is connected (Lemma \ref{lemm:reg.connected}), so $\reg T^a$ or $\reg T^b$ has to be trivial, a contradiction.
\end{proof}

\section{Theorem \ref{theo:main.homology}: Generic smoothness of homological minimizers in manifolds} \label{sec:main.homology}

We show how to generalize Smale's $8$-dimensional generic smoothness result \cite{Smale:generic} to dimensions $9$ and $10$. We prove the following, which in turn implies Theorem \ref{theo:main.homology} in view of the Baire Category Theorem:

\begin{theorem} \label{theo:main.homology.baire}
	Let $N$ be a closed, oriented, $(n+1)$-dimensional manifold with $n+1 \in \{ 8, 9, 10 \}$. Then, there exists a comeager $\operatorname{Met}^*(N) \subset \operatorname{Met}(N)$ such that for every $g \in \operatorname{Met}^*(N)$ and nonzero $[\alpha] \in H_n(N, \ZZ)$ there is a minimizing integral $n$-current in $(N, g)$ that represents $[\alpha]$ and is supported on a smooth hypersurface.
\end{theorem}

We will need two technical lemmas corresponding to \cite[Lemmas 1.3, 1.4]{Smale:generic}. The former has no dimensional restrictions in \cite{Smale:generic}, but its hypotheses and proof appear incomplete. The latter is the key technical lemma in \cite{Smale:generic} and had the dimensional restriction $n+1=8$, which we improve to $n+1 \in \{8,9,10\}$. 

\begin{lemma} \label{lemm:main.homology.uniqueness.perturbation}
	Fix $N$ as in Theorem \ref{theo:main.homology.baire}, $g \in \operatorname{Met}(N)$, and assume that $T \in \II^1_n(N)$ is minimizing in $[\alpha] \in H_n(N, \ZZ)$ with respect to $g$ and $\spt T$ is connected. Then, for every neighborhood $\cG \subset \operatorname{Met}(N)$ of $g$, there exists $g' \in \cG$ such that $T$ is \textit{uniquely} minimizing in $[\alpha]$ with respect to $g'$. If $\spt T$ is smooth, then we can also arrange for $\spt T$ to be strictly stable with respect to $g'$.\footnote{We mean this in the sense of the second variation of $\spt T$ (a smooth hypersurface) being positive definite for normal directions.} 
\end{lemma}

\begin{lemma} \label{lemm:main.homology.main.lemma}
	Fix $N$ as in Theorem \ref{theo:main.homology.baire}, $g \in \operatorname{Met}(N)$, and assume that $T \in \II^1_n(N)$ is uniquely minimizing in $[\alpha] \in H_n(N, \ZZ)$ with respect to $g$ and $\spt T$ is connected. For every neighborhood $\cG \subset \operatorname{Met}(N)$ of $g$, there exists $g' \in \cG$ and a $T' \in \II^1_n(N, g')$ which is minimizing in $[\alpha]$ with respect to $g'$ and with $\spt T$ smooth and connected.
\end{lemma}

Let us first show how these lemmas imply the theorem.

\begin{proof}[Proof of Theorem \ref{theo:main.homology.baire} assuming Lemmas \ref{lemm:main.homology.uniqueness.perturbation}, \ref{lemm:main.homology.main.lemma}]
	For $[\alpha] \in H_n(N, \ZZ)$ let
	\[ \operatorname{Met}^*_{[\alpha]}(N) \subset \operatorname{Met}(N) \]
	consist of all $g \in \operatorname{Met}(N)$ with the following defining property: \textbf{if} (with respect to this $g$) there exists a minimizing $T \in \II_n(N)$ in $[\alpha]$ with respect to $g$ satisfying: 
	\begin{enumerate}
		\item[(C$_0)$] $\spt T$ is connected, 
	\end{enumerate}
	\textbf{then} $T$ also satisfies (with respect to the same $g$):
	\begin{enumerate}
		\item[(C$_1$)] $\spt T$ is a smooth hypersurface, \textit{and}
		\item[(C$_2$)] $\spt T$ is strictly stable, \textit{and}
		\item[(C$_3$)] $T$ is \textit{uniquely} minimizing in $[\alpha]$. 
	\end{enumerate}
	Note that if a single minimizing $T \in \II_n(N)$ in $[\alpha]$ satisfies (C$_0$)-(C$_3$) with respect to $g$, then $g \in \operatorname{Met}^*_{[\alpha]}(N)$ due to (C$_3$).
	
	\begin{claim} \label{clai:main.homology.baire.open}
		$\operatorname{Met}^*_{[\alpha]}(N)$ is open in $\operatorname{Met}(N)$.
	\end{claim}
	\begin{proof}[Proof of claim]
		We argue by contradiction. Suppose that $g \in \operatorname{Met}^*_{[\alpha]}(N)$ and that there exists a sequence (where the index is suppressed) $g' \to g$ with $g' \in \operatorname{Met}(N) \setminus \operatorname{Met}^*_{[\alpha]}(N)$. 
		
		Then, for each $g'$ there exists an associated minimizing $T' \in \II_n(N)$ in $[\alpha]$ satisfying (C$_0$) but failing (C$_1$) or (C$_2$) or (C$_3$), with respect to $g'$. By passing to subsequences, we can assume that $T' \rightharpoonup T \in \II_n(N)$ minimizing in $[\alpha]$, with respect to $g$. Note that since each $T'$ satisfies (C$_0$), $T$ does too by the 
		upper semicontinuity of density, and thus $T$ also satisfies (C$_1$)-(C$_3$), with respect to $g$, since $g \in \operatorname{Met}^*_{[\alpha]}(N)$.
		
		It follows from (C$_1$) on $T$ together with standard regularity results for minimizers that, up to discarding finitely many terms in the sequence,  (C$_1$) must hold for all $T'$ and that $\spt T' \to \spt T$ locally graphically possibly with multiplicity. (One can locally decompose $T'$ into boundaries (\cite[Corollary 27.8]{Simon:GMT}) and invoke De Giorgi's regularity for minimizing boundaries (\cite{DeGiorgi:minimizing-regularity}).) But since the normal bundle of $\spt T$ is trivial, $\spt T' \to \spt T$ smoothly and with multiplicity one. We conclude by showing neither (C$_2$) nor (C$_3$) can fail for infinitely many $T'$, with respect to their $g'$. 
		
		First assume that infinitely many $T'$ fail (C$_2$), with respect to their $g'$. Pass to such a subsequence. Each $\spt T'$ carries nontrivial Jacobi fields, with respect to its $g'$. These nontrivial Jacobi fields converge (after the standard normalization procedure) smoothly to nontrivial Jacobi fields on $\spt T$, with respect to $g$, given the smooth convergence $\spt T' \to \spt T$ and $g' \to g$. Thus $\spt T$ is degenerate with respect to $g$, a contradiction to $T$ satisfying (C$_2$).
		
		So it must be that infinitely many $T'$ fail (C$_3$), with respect to their $g'$. Pass to this subsequence. Let $\tilde T'$ be any competing minimizer with respect to $g'$ that is distinct from $T$. We must have $\tilde T' \rightharpoonup T$ after passing to a further subsequence because $T$ satisfies (C$_3$). By the argument above (it only relies on the limit $T$ satisfying (C$_1$)) the convergence of the supports is again smooth and graphical with multiplicity one. Using the standard construction from the setting of the graphical convergence $\spt T'$, $\spt \tilde T' \to \spt T$, and $g' \to g$, we produce a nontrivial Jacobi field on $\spt T$ with respect to $g$. Thus $\spt T$ is degenerate, again a contradiction to $T$ satisfying (C$_2$).

		Since all cases lead to a contradiction, $\operatorname{Met}^*_{[\alpha]}(N)$ is open. 
	\end{proof}
	
	\begin{claim} \label{clai:main.homology.baire.dense}
		$\operatorname{Met}^*_{[\alpha]}(N)$ is dense in $\operatorname{Met}(N)$.
	\end{claim}
	\begin{proof}[Proof of claim]
		Let $g \in \operatorname{Met}(N)$. We show there exist $g' \in \operatorname{Met}^*_{[\alpha]}(N)$ arbitrarily close to $g$. We assume $g \not \in \operatorname{Met}^*_{[\alpha]}(N)$, otherwise there is nothing to show. 
		
		Let $T$ be a minimizer in $[\alpha]$ satisfying (C$_0$) but failing (C$_1$) or (C$_2$) or (C$_3$), with respect to $g$. The multiplicity $\ell$ of $T$ is a constant by Lemma \ref{lemm:reg.connected} and the constancy theorem. Then $\hat T := \ell^{-1} T \in \II^1_n(N)$ is minimizing (\cite[Lemma 33.4]{Simon:GMT}) in $[\hat \alpha] := \ell^{-1} [\alpha]$ ($\in H_n(N, \ZZ)$ due to the multiplicity of $T$ being $\ell$) with respect to $g$ and satisfies (C$_0$) since $\spt \hat T = \spt T$. We'll perturb $(g, \hat T) \mapsto (g', \hat T')$ so that $\hat T'$ satisfies (C$_0$)-(C$_3)$ with respect to $g'$. Without loss of generality we'll take $T$ to fail (C$_1$), otherwise we skip steps (i) and (ii) below and start with step (iii) and with $(g', \hat T') = (g, \hat T)$:
		\begin{enumerate}
			\item[(i)] Apply Lemma \ref{lemm:main.homology.uniqueness.perturbation} to perturb $g$ (not relabeled) so that $\hat T$ satisfies (C$_0$) and (C$_3$) with respect to $g$. 
			\item[(ii)] Apply Lemma \ref{lemm:main.homology.main.lemma} to perturb $(g, \hat T) \mapsto (g', \hat T')$ so that $\hat T' \in \II_n^1(N)$ satisfies (C$_0$) and (C$_1$) with respect to $g'$. 
			\item[(iii)] Re-apply Lemma \ref{lemm:main.homology.uniqueness.perturbation} to perturb $g'$ (not relabeled) so that $\hat T' \in \II_n^1(N)$ satisfies (C$_0$), (C$_1$), (C$_2$), (C$_3$) with respect to $g'$. 
		\end{enumerate}
		
		It remains to show that $\ell \hat T'$ is minimizing in $[\alpha]$ and satisfies (C$_0$)-(C$_3)$ with respect to $g'$, too. As we've discussed, this shows that $g' \in \operatorname{Met}^*_{[\alpha]}(N)$.
		
		The fact that $\ell \hat T'$ is minimizing follows from \cite[Theorem 5.10]{Federer:real-chains} (see also \cite[Corollary 3]{White:boundary-regularity-multiplicity}). Next, $\ell \hat T'$ clearly inherits (C$_0$)-(C$_2$) from $\hat T'$. Applying White's decomposition theorem \cite{White:boundary-regularity-multiplicity} and using the validity of (C$_3$) for $\hat T'$, we obtain that (C$_3$) holds for $\ell \hat T'$, too. Thus, $\ell \hat T'$ satisfies (C$_0$)-(C$_3$). 
	\end{proof}
	
	It follows from the countable nature of $H_n(N, \ZZ)$ and Claims \ref{clai:main.homology.baire.open}, \ref{clai:main.homology.baire.dense} that
	\[ \operatorname{Met}^*(N) :=\bigcap_{[\alpha] \in H_n(N, \ZZ)} \operatorname{Met}^*_{[\alpha]}(N) \]
	is comeager in $\operatorname{Met}(N)$. 
	
	It remains to prove that $\operatorname{Met}^*(N)$ has the desired smooth minimization property. Let $g \in \operatorname{Met}^*(N)$ and suppose that $T \in \II_n(N)$ is minimizing with respect to $g$ in some homology class. Let $C$ denote any component of $\spt T$. Note that $T \mres C \in \II_n(N)$, and that its multiplicity $\ell$ is constant due to Lemma \ref{lemm:reg.connected} and the constancy theorem. Therefore, $T' := \ell^{-1} T \mres C \in \II^1_n(N)$ and it is also minimizing (\cite[Lemma 33.4]{Simon:GMT}). By construction, $T'$ satisfies (C$_0$). Let its homology class be $[\alpha]$. It follows from $g \in \operatorname{Met}^*(N) \subset \operatorname{Met}^*_{[\alpha]}(N)$ and (C$_1$) in the definition of $\operatorname{Met}^*_{[\alpha]}(N)$ that $\spt T' = C$ is smooth. Iterating over all components, we get that $\spt T$ is smooth.
\end{proof}

The proofs of Lemmas \ref{lemm:main.homology.uniqueness.perturbation} and \ref{lemm:main.homology.main.lemma} are found at the end of the section. On a first pass through the paper, the reader can safely skip the statement and proof of the technical proposition below and go directly to the proof of the lemmas.

\begin{proposition} \label{prop:deformation.manifold}
	Let $N$, $g$ be as in Lemma \ref{lemm:main.homology.main.lemma}. Assume that $T \in \II_n^1(N)$ has $\partial T = 0$, is uniquely minimizing with respect to $g$, has $\spt T$ connected, and that $\cG \subset \operatorname{Met}(N)$ is a neighborhood of $g$. Below, geometric quantities are with respect to $g$ unless otherwise noted.
	
	Fix $p \in \reg T$. If $D_r$ denotes the radius-$r$ geodesic ball in $\reg T$ centered at $p$, take $\rho > 0$ to be such that $D_{3\rho}$ does not intersect the cut locus of $p$ or $\sing T$, and $D_{2\rho}$ has orientation-preserving Fermi coordinates
	\[ \Phi : D_{2\rho} \times (-\sigma, \sigma) \xrightarrow{\approx} C_{2\rho} \subset N \]
	for some $\sigma > 0$. More generally, we set:
	\[ C_r := \Phi(D_r\times (-\sigma, \sigma)) \text{ for } r \in (0, 2\rho]. \]
	Then, for all sufficiently small $\tau > 0$ and all $K \in \NN$, there exist metrics
	\[ g =: g_{K,0}, g_{K,1}, \ldots, g_{K,K} \in \cG, \]
	and integral $n$-currents
	\[ T =: T_{K,0}, T_{K,1}, \ldots, T_{K,K} \in \II_n(N), \]
	that are all homologous to $T$ and minimizing with respect to the corresponding metrics, such that, with the convention $h_{K,0} := 0$, the following hold for all $k = 1, \ldots, K$:
	\begin{enumerate}
		\item[(a)] $g_{K,k} \equiv g$ on $N \setminus C_\rho$,
		\item[(b)] $\Vert g_{K,k} - g_{K,k-1} \Vert_{C^{m,\theta}(N)} \leq C_{1,m,\theta} \tau K^{-1} $ for all $m \in \NN^*$, $\theta \in (0,1)$,
		\item[(c)] $\lim_{\tau \to 0} \dist_{\cF}(T_{K,k}, T) = 0$,\footnote{Here, $\dist_{\cF}$ is the flat distance in $N$, which is denoted $\cF_N$ in \cite[4.1.24]{Federer:GMT}.}
		\item[(d)] $\Phi^{-1}(\spt T_{K,k} \cap C_{2\rho}) = \Graph_{D_{2\rho}} h_{K,k}$ with $\lim_{\tau \to 0} \Vert h_{K,k} \Vert_{C^{m,\alpha}(D_{2\rho})} = 0$,
		\item[(e)] $T_{K,k} \in \II_n^1(N)$ and $\spt T_{K,k}$ is connected,
		\item[(f)] $\inf_{D_{3\rho/2}} (h_{K,k} - h_{K,k-1}) \geq \mu_1 \tau K^{-1}$.
	\end{enumerate}
	We emphasize that $C_{1,m,\theta}$, $\mu_1 > 0$ are independent of $K, k, \tau$. Above, we suppressed the dependence on $\tau$ from the notation of $g_{K,k}$, $T_{K,k}$, $h_{K,k}$ for simplicity.
\end{proposition}
\begin{proof}
Fix once and for all smooth and compactly supported 
\[ \chi : D_\rho \to \RR, \]
\[ \zeta : \RR \to \RR, \]
such that
\begin{equation} \label{eq:smale.chi.setup}
	\chi \geq 0 \text{ on } D_\rho, \; \chi \geq 1 \text{ on } D_{2\rho/3},
\end{equation}
\begin{equation} \label{eq:smale.zeta.setup}
	s \zeta(s) \geq 0 \text{ for all } s \in \RR, \; \zeta'(0) > 0,
\end{equation}
\begin{equation} \label{eq:smale.chi.zeta.support}
	\spt \left( (x, s) \mapsto \chi(x) \zeta(s + s_0) \right) \subset D_\rho \times (-\sigma, \sigma) \text{ for all } s_0 \in [-\tfrac12 \sigma, \tfrac12 \sigma].
\end{equation}

We need three lemmas whose proofs we postpone.

\begin{lemma} \label{lemm:locally.homologically.trivial}
	Let $N, g$ be as in Lemma \ref{lemm:main.homology.main.lemma} and assume that $T \in \II_n^1(N)$ has $\partial T = 0$ and is minimizing with respect to $g$. Then, there exists a smooth open neighborhood $\breve N$ of $\spt T$ such that $T = \partial \llbracket E \rrbracket \mres \breve N$, i.e., $T$ is a  minimizing boundary in $\breve N$.
\end{lemma}

\begin{lemma} \label{lemm:smale.abstract.perturbation.closeness}
	Let $N$ and $g$ be as in Lemma \ref{lemm:main.homology.main.lemma}. Assume that $T \in \II_n^1(N)$ has $\partial T = 0$, is uniquely minimizing with respect to $g$, and $\spt T$ is connected. For every $m \geq 2$, $\theta \in (0, 1)$, and $\eps > 0$, there is $\delta > 0$ so that if $g' \in \operatorname{Met}(N)$, $\Vert g' - g \Vert_{C^{m,\theta}(N)} \leq \delta$, and $T' \in \II_n(N)$ is minimizing with respect to $g'$ in the homology class of $T$, then:
	\begin{enumerate}
		\item[(a)] $\dist_{\cF}(T', T) \leq \eps$;
		\item[(b)] $\Phi^{-1}(\spt T' \cap C_{2\rho}) = \Graph_{D_{2\rho}} h'$ with $\Vert h' \Vert_{C^{m,\theta}(D_{2\rho})} \leq \eps$;
		\item[(c)] $T' \in \II_n^1(N)$ and $\spt T'$ is connected.
	\end{enumerate}
\end{lemma}

In the following lemma we state the result for a general manifold $D$ with a Riemannian metric. We will later apply this with the obvious choice of $\chi, \zeta$, and $D=D_\rho$ above with its induced metric.
	
\begin{lemma} \label{lemm:smale.abstract.product.smallness}
	Let $\tau \in \RR$ and $\zeta : \RR \to \RR$ and $\chi, h_1, \ldots, h_k : D \to \RR$ be smooth. Then, the smooth function $w : D \times \RR \to \RR$ defined by
	\[ w(x, s) := \prod_{j=1}^k (1 - \tau \chi(x) \zeta(s - h_j(x))), \]
	satisfies, for all $m \in \NN$, $\theta \in (0, 1)$, 
	\[ \Vert w - 1 \Vert_{C^{m,\theta}(D \times \RR)} \leq C_{0,m} \sum_{j=1}^k \frac{k!}{(k-j)!} \tau^j \Vert \chi \Vert_{C^{m,\theta}(D)}^j \Vert \zeta \Vert_{C^{m,\theta}(\RR)}^j (1+\max_{j=1, \ldots, k} \Vert h_j \Vert_{C^{m,\theta}(D)})^{mj}. \]
	Here $C_{0,m}$ only depends on $m$. 
\end{lemma}

We return to the proof of Proposition \ref{prop:deformation.manifold}. First, let $\breve N$ be as in Lemma \ref{lemm:locally.homologically.trivial} so that
\[ T = \partial \llbracket E \rrbracket \mres \breve N. \]
After possibly swapping orientations and thus replacing $E$ with its complementary Caccioppoli set in $\breve N$, we may assume that, in our Fermi coordinates,
\begin{equation} \label{eq:smale.product.smallness.orientation}
	\{ (x, s) : \chi(x) \zeta(s) > 0 \} \subset E.
\end{equation}
Then let $m \in \NN$, $m \geq 2$, $\delta > 0$ be chosen so that
\begin{equation} \label{eq:smale.product.smallness.metric}
	\Vert \tilde g - g \Vert_{C^{m,\theta}(N)} \leq \delta \implies \tilde g \in \cG.
\end{equation}
Shrink $\delta > 0$ enough so that Lemma \ref{lemm:smale.abstract.perturbation.closeness} applies with $\eps$ no larger than $\min \{ 1, \tfrac12 \sigma \}$ or so that the flat distance bound in (a) of the lemma, together with the implied Hausdorff distance (by the monotonicity formula), forces $\spt T' \subset \breve N$. Next choose $\gamma \in (0, 1)$ sufficiently small, depending on $\eps$, $\theta$, so that we have for $w : N \to \RR$
\begin{equation} \label{eq:smale.product.smallness.metric.conformal}
	\Vert w - 1 \Vert_{C^{m,\theta}(N)} \leq \gamma \implies \Vert w g - g \Vert_{C^{m,\theta}(N)} \leq \delta.
\end{equation}
With $C_{0,m}$ as in Lemma \ref{lemm:smale.abstract.product.smallness}, choose $\tau_1 > 0$ small enough that
\[ C_{0,m} \sum_{j=1}^\infty \tau_1^j \Vert \chi \Vert_{C^{m,\theta}(D)}^j \Vert \zeta \Vert_{C^{m,\theta}(\RR)}^j 2^{mj} \leq \gamma. \]
For all $k, K \in \NN$, $k \leq K$, 
\begin{align} \label{eq:smale.product.smallness}
	& C_{0,m} \sum_{j=1}^k \frac{k!}{(k-j)!} (\tau_1 K^{-1})^j \Vert \chi \Vert_{C^{m,\theta}(D)}^j \Vert \zeta \Vert_{C^{m,\theta}(\RR)}^j 2^{mj} \nonumber \\
	& \qquad = C_{0,m} \sum_{j=1}^k \frac{k (k-1) \cdots (k-(j-1))}{K^j} \tau_1^j \Vert \chi \Vert_{C^{m,\theta}(D)}^j \Vert \zeta \Vert_{C^{m,\theta}(\RR)}^j 2^{mj} \nonumber \\
	& \qquad \leq C_{0,m} \sum_{j=1}^k \tau_1^j \Vert \chi \Vert_{C^{m,\theta}(D)}^j \Vert \zeta \Vert_{C^{m,\theta}(\RR)}^j 2^{mj} \leq \gamma.
\end{align}
Without loss of generality, we can further shrink $\tau_1$ so that
\begin{equation} \label{eq:smale.pos.def.smallness}
	\tau_1 \sup_M \chi \sup_\RR \zeta < 1 
\end{equation}
and always suppose $\tau \in (0, \tau_1]$ going forward. 

Now let $K \in \NN$ be arbitrary and denote
\[ g_{K,0} := g, \; T_{K,0} := T, \; E_{K,0} := E, \; h_{K,0} := 0. \]
We will perform an iterative construction that guarantees (a)-(e) and, at first, a weaker version of (f): a weak monotonicity that we call (f').

\textbf{Perturbation $1$}. Define a metric $g_{K,1}$ on $\operatorname{img} \Phi = C_{2\rho}$ via
\[ (\Phi^* g_{K,1})(x, s) := (1 - \tau K^{-1} \chi(x) \zeta(s - h_{K,0}(x))) (\Phi^* g_{K,0})(x,s). \]
Using \eqref{eq:smale.chi.zeta.support}, \eqref{eq:smale.product.smallness.metric}, \eqref{eq:smale.product.smallness.metric.conformal}, \eqref{eq:smale.product.smallness} (with $k = 1$), and \eqref{eq:smale.pos.def.smallness}, 
\begin{enumerate}
	\item[(a)] $g_{K,1}$ extends to a metric on $N$ that coincides with $g$ off $C_\rho$, and
	\item[(b)] $\Vert g_{K,1} - g_{K,0} \Vert_{C^{m,\theta}(N)} \leq C_{1,m,\theta} \tau K^{-1}$.
\end{enumerate}
Take $T_{K,1} \in \II_n(N)$ to be any (for now) candidate minimizer with respect to $g_{K,1}$ homologous to $T$. By Lemma \ref{lemm:smale.abstract.perturbation.closeness} and our choices of parameters, we immediately obtain $\spt T_{K,1} \subset \breve N$ and all the following properties:
\begin{enumerate}
	\item[(c)] $\lim_{\tau \to 0} \dist_{\cF}(\spt T_{K,1}, \spt T) = 0$,
	\item[(d)] $\Phi^{-1}(\spt T_{K,1} \cap C_{2\rho}) = \Graph_{D_{2\rho}} h_{K,1}$ with $\lim_{\tau \to 0} \Vert h_{K,1} \Vert_{C^{m,\theta}(D_{2\rho})} = 0$,
	\item[(e)] $T_{K,1} \in \II_n^1(N)$ and $\spt T_{K,1}$ is connected.
\end{enumerate}

We wish to choose a particular minimizer, namely, the topmost one. Indeed, the standard codimension-1 decomposition theorem \cite[Theorem 27.6]{Simon:GMT}, we obtain from (e) that the decomposition of $T_{K,1}$ into boundaries of Caccioppoli sets must have a single summand, i.e., $T_{K,1} = \partial \llbracket E_{K,1} \rrbracket$ for a Caccioppoli set $E_{K,1} \subset \breve N$. In particular, all candidate minimizers $T_{K,1}$ are described as consistently oriented boundaries of Caccioppoli sets. Those minimizing boundaries cannot intersect by cut and paste and the strong maximum principle (Lemma \ref{lemm:disjoint}), so they can be ordered by inclusion. In particular, there is a unique (up to measure zero) topmost one. Fix
\[ T_{K,1} = \partial \llbracket E_{K,1} \rrbracket \mres \breve N \]
to be this topmost minimizer going forward. 

We claim that a standard cut-and-paste argument implies:
\begin{enumerate}
	\item[(f')] $E_{K,1} \subset E_{K,0}$ and $\spt T_{K,1} \cap \spt T_{K,0} = \emptyset$.
\end{enumerate}
Again, set inclusions are up to measure zero. To see this, take
\[ E_{K,1}' := E_{K,0} \cap E_{K,1} \subset E_{K,1} \text{ and } T_{K,1}' := \partial \llbracket E_{K,1}' \rrbracket \mres \breve N. \]
This is a competitor for $T_{K,1}$ in the same homology class. Cut and paste gives:

\begin{claim} \label{clai:smale.competitor.t1prime}
	$\MM_{g_{K,1}}(T_{K,1}') \leq \MM_{g_{K,1}}(T_{K,1})$.
\end{claim}

We briefly postpone the proof to see how Claim \ref{clai:smale.competitor.t1prime} implies (f'). First, it implies the inclusion assertion in (f'): otherwise, $E_{K,1}' \subsetneq E_{K,1}$, contradicting our choice of $T_{K,1}$ as the topmost minimizer. The disjointness assertion of (f') follows from the maximum principle: $\zeta'(0) > 0$ and \eqref{eq:smale.product.smallness.orientation} force $T_{K,0}$ to be a strict barrier for mass minimization with respect to $g_{K,1}$. This completes the proof of (f').

\begin{proof}[Proof of Claim \ref{clai:smale.competitor.t1prime}]
It is convenient to also consider
\[ E_{K,0}' := E_{K,0} \cup E_{K,1} \text{ and } T_{K,0}' := \partial \llbracket E_{K,0}' \rrbracket \mres \breve N. \]
Proceeding as in the proof of Lemma \ref{lemm:disjoint}, we note that
\begin{equation} \label{eq:smale.competitor.t1prime.sum}
	T_{K,0} + T_{K,1} = T_{K,0}' + T_{K,1}'
\end{equation}
and, since $T_{K,0}'$, $T_{K,1}'$ differ by the boundary of a Caccioppoli set ($E_{K,0} \Delta E_{K,1}$) we have the equality of measures (taken with respect to any fixed metric)
\begin{equation} \label{eq:smale.competitor.t1prime.dec}
	\Vert T_{K,0}' \Vert + \Vert T_{K,1}' \Vert = \Vert T_{K,0}' + T_{K,1}' \Vert.
\end{equation}
By \eqref{eq:smale.competitor.t1prime.dec} with $g_{K,1}$, \eqref{eq:smale.competitor.t1prime.sum}, and the triangle inequality:
\begin{align*} 
	\MM_{g_{K,1}}(T_{K,0}') + \MM_{g_{K,1}}(T_{K,1}') 
		& = \MM_{g_{K,1}}(T_{K,0}'+T_{K,1}') \\
		& = \MM_{g_{K,1}}(T_{K,0} + T_{K,1}) \leq \MM_{g_{K,1}}(T_{K,0}) + \MM_{g_{K,1}}(T_{K,1}).
\end{align*}
The claim follows since $g_{K,1} \geq g_{K,0}$ along $\spt T_{K,0}'$ and $g_{K,1} = g_{K,0}$ along $\spt T_{K,0}$, together with $\MM_{g_{K,0}}(T_{K,0}) \leq  \MM_{g_{K,0}}(T_{K,0}')$ due to the minimizing nature of $T_{K,0}$ with respect to $g_{K,0}$.
\end{proof}

\textbf{Perturbation $2$}. Assuming $K \geq 2$ we perturb again. Define $g_{K,2}$ on $C_{2\rho}$ via
\[ (\Phi^* g_{K,2})(x,s) := (1 - \tau K^{-1} \chi(x) \zeta(s - h_{K,1}(x))) (\Phi^* g_{K,1})(x,s). \]
As before, except with $k=2$ in \eqref{eq:smale.product.smallness}, 
\begin{enumerate}
	\item[(a)] $g_{K,1}$ extends to a metric on $N$ that coincides with $g$ off $C_\rho$, and
	\item[(b)] $\Vert g_{K,2} - g_{K,1} \Vert_{C^{m,\theta}(N)} \leq C_{1,m,\theta} \tau K^{-1}$.
\end{enumerate}
Take $T_{K,2} \in \II_n(N)$ to be the topmost minimizer homologous to $T$ with respect to $g_{K,2}$. As before, $\spt T_{K,2} \subset \breve N$, $T_{K,2} = \partial \llbracket E_{K,2} \rrbracket \mres \breve N$ for $E_{K,2} \subset \breve N$, and
\begin{enumerate}
	\item[(c)] $\lim_{\tau \to 0} \dist_{\cF}(\spt T_{K,2}, \spt T) = 0$, 
	\item[(d)] $\Phi^{-1}(\spt T_{K,2} \cap C_{2\rho}) = \Graph_{D_{2\rho}} h_{K,2}$ with $\lim_{\tau \to 0} \Vert h_{K,2} \Vert_{C^{m,\theta}(D_{2\rho})} = 0$,
	\item[(e)] $T_{K,2} \in \II_n^1(N)$ and $\spt T_{K,2}$ is connected,
	\item[(f')]  $E_{K,2} \subset E_{K,1}$ and $\spt T_{K,2} \cap \spt T_{K,1} = \emptyset$.
\end{enumerate}

\textbf{Subsequent perturbations}. We repeat this construction to end up with
\[ g_{K,0}, g_{K,1}, g_{K,2}, \ldots, g_{K,K}, \]
\[ T_{K,0}, T_{K,1}, T_{K,2}, \ldots, T_{K,K}, \]
\[ E_{K,0}, E_{K,1}, E_{K,2}, \ldots, E_{K,K}. \]
These satisfy (a)-(e), (f') by construction and also Lemma \ref{lemm:smale.abstract.product.smallness} and \eqref{eq:smale.product.smallness}.

It remains to prove that they all satisfy (f) too. We prove this by contradiction. If (f) failed, then for some $K_n \to \infty$, $k_n \in \{ 1, \ldots, K_n \}$, and $\tau_n \to 0$, one would have had
\begin{equation} \label{eq:smale.Tk.gap.contradiction}
	\inf_{D_{3\rho/2}} (h_{K_n,k_n} - h_{K_n,k_n-1}) = \eps_n \tau_n K_n^{-1},
\end{equation}
with $\eps_n \to 0$.

Fix $g_{K_n,k_n}$ as a background metric. With respect to it, $T_{K_n,k_n}$ is minimizing globally, so its mean curvature vanishes on $C_{2\rho}$. On the other hand, $T_{K_n,k_n-1}$ is only minimizing in $N \setminus \bar C_\rho$ where $g_{K_n,k_n-1} \equiv g \equiv g_{K_n,k_n}$. To compute its mean curvature on $C_{2\rho}$ with respect to $g_{K_n,k_n}$, we recall that
\[ (\Phi^* g_{K_n,k_n})(x,s) := (1 - \tau_n K_n^{-1} \chi(x) \zeta(s - h_{K_n,k_n-1}(x))) (\Phi^* g_{K_n,k_n-1})(x,s). \]
Since $T_{K_n,k_n-1}$ is minimizing globally with respect to $g_{K_n,k_n-1}$, its mean curvature with respect to $g_{K_n,k_n-1}$ vanishes on $C_{2\rho}$. It follows from the conformal change formula for mean curvature, and the uniform graphicality of $\spt T_{K_n,k_n-1} \cap C_{2\rho}$ from (d), and \eqref{eq:smale.zeta.setup} that the mean curvature vector along $\spt T_{K_n,k_n-1} \cap C_{2\rho}$ (which exists classically by (d)) points weakly toward $\spt T_{K_n,k_n}$ and the mean curvature scalar (again, both with respect to $g_{K_n,k_n}$) satisfies 
\begin{equation} \label{eq:smale.Tk.gap.contradiction.meancurv}
	\mu_1'' \tau_n K_n^{-1} \chi \geq H \geq \mu_1' \tau_n K_n^{-1} \chi \text{ on } \spt T_{K_n,k_n-1} \cap C_{2\rho}.
\end{equation}
with $0 < \mu_1' < \mu_1''$ independent of $n$. It follows from \eqref{eq:smale.chi.setup}, (f'), and a straightforward barrier argument using the definite sign of $\zeta'(0) > 0$ that
\begin{equation} \label{eq:smale.hk.gap.contradiction.heightgap}
	h_{K_n,k_n} - h_{K_n,k_n-1} \geq \tilde \mu_1' \tau_n K_n^{-1} \text{ on } D_{\rho/2}
\end{equation}
where $\tilde \mu_1' > 0$ is also independent of $n$. 
Indeed, fix a smooth function $\tilde \chi$ on $D_{2\rho}$ with $\tilde \chi \geq 1$ on $D_{\rho/2}$ and $ \tilde \chi < 0$ on $D_{2\rho}\setminus D_{2\rho/3}$. Consider the graph of $h_{K_n,k_n} - t\tilde \chi $.  By Taylor expanding around $t=0$, there exists $\tilde \mu_1' > 0$ (depending on $\tilde \chi$ but not $n$) such that, for $t \in (0, \tilde \mu_1' \tau_n K_n^{-1}]$, the graph of $h_{K_n,k_n} - t\tilde \chi $ has mean curvature $\tilde H(t; n) \leq \frac 12 \mu_1'\tau_n K_n^{-1}$ over $D_{2\rho}$ with respect to $g_{K_n,k_n}$, with $\mu_1'$ as in \eqref{eq:smale.Tk.gap.contradiction.meancurv}. In particular, if \eqref{eq:smale.hk.gap.contradiction.heightgap} failed, then there would exist $t \in (0, \tilde \mu_1' \tau_n K_n^{-1}]$ so that the graph of $h_{K_n,k_n} - t\tilde \chi$ makes one-sided contact with the graph of $h_{K_n,k_n-1}$. This contradicts $\tilde H(t; n) \leq \frac 12 \mu_1'\tau_n K_n^{-1}$, \eqref{eq:smale.Tk.gap.contradiction.meancurv}, and the maximum principle. This proves \eqref{eq:smale.hk.gap.contradiction.heightgap}. 

Now set
	\[ Q_n := \inf_{D_{\rho/2}} (h_{K_n,k_n} - h_{K_n,k_n-1}), \]
	\[ h_n := \frac{h_{K_n,k_n} - h_{K_n,k_n-1}}{Q_n}. \]
Note that $h_n > 0$ by (f'). Using the bounds from (b), (d), \eqref{eq:smale.Tk.gap.contradiction.meancurv}, \eqref{eq:smale.hk.gap.contradiction.heightgap},  $\tau_n \to 0$, and the uniquely minimizing nature of $T$, we get $Q_n \to 0$, and from the standard Harnack inequality and standard Schauder theory on $D_{2\rho}$ to let us send $n \to \infty$ and, after passing to a subsequence, obtain $h : D_{2\rho} \to \RR$ which is nonnegative, $C^{1,\theta}$, and weakly satisfies
	\[ - \Delta h - V h \geq 0 \text{ on } D_{2\rho}, \]
where $-\Delta -V$ is the Jacobi operator on the regular part of $\lim_n T_{K_n,k_n}$ (not labeled) with respect to the metric $\lim_n g_{K_n,k_n}$ (also not labeled), as well as 
	\[ \inf_{D_{3\rho/2}} h = 0 \]
	by \eqref{eq:smale.Tk.gap.contradiction}, and
	\[ \inf_{D_{\rho/2}} h = 1 \]
	by construction. This is a contradiction to the strong maximum principle for nonnegative supersolutions of the Jacobi equation on $D_{2\rho}$. Thus, \eqref{eq:smale.Tk.gap.contradiction} fails, so (f) holds.
\end{proof}

\begin{proof}[Proof of Lemma \ref{lemm:locally.homologically.trivial}]
	Fix $r > 0$ smaller than the injectivity radius of $(N, g)$. Let $\bx \in \spt T$ be arbitrary. Note that, by the contractibility of $B_r(\bx)$, $T \mres B_r(\bx)$ is decomposable as the sum of boundaries of nested, minimizing Caccioppoli sets \cite[Theorem 27.6]{Simon:GMT}. It follows from the nested property of the Caccioppoli sets and Lemma \ref{lemm:disjoint.or.cross} that any two are either disjoint or overlap in $B_r(\bx)$, which of course means any two have to be disjoint in $B_r(\bx)$ because $T$ has multiplicity one. In particular, there exists $\delta_\bx \in (0, r)$ such that $B_{\delta_\bx}(\bx)$ intersects exactly one component of $\spt T \cap B_r(\bx)$ (else we'd get an interior intersection point of two components by sending $\delta \to 0$) and thus
	\[ T \mres B_{\delta_\bx}(\bx) = \partial \llbracket E_{\bx} \rrbracket \mres B_{\delta_\bx}(\bx) \]
	for a Caccioppoli subset $E_{\bx} \subset B_r(\bx)$. 
	
	\begin{claim} \label{clai:locally.homologically.trivial.uniqueness}
		If $B \subset B_{\delta_\bx}(\bx)$ is an open ball intersecting $\spt T$ and $T \mres B = \partial \llbracket E \rrbracket \mres B$, then $E_\bx \cap B = E \cap B$ up to a measure zero subset of $B$. 
	\end{claim}
	\begin{proof}[Proof of claim]
		Note that we also have $T \mres B = \partial \llbracket E_\bx \rrbracket \mres B$, so in particular
		\[ 0 = \partial (\llbracket E_\bx \rrbracket - \llbracket E \rrbracket) \mres B. \]
		It follows from the proof of \cite[Theorem 27.6]{Simon:GMT} that the density of $\llbracket E_\bx \rrbracket - \llbracket E \rrbracket$ must vanish a.e.\ in $B$, or $\partial (\llbracket E_\bx \rrbracket - \llbracket E \rrbracket) \mres B$ would have had positive measure in $B$, a contradiction. 
	\end{proof}
	
	Consider the open covering $\{ B_{\delta_{\bx}/3}(\bx) \}_{\bx \in \spt T}$ of $\spt T$. Since $\spt T$ is compact, there exists a finite open subcover $\{ B_{\delta_i/3}(\bx_i) \}_{i=1}^N$ of $\spt T$, $\bx_1, \ldots, \bx_N \in \spt T$, $\delta_i := \delta_{\bx_i}$. Set
	\[ \hat N := \cup_{i=1}^N B_{\delta_i/3}(\bx_i), \]
	which is an open neighborhood of $\spt T$, and 
	\[ E := \cup_{i=1}^N (E_{\bx_i} \cap B_{\delta_i/3}(\bx_i)). \]
	
	\begin{claim} \label{clai:locally.homologically.trivial.welldefined}
		For every $j \in \{1, \ldots, N\}$, $E \cap B_{\delta_j/3}(\bx_j) = E_{\bx_j} \cap B_{\delta_j/3}(\bx_j)$.
	\end{claim}
	\begin{proof}[Proof of claim]
		Suppose $i \in \{ 1, \ldots, N \}$, $i \neq j$, is such that $B_{\delta_i/3}(\bx_i) \cap B_{\delta_j/3}(\bx_j) \neq \emptyset$. Then $B_{\delta_j/3}(\bx_j) \subset B_{\delta_i}(\bx_i)$. By Claim \ref{clai:locally.homologically.trivial.uniqueness}, applied with $\bx = \bx_i$ and $B = B_{\delta_j/3}(\bx_j)$, 
		\[ E_{\bx_j} \cap B_{\delta_j/3}(\bx_j) = E_{\bx_i} \cap B_{\delta_j/3}(\bx_j). \]
		It then follows that
		\begin{align*}
			E \cap B_{\delta_j/3}(\bx_j) 
				& = \left( \cup_{i=1}^N E_{\bx_i} \cap B_{\delta_i/3}(\bx_i) \right) \cap B_{\delta_j/3}(\bx_j) \\
				& = \cup_{i=1}^N ( E_{\bx_i} \cap B_{\delta_i/3}(\bx_i) \cap B_{\delta_j/3}(\bx_j)) \\
				& = \cup_{i=1}^N ( E_{\bx_j} \cap B_{\delta_i/3}(\bx_i) \cap B_{\delta_j/3}(\bx_j)) \\
				& = E_{\bx_j} \cap B_{\delta_j/3}(\bx_j) \cap ( \cup_{i=1}^N B_{\delta_i/3}(\bx_i)) \\
				& = E_{\bx_j} \cap B_{\delta_j/3}(\bx_j) \cap \hat N \\
				& = E_{\bx_j} \cap B_{\delta_j/3}(\bx_j),
		\end{align*}
		as claimed.
	\end{proof}
	
	As a result of Claim \ref{clai:locally.homologically.trivial.welldefined}, $E$ is a Caccioppoli subset of $\hat N$ and $T = \partial \llbracket E \rrbracket \mres \hat N$. The result follows by taking an inward smoothing $\breve N$ of $\hat N$ and restricting $E$ accordingly.
\end{proof}

\begin{proof}[Proof of Lemma \ref{lemm:smale.abstract.perturbation.closeness}]
	The only somewhat non-standard conclusion is the connectedness of $T'$ in (c), which is similar to the connectedness proof in Lemma \ref{lemm:plateau.minimizing.family}.
\end{proof}

\begin{proof}[Proof of Lemma \ref{lemm:smale.abstract.product.smallness}]
	We expand $w$ using the degree $j = 1, \ldots, k$ elementary symmetric polynomials $\{ \mathfrak{s}_{j}(t_1, \ldots, t_k) \}_{j=1,\ldots,k}$ of $k$ real variables:
	\[ w(x,s) = 1 + \sum_{j=1}^k (-1)^j \tau^j \chi(x)^j \mathfrak{s}_{j}(\zeta(s-h_1(x)), \ldots, \zeta(s-h_k(x))). \]
	We will estimate $w-1$ using the coarse inequalities
	\begin{align*}
		\Vert f_1 \cdots f_\ell \Vert_{C^{m,\theta}} 
			& \leq \ell^{m+1} \Vert f_1 \Vert_{C^{m,\theta}} \cdots \Vert f_\ell \Vert_{C^{m,\theta}}, \\
		\ell^{m+1} 
			& \leq C_{0,m} \ell!, \\
		\Vert \zeta(s - v_j) \Vert_{C^{m,\theta}} 
			& \leq C_{0,m} \Vert \zeta \Vert_{C^{m,\theta}} (1 + \Vert v_j \Vert_{C^{m,\theta}})^m.
	\end{align*}
	We have:
	\begin{align*}
		\Vert w - 1 \Vert_{C^{m,\theta}} 
			& = \Vert \sum_{j=1}^k (-1)^j \tau^j \chi(x)^j \mathfrak{s}_{j}(\zeta(s-h_1(x)), \ldots, \zeta(s-h_k(x))) \Vert_{C^{m,\theta}} \\
			& \leq \sum_{j=1}^k \tau^j \Vert \chi(x)^j \mathfrak{s}_{j}(\zeta(s-h_1(x)), \ldots, \zeta(s-h_k(x))) \Vert_{C^{m,\theta}} \\
			& \leq C_{0,m} \sum_{j=1}^k \binom{k}{j} (2j)^{m+1} \tau^j \Vert \chi \Vert_{C^{m,\theta}}^j \Vert \zeta \Vert_{C^{m,\theta}}^j (1+\max_{j=1, \ldots, k} \Vert h_j \Vert_{C^{m,\theta}})^{mj} \\
			& \leq 2^{m+1} C_{0,m}^2 \sum_{j=1}^k \frac{k!}{(k-j)!} \tau^j \Vert \chi \Vert_{C^{m,\theta}}^j \Vert \zeta \Vert_{C^{m,\theta}}^j (1+\max_{j=1, \ldots, k} \Vert h_j \Vert_{C^{m,\theta}})^{mj},
	\end{align*}
	as desired, with $2^{m+1} C_{0,m}^2$ in place of $C_{0,m}$.
\end{proof}

Finally we return to the proof of Lemmas \ref{lemm:main.homology.uniqueness.perturbation} and \ref{lemm:main.homology.main.lemma}.

\begin{proof}[Proof of Lemma \ref{lemm:main.homology.uniqueness.perturbation}]
	Define $g' := (1+\eps u) g$ where $\eps > 0$ and $u : N \to \RR$ is a smooth positive function with $\spt T = \{ u = 0 \}$. This is doable via a partition of unity because $\spt T$ is closed. If $\spt T$ is smooth, then we can additionally require that $\nabla^2 u(\nu, \nu) > 0$ along $\spt T$, where $\nu$ is the unit normal, so that $\spt T$ is strictly stable with respect to $g'$. 
	
	It follows from the minimizing nature of $T$, and from $g' > g$ on $N \setminus \spt T$, and a direct area comparison argument that $T$ is minimizing in $[\alpha]$ with respect to $g'$. Conversely, let $\tilde T \in \II_n(N)$ be any minimizer in $[\alpha]$ with respect to $g'$. It follows (again) from the minimizing nature of $T$ and from $g' > g$ on $N \setminus \spt T$ that $\spt \tilde T \subset \spt T$. It also follows from the connectedness of $\spt T$ and unique continuation that $\spt \tilde T = \spt T$, so $\tilde T = T$, so $T$ is uniquely minimizing with respect to $g'$. This completes the proof by sending $\eps \to 0$ to ensure $g' \in \cG$. 
\end{proof}

\begin{proof}[Proof of Lemma \ref{lemm:main.homology.main.lemma}]
	Note that we may assume $\sing T \neq \emptyset$, otherwise there is nothing to do. This lemma is modeled after \cite[Lemma 1.4]{Smale:generic}. The strategy in \cite{Smale:generic} is to approximate $g$ by $g'$ in such a way that minimizing $T' \in \II_n(N)$ with respect to $g'$ representing $[\alpha]$ get pushed to ``one side'' of $T$. If $\sing T' \neq \emptyset$, by Allard's regularity theorem and a blow-up argument, one obtains a singular minimizing boundary in $\RR^8$ lying on one side of a nonflat minimizing tangent cone $\cC$ of $T$. This contradicts the classification aspect of Theorem \ref{thm:HS-fol}, according to which the minimizing boundary is unique up to dilations and smooth. Due to the extra complexity of minimizing cones in $\RR^9$, $\RR^{10}$ versus $\RR^8$ (the dimension of their spines can vary, and Theorem \ref{thm:HS-fol} no longer applies), we need to deviate from \cite{Smale:generic} and rely on our new results in Section \ref{sec:minimizing.foliations}. 
	
	\textit{Iterative improvement}. Apply Proposition \ref{prop:deformation.manifold} with $\cG$, our metric $g$, and our unique minimizer $T$ with multiplicity one. The proposition yields 
	\[ g =: g_{K,0}, g_{K,1}, \ldots, g_{K,K} \in \cG \]
	\[ T =: T_{K,0}, T_{K,1}, \ldots, T_{K,K} \in \II_n^1(N), \]
	the latter of which enjoy a quantitative separation estimate. Borrowing notation from the proposition, note that $g_{K,k} \equiv g$ on $N \setminus \bar C_\rho$, and that for small $\tau$ the restrictions $T_{K,k} \mres (N \setminus \bar C_{3\rho/2})$ satisfy all hypotheses of Theorem \ref{theo:high.density.packing} with $N \setminus \bar C_\rho$ in place of $\breve N$ and a slightly inward smoothing of $N \setminus \bar C_{3\rho/2}$ in place of $\Omega$. If $K$ is large enough, then Theorem \ref{theo:high.density.packing} guarantees the existence of $k^* \in \{ 1, \ldots, K \}$ such that
	\[ \cD(T_{K,k^*}) \leq \cD(T) - \eta, \]
	where $\cD(\cdot)$ is as in \eqref{eq:calD} and $\eta = \eta(n) > 0$. Using Lemma \ref{lemm:main.homology.uniqueness.perturbation}, perturb $g_{K,k^*}$ to $g' \in \cG$ so that $T' = T_{K,k^*}$ becomes \textit{uniquely} minimizing in $[\alpha]$ with respect to $g'$. Note that
	\[ \cD(T') = \cD(T_{K,k^*}) \leq \cD(T) - \eta. \]
	If $\cD(T') < \Theta_n^*$, we are done: by Allard's regularity theorem, there are no more singular points. Otherwise, we restart the iterative argument with $g'$ and $T'$ in place of $g$ and $T$. We will only iterate finitely many times before eventually reaching $\cD(T') < \Theta_n^*$, since $\eta$ is fixed.
\end{proof}

\appendix

\section{Hardt--Simon decomposition of minimizing currents} \label{app:hardt.simon.boundary}

Recall the Hardt--Simon regularity theorem (see \cite{Steinbruchel} for manifold adaptation):

\begin{theorem} \label{theo:hardt-simon-boundary-regularity}
	Let $\Gamma$ be a smooth, closed, oriented, $(n-1)$-dimensional submanifold of $\RR^{n+1}$, and $T \in \II_n(\RR^{n+1})$ be minimizing with $\partial T = \llbracket \Gamma \rrbracket$. Then,
	\[ \spt T = M \cup E \]
	where $M$ is a smooth, oriented, minimizing hypersurface containing $\Gamma$, with boundary $\partial M$ consisting of at least one component of $\Gamma$, and with $E = \bar M \setminus M \subset \RR^{n+1}$ compact with Hausdorff dimension $\leq n-7$. Note that $M$ is connected if $\spt T$ is.
	
	Moreover, for all $\bx \in \Gamma$ there exists a neighborhood $U$ of $\bx$ such that $U \cap E = \emptyset$ and:
	\begin{enumerate}
		\item[(i)] either $\bx \in \partial M$ (this is automatic if $\bx$ is on a component of $\Gamma$ intersecting the convex hull of $\Gamma$), in which case $M \cap U$ is a manifold-with-boundary, $\partial (M \cap U) = \Gamma \cap U$, and $T \mres U = \llbracket M \rrbracket \mres U$; 
		\item[(ii)] or $\bx \not \in \partial M$, in which case $M \cap U$ is a manifold (without boundary) that contains $\Gamma \cap U$, and if $(M \setminus \Gamma) \cap U = M' \cup M''$ and $M', M''$ inherit their orientation from $M$, then $T \mres U = k \llbracket M' \rrbracket + (k+1) \llbracket M'' \rrbracket$ for some natural number  $k \geq 1$.
	\end{enumerate}
\end{theorem}
\begin{proof}
	The first part of the statement follows from \cite[Corollary 11.2]{Hardt-Simon:boundary-regularity}. The second part, starting with ``Moreover,'' also follows from \cite{Hardt-Simon:boundary-regularity}, but the statement is taken from \cite[Corollary 2]{White:boundary-regularity-multiplicity} if applied with multiplicity $1$.
\end{proof}

Now if $\Gamma$, $T$ are as in the setting of Theorem \ref{theo:hardt-simon-boundary-regularity}, and we write
\[ \spt T = M \cup E \]
then we can uniquely partition the points $\bx \in M$ into three sets depending on their density with respect to $T$, which is necessarily an element of $\tfrac12 \NN^*$:
\begin{enumerate}
	\item[(i)] if $\Theta_T(\bx) \in \NN^*$, then $\bx \in M \setminus \Gamma$ and $T \mres U = \Theta_T(\bx) \llbracket M \rrbracket \mres U$ for a neighborhood $U$ of $\bx$; or
	\item[(ii)] if $\Theta_T(\bx) = \tfrac12$, then $\bx \in \Gamma$ and $T \mres U = \llbracket M \rrbracket \mres U$ for a neighborhood $U$ of $\bx$; or
	\item[(iii)] if $\Theta_T(\bx) = \NN^* + \tfrac12$, then $(M \setminus \Gamma) \cap U$ has two components $M', M''$ for a neighborhood $U$ of $\bx$, and $T \mres U = (\Theta_T(\bx)-\tfrac12) \llbracket M' \rrbracket + (\Theta_T(\bx)+\tfrac12) \llbracket M'' \rrbracket$ if $M', M''$ inherit their orientation from $M$.
\end{enumerate}
Note that $T$ and $M$ have the same orientation, so in particular
\[ \Vert T \Vert = \Vert \llbracket M \rrbracket \Vert + \Vert T - \llbracket M \rrbracket \Vert. \]
By \cite[Lemma 33.4]{Simon:GMT}, $\llbracket M \rrbracket$ and $T - \llbracket M \rrbracket$ are both minimizing and a similar decomposition applies to both. For $\llbracket M \rrbracket$, which has multiplicity one, this decomposition only contains points in (i) with density $1$ and points in (ii). For $T - \llbracket M \rrbracket$, the densities all drop by $1$ (and points with density $\tfrac12$ and $1$ get discarded from $M$), and a straightforward induction gives:

\begin{corollary} \label{coro:hardt-simon-decomposition}
	Let $\Gamma$ be a smooth, closed, oriented, $(n-1)$-dimensional submanifold of $\RR^{n+1}$, and $T \in \II_n(\RR^{n+1})$ be minimizing with $\partial T = \llbracket \Gamma \rrbracket$. For all $i \in \NN^*$, set
	\begin{equation} \label{eq:hardt-simon-decomposition.mi}
		M^{(i)} := \{ \bx \in M : \Theta_T(\bx) \geq i - \tfrac12 \},
	\end{equation}
	\begin{equation} \label{eq:hardt-simon-decomposition.gammaj}
		\Gamma^{(i)} := \{ \bx \in \Gamma : \Theta_T(\bx) = i - 1/2 \}.
	\end{equation}	
	Then, for every $i = 1, \ldots, m$, $M^{(i)}$ is a smooth, oriented hypersurface with boundary $\partial M^{(i)} = \Gamma^{(i)}$, with $\dim (\bar M^{(i)} \setminus M^{(i)}) \leq n-7$, and whose associated current $\llbracket M^{(i)} \rrbracket$ is minimizing. Moreover, 
	\begin{equation} \label{eq:hardt-simon-decomposition.nested}
		\bar M^{(j)} \subseteq \bar M^{(i)} \setminus \Gamma^{(i)}
	\end{equation}
	for all $i < j$, and
	\begin{equation} \label{eq:hardt-simon-decomposition.sum}
		\Vert T \Vert = \sum_{i=1}^m \Vert \llbracket M^{(i)} \rrbracket \Vert.
	\end{equation}
	Here, $m$ denotes the maximum multiplicity of $T$ on $\reg T$.
\end{corollary}

\begin{lemma} \label{lemm:hardt-simon-uniqueness}
	Let $\Gamma$ and $T$ be as in Theorem \ref{theo:hardt-simon-boundary-regularity}, and $m$, $\Gamma^{(i)}$, $M^{(i)}$ be as in Corollary \ref{coro:hardt-simon-decomposition}. Let $\sigma > 0$ be small enough that $U_\sigma(\Gamma^{(1)}) \approx \Gamma^{(1)} \times \mathbb{D}^2$, and $\bar U_\sigma(\Gamma^{(1)}) \cap M^{(2)} = \emptyset$, and $\partial U_\sigma (\Gamma^{(1)})$ intersects $M^{(1)}$ transversely. Set 
	\[ M'^{(1)} := M^{(1)} \setminus U_\sigma(\Gamma^{(1)}), \; \Gamma'^{(1)} := M^{(1)} \cap \partial U_\sigma (\Gamma^{(1)}), \]
	\[ M'^{(i)} := M^{(i)}, \; \Gamma'^{(i)} := \Gamma^{(i)}, \; i = 2, \ldots, m. \]
	Then, $T' := \sum_{i=1}^m \llbracket M'^{(i)} \rrbracket$ is uniquely minimizing with $\partial T' = \sum_{i=1}^m \llbracket \Gamma'^{(i)} \rrbracket$.
\end{lemma}
\begin{proof}
	Note, first, that $T'$ is indeed minimizing (\cite[Lemma 33.4]{Simon:GMT}). Now let $\tilde T' \in \II_n(\RR^{n+1})$ be any minimizer with $\partial \tilde T' = \partial T'$. By \cite[Corollary 3]{White:boundary-regularity-multiplicity}, $T' + \tilde T'$ is minimizing. It follows from the Hardt--Simon boundary regularity theorem (Theorem \ref{theo:hardt-simon-boundary-regularity}) that $\tilde T'$ is smooth with multiplicity one near an extremal component $\Gamma'_e \subset \cup_i \Gamma'^{(i)}$ of the boundary. Then, \cite[Corollary 2]{White:boundary-regularity-multiplicity} ensures that the only possible relative configuration of $T'$ and $\tilde T'$ near $\Gamma'_e$ is that they \textit{coincide} or they \textit{meet transversely}. Note that $\MM(T - T' + \tilde T') \leq \MM(T)$, so $T - T' + \tilde T'$ is minimizing for the original $\Gamma$. However, if $T'$ and $\tilde T'$ meet transversely along $\Gamma'_e$, then $T - T' + \tilde T'$ has a codimension-$1$ singular set along $\Gamma'_e$. It is not hard to see that $\Gamma'_e \not \subset \cup_{j=2}^m\Gamma^{(2)}$, so in particular $T - T' + \tilde T'$ has an interior codimension-$1$ singular set along $\Gamma’_e$, which contradicts that the interior singular set of $T - T' + \tilde T’$ has Hausdorff dimension $\leq n-7$. Thus, $T'$ and $\tilde T'$ coincide near $\Gamma'_e$. By unique continuation $T' \mres C = \tilde T' \mres C$ along the component $C \subset \spt T'$ containing $\Gamma'_e$. We then repeat the argument with $T - (T' \mres C)$ in place of $T$ and $T' - (T' \mres C)$ in place of $T'$ (both minimizing by \cite[Lemma 33.4]{Simon:GMT}) until we have no more components left. This process terminates after finitely many steps because the number of components of $\spt \partial T$ decreases at each iteration.
\end{proof}

\section{Positive Jacobi fields on nonflat minimizing cones}  \label{app:positive.jacobi.fields.cones}

We collect some useful results from \cite{Simon:asymptotic-decay} and \cite[Appendix A]{Wang:smoothing} on minimizing cones $\cC^n \subset \RR^{n+1}$ with possibly singular links $\Sigma^{n-1}(\cC) = \cC \cap \partial B_1(\bOh)$. 

For $\cC$, $\Sigma = \Sigma(\cC)$ fixed, and compact smooth domains $\Omega \subset \reg\Sigma$, define
\[
\lambda_1(\Omega) : = \inf\left\{ \int_D |\nabla_\Sigma\varphi|^2 - |A_\Sigma|^2\varphi^2 : \varphi \in C^1_c(\Omega), \Vert \varphi\Vert_{L^2(\Omega)} = 1\right\}\, ,
\]
where $A_\Sigma$ is the second fundamental form of $\reg \Sigma$ in $\partial B_1(\bOh)$. We denote with $\varphi_\Omega$ the unique positive first eigenfunction with $\Vert \varphi_\Omega \Vert_{L^2(\Omega)}=1$. Recall that this means that $\varphi_\Omega$ solves
\begin{equation*}
 \begin{cases}
  -(\Delta_\Sigma + |A_\Sigma|) \varphi_\Omega = \lambda_1(\Omega) \varphi_\Omega \quad & \text{on}\ \Omega\, ,\\
\qquad\qquad \qquad \,  \varphi_\Omega = 0 &\text{on}\ \partial \Omega \, .
\end{cases}
\end{equation*}
Finally, we denote for $\delta > 0$
\begin{equation} \label{eq:r.link}
	\cR_{\geq \delta}(\Sigma) : = \cR_{\geq \delta}(\cC) \cap \partial B_1(\bOh).
\end{equation}

The following is a consequence of the eigenvalue estimate of \cite{Simons:minvar} and \cite{Zhu:eigenvalue}.

\begin{lemma}[\text{\cite[Lemma A.1]{Wang:smoothing}}]\label{lem:Wang-1}
There is $\rho_0=\rho_0(n)>0$ so that
\[ \lambda_1(\cR_{\geq 2\rho_0}(\Sigma)) \leq -(n-1) \]
for all nonflat minimizing cones $\cC^n \subset \RR^{n+1}$ and $\Sigma = \Sigma(\cC)$.
\end{lemma}

\begin{lemma}[\text{cf. \cite[Lemma A.2]{Wang:smoothing}}] \label{lem:Wang-2}
Suppose $\rho_0$ is as in Lemma \ref{lem:Wang-1}, $\cC^n \subset \RR^{n+1}$ is a nonflat minimizing cone, $\Sigma = \Sigma(\cC)$, and $\Omega$ is a smooth compact domain with $\cR_{\geq \rho_0}(\Sigma) \subset \Omega \subset \reg \Sigma$ and with principal Dirichlet eigenfunction $\varphi_\Omega > 0$. If $u$ is any positive Jacobi field on $\reg \cC$ (as in Definition \ref{defi:jacobi.field}), then the function
\[
V(r) : = \int_\Omega \varphi_\Omega(\omega) u(r\omega) d\mu_\Sigma(\omega)
\]
satisfies
\[ (V(r) r^{\kappa_n^*})' \leq 0 \]
where 
\begin{equation*}
\kappa_n^* : = \frac{n-2}{2} - \sqrt{\frac{(n-2)^2}{4} - (n-1)} . 
\end{equation*}
In particular, $V(r) \leq  r^{-\kappa_n^*}V(1)$ for $r\geq 1$. 
\end{lemma}
\begin{proof}
	This is a direct consequence of \cite[(A.3)-(A.4)]{Wang:smoothing}.
\end{proof}

\bibliographystyle{alpha}
\bibliography{generic-min-surf-9-10}

\begin{thebibliography}{CCMS24b}

\bibitem[All72]{Allard:first-variation}
William~K. Allard.
\newblock On the first variation of a varifold.
\newblock {\em Ann. of Math. (2)}, 95:417--491, 1972.

\bibitem[Alm66]{Almgren:regularity}
F.~J. Almgren, Jr.
\newblock Some interior regularity theorems for minimal surfaces and an
  extension of {B}ernstein's theorem.
\newblock {\em Ann. of Math. (2)}, 84:277--292, 1966.

\bibitem[BDGG69]{BDG:Simons}
E.~Bombieri, E.~De~Giorgi, and E.~Giusti.
\newblock Minimal cones and the {B}ernstein problem.
\newblock {\em Invent. Math.}, 7:243--268, 1969.

\bibitem[BG72]{BG}
E.~Bombieri and E.~Giusti.
\newblock Harnack's inequality for elliptic differential equations on minimal
  surfaces.
\newblock {\em Invent. Math.}, 15:24--46, 1972.

\bibitem[Bro86]{Brothers:open-problems}
Some open problems in geometric measure theory and its applications suggested
  by participants of the 1984 {AMS} summer institute.
\newblock In J.~E. Brothers, editor, {\em Geometric measure theory and the
  calculus of variations ({A}rcata, {C}alif., 1984)}, volume~44 of {\em Proc.
  Sympos. Pure Math.}, pages 441--464. Amer. Math. Soc., Providence, RI, 1986.

\bibitem[CCMS24a]{CCMS:generic1}
Otis Chodosh, Kyeongsu Choi, Christos Mantoulidis, and Felix Schulze.
\newblock Mean curvature flow with generic initial data {I}.
\newblock {\em Invent. Math.}, 237(1):121--220, 2024.

\bibitem[CCMS24b]{CCMS:low-ent-gen}
Otis Chodosh, Kyeongsu Choi, Christos Mantoulidis, and Felix Schulze.
\newblock Mean curvature flow with generic low-entropy initial data.
\newblock {\em Duke Math. J.}, 173(7):1269--1290, 2024.

\bibitem[CLS22]{CLS}
Otis Chodosh, Yevgeny Liokumovich, and Luca Spolaor.
\newblock Singular behavior and generic regularity of min-max minimal
  hypersurfaces.
\newblock {\em Ars Inven. Anal.}, pages Paper No. 2, 27, 2022.

\bibitem[CM12]{ColdingMinicozzi:generic}
Tobias~H. Colding and William~P. Minicozzi, II.
\newblock Generic mean curvature flow {I}: generic singularities.
\newblock {\em Ann. of Math. (2)}, 175(2):755--833, 2012.

\bibitem[CMMS24]{CaldiniMarcheseMerloSteinbruchel}
Gianmarco Caldini, Andrea Marchese, Andrea Merlo, and Simone Steinbr\"{u}chel.
\newblock Generic uniqueness for the {P}lateau problem.
\newblock {\em J. Math. Pures Appl. (9)}, 181:1--21, 2024.

\bibitem[CSV18]{CSV:obst}
Maria Colombo, Luca Spolaor, and Bozhidar Velichkov.
\newblock A logarithmic epiperimetric inequality for the obstacle problem.
\newblock {\em Geom. Funct. Anal.}, 28(4):1029--1061, 2018.

\bibitem[DG61]{DeGiorgi:minimizing-regularity}
Ennio De~Giorgi.
\newblock {\em Frontiere orientate di misura minima}.
\newblock Editrice Tecnico Scientifica, Pisa, 1961.
\newblock Seminario di Matematica della Scuola Normale Superiore di Pisa,
  1960-61.

\bibitem[DG65]{DeGiorgi:bernstein}
Ennio De~Giorgi.
\newblock Una estensione del teorema di {B}ernstein.
\newblock {\em Ann. Scuola Norm. Sup. Pisa Cl. Sci. (3)}, 19:79--85, 1965.

\bibitem[Ede24]{Edelen:degen}
Nick Edelen.
\newblock Degeneration of 7-dimensional minimal hypersurfaces which are stable
  or have a bounded index.
\newblock {\em Arch. Ration. Mech. Anal.}, 248(4):Paper No. 65, 73, 2024.

\bibitem[ES91]{Evans-Spruck:level-set-flow-i}
L.~C. Evans and J.~Spruck.
\newblock Motion of level sets by mean curvature. {I}.
\newblock {\em J. Differential Geom.}, 33(3):635--681, 1991.

\bibitem[ES92a]{Evans-Spruck:level-set-flow-ii}
L.~C. Evans and J.~Spruck.
\newblock Motion of level sets by mean curvature. {II}.
\newblock {\em Trans. Amer. Math. Soc.}, 330(1):321--332, 1992.

\bibitem[ES92b]{Evans-Spruck:level-set-flow-iii}
L.~C. Evans and J.~Spruck.
\newblock Motion of level sets by mean curvature. {III}.
\newblock {\em J. Geom. Anal.}, 2(2):121--150, 1992.

\bibitem[ES95]{Evans-Spruck:level-set-flow-iv}
Lawrence~C. Evans and Joel Spruck.
\newblock Motion of level sets by mean curvature. {IV}.
\newblock {\em J. Geom. Anal.}, 5(1):77--114, 1995.

\bibitem[ES24]{EdelenSzekelyhidi:liouville}
Nick Edelen and G\'{a}bor Sz\'{e}kelyhidi.
\newblock A {L}iouville-type theorem for cylindrical cones.
\newblock {\em Comm. Pure Appl. Math.}, 77(8):3557--3580, 2024.

\bibitem[Fed69]{Federer:GMT}
Herbert Federer.
\newblock {\em Geometric measure theory}.
\newblock Die Grundlehren der mathematischen Wissenschaften, Band 153.
  Springer-Verlag New York, Inc., New York, 1969.

\bibitem[Fed70]{Federer:dimension-reduction}
Herbert Federer.
\newblock The singular sets of area minimizing rectifiable currents with
  codimension one and of area minimizing flat chains modulo two with arbitrary
  codimension.
\newblock {\em Bull. Amer. Math. Soc.}, 76:767--771, 1970.

\bibitem[Fed75]{Federer:real-chains}
Herbert Federer.
\newblock Real flat chains, cochains and variational problems.
\newblock {\em Indiana Univ. Math. J.}, 24:351--407, 1974/75.

\bibitem[FF60]{FedererFleming}
Herbert Federer and Wendell~H. Fleming.
\newblock Normal and integral currents.
\newblock {\em Ann. of Math. (2)}, 72:458--520, 1960.

\bibitem[Fle62]{Fleming:plateau}
Wendell~H. Fleming.
\newblock On the oriented {P}lateau problem.
\newblock {\em Rend. Circ. Mat. Palermo (2)}, 11:69--90, 1962.

\bibitem[FROS20]{FROS:generic}
Alessio Figalli, Xavier Ros-Oton, and Joaquim Serra.
\newblock Generic regularity of free boundaries for the obstacle problem.
\newblock {\em Publ. Math. Inst. Hautes \'{E}tudes Sci.}, 132:181--292, 2020.

\bibitem[FRRO21]{FernandezRealRosOton}
Xavier Fern\'andez-Real and Xavier Ros-Oton.
\newblock Free boundary regularity for almost every solution to the {S}ignorini
  problem.
\newblock {\em Arch. Ration. Mech. Anal.}, 240(1):419--466, 2021.

\bibitem[FRTL23]{FernandezRealTorresLatorre}
Xavier Fern\'andez-Real and Clara Torres-Latorre.
\newblock Generic regularity of free boundaries for the thin obstacle problem.
\newblock {\em Adv. Math.}, 433:Paper No. 109323, 29, 2023.

\bibitem[FS19]{FS:obst-fine}
Alessio Figalli and Joaquim Serra.
\newblock On the fine structure of the free boundary for the classical obstacle
  problem.
\newblock {\em Invent. Math.}, 215(1):311--366, 2019.

\bibitem[Giu84]{Giusti}
Enrico Giusti.
\newblock {\em Minimal surfaces and functions of bounded variation}, volume~80
  of {\em Monographs in Mathematics}.
\newblock Birkh\"{a}user Verlag, Basel, 1984.

\bibitem[Gro17]{Gromov:101}
Misha Gromov.
\newblock 101 questions, problems and conjectures around scalar curvature.
  (incomplete and unedited version).
\newblock
  \url{https://www.ihes.fr/~gromov/wp-content/uploads/2018/08/101-problemsOct1-2017.pdf},
  2017.
\newblock Accessed: February 22, 2023.

\bibitem[Gro19]{gromov2019lectures}
Misha Gromov.
\newblock Four lectures on scalar curvature.
\newblock {\em \url{https://arxiv.org/abs/1908.10612}}, 2019.

\bibitem[HS79]{Hardt-Simon:boundary-regularity}
Robert Hardt and Leon Simon.
\newblock Boundary regularity and embedded solutions for the oriented {P}lateau
  problem.
\newblock {\em Ann. of Math. (2)}, 110(3):439--486, 1979.

\bibitem[HS85]{Hardt-Simon:isolated-singularities}
Robert Hardt and Leon Simon.
\newblock Area minimizing hypersurfaces with isolated singularities.
\newblock {\em J. Reine Angew. Math.}, 362:102--129, 1985.

\bibitem[Ilm94]{Ilmanen:elliptic-regularization}
Tom Ilmanen.
\newblock Elliptic regularization and partial regularity for motion by mean
  curvature.
\newblock {\em Mem. Amer. Math. Soc.}, 108(520):x+90, 1994.

\bibitem[Ilm96]{Ilmanen:smp}
T.~Ilmanen.
\newblock A strong maximum principle for singular minimal hypersurfaces.
\newblock {\em Calc. Var. Partial Differential Equations}, 4(5):443--467, 1996.

\bibitem[Loh18]{Lohkamp:HS}
Joachim Lohkamp.
\newblock Minimal smoothings of area minimizing cones.
\newblock {\em \url{https://arxiv.org/abs/1810.03157}}, 2018.

\bibitem[Loh23]{Lohkamp:PSC}
Joachim Lohkamp.
\newblock The secret hyperbolic life of positive scalar curvature.
\newblock In {\em Perspectives in scalar curvature. {V}ol. 1}, pages 611--642.
  World Sci. Publ., Hackensack, NJ, [2023] \copyright 2023.

\bibitem[LW25]{LiWang:generic8dim}
Yangyang Li and Zhihan Wang.
\newblock Minimal hypersurfaces for generic metrics in dimension 8.
\newblock {\em Invent. Math.}, 240(3):1193--1303, 2025.

\bibitem[Mag12]{Maggi:finite-per}
Francesco Maggi.
\newblock {\em Sets of finite perimeter and geometric variational problems},
  volume 135 of {\em Cambridge Studies in Advanced Mathematics}.
\newblock Cambridge University Press, Cambridge, 2012.
\newblock An introduction to geometric measure theory.

\bibitem[MM84]{MassariMiranda}
Umberto Massari and Mario Miranda.
\newblock {\em Minimal surfaces of codimension one}, volume~91 of {\em
  North-Holland Mathematics Studies}.
\newblock North-Holland Publishing Co., Amsterdam, 1984.
\newblock Notas de Matem\'{a}tica [Mathematical Notes], 95.

\bibitem[Mon03]{Monneau}
R.~Monneau.
\newblock On the number of singularities for the obstacle problem in two
  dimensions.
\newblock {\em J. Geom. Anal.}, 13(2):359--389, 2003.

\bibitem[Mor86]{Morgan:regularity-modulo-n}
Frank Morgan.
\newblock A regularity theorem for minimizing hypersurfaces modulo {$\nu$}.
\newblock {\em Trans. Amer. Math. Soc.}, 297(1):243--253, 1986.

\bibitem[MS94]{MazzeoSmale}
Rafe Mazzeo and Nathan Smale.
\newblock Perturbing away higher-dimensional singularities from area minimizing
  hypersurfaces.
\newblock {\em Comm. Anal. Geom.}, 2(2):313--336, 1994.

\bibitem[Sim68]{Simons:minvar}
James Simons.
\newblock Minimal varieties in riemannian manifolds.
\newblock {\em Ann. of Math. (2)}, 88:62--105, 1968.

\bibitem[Sim83]{Simon:GMT}
Leon Simon.
\newblock {\em Lectures on geometric measure theory}, volume~3 of {\em
  Proceedings of the Centre for Mathematical Analysis, Australian National
  University}.
\newblock Australian National University, Centre for Mathematical Analysis,
  Canberra, 1983.

\bibitem[Sim87]{Simon:smp}
Leon Simon.
\newblock A strict maximum principle for area minimizing hypersurfaces.
\newblock {\em J. Differential Geom.}, 26(2):327--335, 1987.

\bibitem[Sim08]{Simon:asymptotic-decay}
Leon Simon.
\newblock A general asymptotic decay lemma for elliptic problems.
\newblock In {\em Handbook of geometric analysis. {N}o. 1}, volume~7 of {\em
  Adv. Lect. Math. (ALM)}, pages 381--411. Int. Press, Somerville, MA, 2008.

\bibitem[Sim21]{Simon:liouville}
Leon Simon.
\newblock A {L}iouville-type theorem for stable minimal hypersurfaces.
\newblock {\em Ars Inven. Anal.}, pages Paper No. 5, 35, 2021.

\bibitem[Sim23]{Simon:stable.sing}
Leon Simon.
\newblock Stable minimal hypersurfaces in {$\Bbb R^{N+1+\ell}$} with singular
  set an arbitrary closed {$K\subset\{0\}\times\Bbb R^{\ell}$}.
\newblock {\em Ann. of Math. (2)}, 197(3):1205--1234, 2023.

\bibitem[Sma93]{Smale:generic}
Nathan Smale.
\newblock Generic regularity of homologically area minimizing hypersurfaces in
  eight-dimensional manifolds.
\newblock {\em Comm. Anal. Geom.}, 1(2):217--228, 1993.

\bibitem[Ste22]{Steinbruchel}
Simone Steinbr\"{u}chel.
\newblock Boundary regularity of minimal oriented hypersurfaces on a manifold.
\newblock {\em ESAIM Control Optim. Calc. Var.}, 28:Paper No. 52, 57, 2022.

\bibitem[SY79]{SY:descent}
R.~Schoen and S.~T. Yau.
\newblock On the structure of manifolds with positive scalar curvature.
\newblock {\em Manuscripta Math.}, 28(1-3):159--183, 1979.

\bibitem[SY22]{SY:sing}
Richard Schoen and Shing-Tung Yau.
\newblock Positive scalar curvature and minimal hypersurface singularities.
\newblock In {\em Surveys in differential geometry 2019. {D}ifferential
  geometry, {C}alabi-{Y}au theory, and general relativity. {P}art 2}, volume~24
  of {\em Surv. Differ. Geom.}, pages 441--480. Int. Press, Boston, MA, [2022]
  \copyright 2022.

\bibitem[Wan20]{Wang:deformation}
Zhihan Wang.
\newblock Deformations of singular minimal hypersurfaces {I}, isolated
  singularities.
\newblock {\em \url{https://arxiv.org/abs/2011.00548}}, 2020.

\bibitem[Wan24]{Wang:smoothing}
Zhihan Wang.
\newblock Mean convex smoothing of mean convex cones.
\newblock {\em Geom. Funct. Anal.}, 34(1):263--301, 2024.

\bibitem[Whi83]{White:boundary-regularity-multiplicity}
Brian White.
\newblock Regularity of area-minimizing hypersurfaces at boundaries with
  multiplicity.
\newblock In {\em Seminar on minimal submanifolds}, volume 103 of {\em Ann. of
  Math. Stud.}, pages 293--301. Princeton Univ. Press, Princeton, NJ, 1983.

\bibitem[Whi19]{White:generic2dim}
Brian White.
\newblock Generic transversality of minimal submanifolds and generic regularity
  of two-dimensional area-minimizing integral currents.
\newblock {\em \url{https://arxiv.org/abs/1901.05148}}, 2019.

\bibitem[Yau82]{Yau:problems}
Shing~Tung Yau, editor.
\newblock {\em Seminar on {D}ifferential {G}eometry}.
\newblock Annals of Mathematics Studies, No. 102. Princeton University Press,
  Princeton, N.J.; University of Tokyo Press, Tokyo, 1982.
\newblock Papers presented at seminars held during the academic year
  1979--1980.

\bibitem[Zhu18]{Zhu:eigenvalue}
Jonathan~J. Zhu.
\newblock First stability eigenvalue of singular minimal hypersurfaces in
  spheres.
\newblock {\em Calc. Var. Partial Differential Equations}, 57(5):Paper No. 130,
  13, 2018.

\end{thebibliography}
\end{document}